\RequirePackage{tikz}
\documentclass[sn-mathphys]{sn-jnl}
\pdfoutput=1
\newif\iflong
\longtrue 
\usepackage[utf8]{inputenc}
\usepackage{amsmath}
\usepackage{amssymb}
\usepackage{amsthm}
\usepackage{amsfonts}
\usepackage{tikz,tikz-3dplot}
\usepackage{epsfig}
\usepackage{url}
\usepackage[all,knot,poly]{xy}
\tdplotsetmaincoords{70}{115}
\usepackage{todonotes}
\usetikzlibrary{matrix}

\newtheorem{definition}{Definition}[section]
\newtheorem{theorem}[definition]{Theorem} 
\newtheorem{proposition}[definition]{Proposition} 
\newtheorem{lemma}[definition]{Lemma} 
\newtheorem{remark}[definition]{Remark} 
\newtheorem{ground}[definition]{Ground assumption} 
\newtheorem{example}[definition]{Example} 
\newtheorem{corollary}[definition]{Corollary} 

\usepackage{tikz-cd}

\def\Ab{\mbox{\bf Ab}}
\def\C{\mbox{$\cal C$}}

\def\ie{{\emph{i.e.~ }}}
\newcommand\look[1]{{\bf #1}}
\newcommand\I{\overrightarrow{I}}
\newcommand\J{\overrightarrow{J}}
\newcommand\K{\overrightarrow{K}}
\newcommand\dipi[1]{\overrightarrow{\pi}_1(#1)}
\newcommand\diS{\mbox{$\overrightarrow{S}^1$}}
\newcommand\diCub{\mbox{$\overrightarrow{C}$}}

\newcommand\ForAuthors[1]
 {\par\smallskip                     
  \begin{center}
   \fbox
   {\parbox{0.9\linewidth}
    {\raggedright--- #1}
   }
  \end{center}
  \par\smallskip                     %
 }

\newcommand\Int[0]{\mathbb{Z}}

\newtheorem*{definition15}{Definition \ref{def:hommodule}}
\newtheorem*{definition18}{Definition \ref{def:relative}}
\newtheorem*{theorem1}{Theorem \ref{thm:bisimequivnathom}}
\newtheorem*{theorem2}{Theorem \ref{thm:invariance}}
\newtheorem*{theorem3}{Theorem \ref{thm:kunneth}}
\newtheorem*{theorem4}{Theorem \ref{prop:relhomology}}
\newtheorem*{theorem5}{Theorem \ref{thm:MayerVietoris}}

\tikzset{
  curarrow/.style={
  rounded corners=8pt,
  execute at begin to={every node/.style={fill=red}},
    to path={-- ([xshift=-50pt]\tikztostart.center)
    |- (#1) node[fill=white] {$\scriptstyle \partial_*$}
    -| ([xshift=50pt]\tikztotarget.center)
    -- (\tikztotarget)}
    }
}

\definecolor{mred}{rgb}{0.7,0.1,0.1}
\definecolor{mblue}{rgb}{0,0,0.8}
\definecolor{mgreen}{rgb}{0,0.6,0.3}

\newcommand\Fact[1]{{\mathcal{F}}#1}
\newcommand\RFact[1]{{\mathcal{R}}#1}
\newcommand\LFact[1]{{\mathcal{L}}#1}
\newcommand\map[3]{#1: #2 \longrightarrow #3}
\newcommand\twomap[3]{#1: #2 \Longrightarrow #3}

\newcommand\Pairs[1]{\textbf{Pairs}(#1)}

\newcommand\ab[0]{\textbf{Ab}}

\newcommand\gp[0]{\textbf{V}}

\newcommand\natsys[2]{NatSys(#1,#2)}
\newcommand\natsysnu[2]{NatSys_{\nu}(#1,#2)}

\newcommand\syshd[2]{\overrightarrow{h}_{#1}(#2)}

\newcommand\Real[0]{\mathbb{R}}
\newcommand\M[0]{\mathcal{M}}

\newcommand\N[0]{\mathbb{N}}
\newcommand\D[0]{\mathcal{D}}

\newcommand\Trace[0]{\mathcal{T}}

\newcommand\diP[1]{\overrightarrow{P}(#1)}

\def\Z{\mbox{$\mathbb{Z}$}}

\def\TM{\mbox{$\mathbb T$}}
\def\SM{\mbox{$\mathbb M$}}

\newcommand\functor[1][l]{\csname#1functor\endcsname}
\newcommand\lfunctor[3]{%
  \setbox0=\hbox{$#2$}%
  \kern\wd0%
  \ensurestackMath{\Centerstack[c]{#1\\ \mathllap{#2\;\,}\mathclap{\DownArrow}\\#3}}%
}
\newcommand\rfunctor[3]{%
  \setbox0=\hbox{$#2$}%
  \ensurestackMath{\Centerstack[c]{#1\\\mathclap{\DownArrow}\mathrlap{\,\;#2}\\#3}}%
  \kern\wd0%
}

\newcommand\DownArrow{\rotatebox[origin=c]{-90}{$\longrightarrow$\,}}

\title{Directed Homology and Persistence Modules}
\author[1]{Eric Goubault}\email{eric.goubault@polytechnique.edu}
\affil[2]{\orgdiv{LIX}, \orgname{CNRS, Ecole Polytechnique, Institut Polytechnique de Paris}, \orgaddress{\city{Palaiseau}, \postcode{91128 Cedex}, \country{France}}}
\date{}

\begin{document}

\abstract{We give a mostly self-contained account on a construction for a directed homology theory based on modules over algebras, linking it to both persistent homology and natural homology. We study its first properties, among which are relative homology and Mayer-Vietoris long exact sequences and a K\"unneth formula.}


\maketitle

\section{Introduction}

Persistence modules have been introduced originally for topological data analysis, in the seminal work \cite{Carlsson} and further generalized to multiparameter filtrations in \cite{multidimpersistence}. In topological data analysis, the data set is organized as a filtration of spaces, where the data is seen at different ``resolutions'' (at least in the Vietoris-Rips complex approach). In case this is an actual filtration of topological spaces, i.e. a set of spaces totally ordered by inclusion, the evolution of the geometry, in particular of the homology, of the spaces in the filtration is well studied. This gives rise to a simple persistence module over the path algebra of the (quiver of the) linear order of the filtration. 

Generalized persistence modules were introduced in \cite{Bubenik}, where a persistence module was defined to be any functor from a partially ordered set, or more generally a preordered set, to an arbitrary category (in general, a category of vector spaces). This view is a categorification of modules, as is well known. This is also a natural way to encode representations of these preordered sets, which, by Morita equivalence, are the same as persistence modules over their path algebra \cite{assocalg}.  

Directed topology has its roots in a very different application area, concurrency and distributed systems theory \cite{thebook,grandisbook} rather than data analysis. Its interest lies in the study of (multi-dimensional) dynamical systems, discrete (precubical sets, for application to concurrency theory) or continuous time (differential inclusions, for e.g. applications in control), that appear in the study of multi-agent systems. In directed topology, topological spaces under study are ``directed'', meaning they have ``preferred directions'', for instance a cone in the tangent space, if we are considering a manifold, or the canonical local coordinate system in a precubical set. The topological invariants needed for classifying such spaces have been long in the making. Indeed, classical homotopy or homology theories would not distinguish between distinct directed structures on the same underlying topological space. Natural homology has been introduced in \cite{Dubut} to deal with this, by defining a natural system \cite{bauwir} of modules, describing the evolution of the standard (simplicial, or singular) homology of certain path spaces, along their endpoints. Indeed, this is in spirit similar to persistent homology. 

This paper is trying to make a case about the fact that both frameworks have been introducing tools for different reasons, that are slowly overlapping, if not converging. Not surprisingly, the topological data analysis community has also become interested in dynamical systems recently, see e.g. \cite{Munch2013ApplicationsOP,conleyindex}. 

First steps in that direction have been published in \cite{calk2023persistent}, where natural homology has been shown to be, for a certain class of directed spaces, a colimit of one-parameter persistence modules. Another step was made in \cite{goubault2023semiabelian} where the author constructed a directed homology theory based on the semi-abelian category of (unital) algebras. 

In directed topology, a particular effort has been made into axiomatizing directed homotopy and homology theories, in the style of algebraic homotopy \cite{baues,grandisbook} and Eilenberg-Steenrod axioms \cite{Dubut}. This has less been the case in persistence, although some recent work \cite{miller2020modules} is paving the way. In this paper, we are discussing K\"unneth formulas and relative homology sequences for directed homology, and hope for some interest in the persistent homology community. 

One fundamental hurdle in both frameworks, is the cost of evaluating such homologies in practice. In one-parameter homology, the representations are ``tame'', and persistence modules are just modules over a univariate polynomial ring. In multiparameter homology, or more complicated preordered sets filtrations, representations can be ``wild'', and similarly for arbitrary directed spaces. 

For certain classes of directed spaces, mostly partially ordered topological spaces that are the geometric realization of finite precubical sets, component categories \cite{apcs,components} have provided a framework in which to represent, in a finite manner, the essential homology types of the path spaces involved. In topological data analysis, tame persistence modules have been studied in e.g. \cite{miller2020modules}. An obvious goal is indeed to get these different point of views to coincide, or at least to be expressed in a joint framework. This would allow in particular the use of advances in persistence modules' theory to help with practical computations for directed topology, using e.g. local information as in \cite{levi}. 

\paragraph{Organization of the paper}

This article has been written so that it would be mostly self-contai\-ned and so that it would use only elementary and detailed arguments. As this draws from a variety of mathematical concepts, we begin by introducing the core notions at length in the first few sections, complementing them for the interested reader in Appendices. The 
background on directed topology is given in  Section \ref{sec:directedtopology}, the one on associative algebra and modules, Section \ref{sec:assoc}, and on paths algebras and their modules, in Section \ref{sec:pathalgebras}.

We then come into the heart of the matter: Section \ref{sec:homologymodules} introduces (directed) homology modules as certain bimodules on the underlying path algebra of a precubical set. Of particular importance for the theory are the notions of restriction and extension of scalars that we recap in Section \ref{sec:assoc}, and particularize to the bimodules we need, in Section \ref{sec:homologymodules}.

We make links with persistence modules in Section \ref{sec:persistencemod} and with natural homology, in Section \ref{sec:naturalhomology}. For the advanced reader, we show evidence in Appendix \ref{sec:tameness} that component categories should provide tame characterizations of the first (directed) homology bimodule, for a class of precubical set, interesting for applications (as they come from the semantics of a well studied concurrent language, the PV language, with no loops, see \cite{thebook}). 

Finally, we prove a few facts on homology modules of (a certain class of) precubical sets. First, they are ``invariants'' under homeomorphism of their geometric realization. Invariance meaning here ``bisimulation equivalence'', introduced in \cite{Dubut}. We also prove a K\"unneth formula which links the homology modules of a tensor product of precubical sets, introduced in this paper, with the tensor product of the homology modules for each space, at least for the restricted case of ``cubical complexes'' (as of \cite{Dubut}). 
Last but not least, we prove results concerning exact sequences in this homology theory, and in particular the relative homology exact sequence and Mayer-Vietoris. 

\paragraph{Main results}

The first important result is in Section \ref{sec:homologymodules} with the definition of the homology module of a finite precubical set $X$, as standard chain homology but in a category of bimodules over the path algebra $R[X]$ of $X$:

\begin{definition15} 
The \look{homology module} of a finite precubical set $X$ is defined as the homology $(HM_{i+1})_{i\geq 0}[X]$ in the abelian category of $R[X]$-bimodules, of the complex of $R[X]$-bimodules defined in Lemma \ref{lem:precubboundary}, shifted by one, i.e. 
$$
HM_{i+1}[X]=Ker \ \partial_{\mid R_{i+1}[X]}/Im \ \partial_{\mid R_{i+2}[X]}
$$
\end{definition15}

And then, in the same Section, after some having gone through technicalities to find the right notion of a relative pair of precubical sets, we are able to define relative homology modules in a similar manner as done classically: 

\begin{definition18}
Let $(X,Y)$ be a relative pair of precubical sets. We define 
the \look{relative $i$th-homology $R[X]$-bimodule} $HM_{i}[X,Y]$ as the homology of the quotient of $R_i[X]$ with the sub-$R[X]$-bimodule ${ }^X R_i[Y]$ within the category of $R[X]$-bimodules, with boundary operator defined in the quotient as $\partial [x] = [\partial x]$, $[x]$ denoting a class with representative $x$ in $R_i[X]/{ }^X R_i[Y]$.
\end{definition18}

We then show that our definition of homology module captures the information that is present in another directed homology theory, natural homology \cite{Dubut}, at least for cubical complexes, a simple class of precubical sets, up to applying a barycentric subdivision: 

\begin{theorem1}
 Let $X$ be a cubical complex and $Subd(X)$ be its barycentric subdivision. Then $HM_i[Subd(X)]$ is bisimulation equivalent to the natural homology $HN_i[\mid X\mid ]$, for all $i\geq 1$.    
\end{theorem1}

In turn, this implies that homology modules are an invariant under dihomeomorphism, up to some technical details: 

\begin{theorem2}
Suppose $\mid X\mid $ and $\mid Y\mid $ are dihomeomorphic, i.e. isomorphic in the category of directed spaces. Then for all $i\geq 1$, $HM_i[Subd(X)]$ is bisimulation equivalent to $HM_i[Subd(Y)]$.
\end{theorem2}

Finally, we claim, once again for the restricted class of cubical complexes, that our module homology theory enjoys a number of nice features that classical homology theories have, such as a form of K\"unneth theorem: 

\begin{theorem3}
Let $X$ and $Y$ be cubical complexes and $R$ a field. Then 
$$
HM_*[X\otimes Y]=HM_*[X]\boxtimes HM_*[Y]
$$
\noindent with the action of shuffles of paths $p$ of $X$ with $q$ of $Y$ on $HM_*[X]\boxtimes HM_*[Y]$ to be the action of $p$ on the first component, and action of $q$ on the second component.
\end{theorem3}

As well as a relative homology sequence: 

\begin{theorem4}
Suppose $(X,Y)$ is a relative pair of precubical sets. We have the following relative homology sequence: 

\begin{center}
    \begin{tikzcd}[arrow style=math font,cells={nodes={text height=2ex,text depth=0.75ex}}]
    & & 0 \arrow[draw=none]{d}[name=X,shape=coordinate]{} \\
       HM_1[X,Y] \arrow[curarrow=X]{urr}{} 
       & HM_{1}[X] \arrow[l] \arrow[draw=none]{d}[name=Y, shape=coordinate]{} & \arrow[l] \cdots \\
       HM_{i}[X,Y] \arrow[curarrow=Y]{urr}{} & HM_{i}[X] \arrow[l] \arrow[draw=none]{d}[name=Z,shape=coordinate]{} & { }^X HM_i[Y] \arrow[l] \\
       HM_{i+1}[X,Y] \arrow[curarrow=Z]{urr}{} & HM_{i+1}[X] \arrow[l] & \cdots \arrow[l]
   \end{tikzcd}
\end{center}
\end{theorem4}

And a Mayer-Vietoris like theorem: 

\begin{theorem5}
Let $X$ a precubical set, $X_1$ and $X_2$ be two sub precubical sets that form a good cover of $X$, and such that $(X,X_1)$ and $(X,X_2)$ are relative pairs. 
Then we have the following long exact sequence: 
\begin{center}
    \begin{tikzcd}[arrow style=math font,cells={nodes={text height=2ex,text depth=0.75ex}}]
    & & 0 \arrow[draw=none]{d}[name=X,shape=coordinate]{} \\
       HM_{1}[X] \arrow[curarrow=X]{urr}{} 
       &  { }^X HM_1[X_1]\oplus { }^X HM_1[X_2]    \arrow[l] \arrow[draw=none]{d}[name=Y, shape=coordinate]{} & \arrow[l] \cdots \\
       HM_{i}[X] \arrow[curarrow=Y]{urr}{} & { }^X HM_{i}[X_1]\oplus { }^X HM_i[X_2] \arrow[l] \arrow[draw=none]{d}[name=Z,shape=coordinate]{} & { }^X HM_i[X_1\cap X_2]) \arrow[l] \\
       HM_{i+1}[X] \arrow[curarrow=Z]{urr}{} & { }^X HM_{i+1}[X_1]\oplus { }^X HM_{i+1}[X_2] \arrow[l] & \cdots \arrow[l]
   \end{tikzcd}
\end{center}
\end{theorem5}

\section{Precubical sets and directed spaces}

\label{sec:directedtopology}



\paragraph{Precubical sets}

We remind the reader of the structure of precubical set (see e.g. \cite{Massey}) we are going to use in the sequel: 

\begin{definition}
\label{def:precub}
A \look{precubical set} is a graded set $X=(X_i)_{i \in \N}$ with two families
of operators:
$$
d^0_i, d^1_i: \ X_{n} \rightarrow X_{n-1}
$$ 
($i, j=0,\ldots,n-1$) satisfying the relations
$$
\begin{array}{ccc}
d^k_i \circ d^l_j & = & d^l_{j-1} \circ d^k_i
\end{array}
$$ 
($i < j$, $k, l = 0,1$)

\look{A morphism of precubical sets} $f$ from precubical set $X$ to precubical set $Y$ is a graded function $f: \ X_n \rightarrow Y_n$ for all $n\in \N$ such that it commutes with all boundary operators $d^\epsilon_i$, $\epsilon \in \{0,1\}$.
We denote by $Precub$ the category of precubical sets.
\end{definition}

Note that the 1-skeleton $X_{\leq 1}=(X_0,X_1,d^0_0: X_1\rightarrow X_0, d^1_0: X_1 \rightarrow X_0)$ of any precubical set $X$, is a quiver, or directed graph \cite{assocalg}. 

Precubical sets form a presheaf category on a site $\Box$, whose objects are $[n]$, $n\in \N$ and whose morphisms are $\delta^\epsilon_i: \ [n-1] \rightarrow [n]$ for $n >0$, $\epsilon \in \{0,1\}$, $i \in \{1,\ldots,n\}$, such that $\delta^\eta_j \delta^\epsilon_i  = \delta^\epsilon_i \delta^\eta_{j-1}$, for 
$i < j$. 

We write $d^0_\mathcal{I}$ (resp. $d^1_\mathcal{I}$) for the composition $d^0_{i_1} \circ \ldots \circ d^0_{i_k}$ (resp. $d^1_{i_1} \circ \ldots \circ d^0_{i_k}$) if $\mathcal{I}=\{i_1,\ldots,i_k\}$ with $i_1 < \ldots < i_k$. 
We write $d^0$ (resp. $d^1$) for $d^0_{\{0,\ldots,n-1\}}$ (resp. $d^1_{\{0,\ldots,n-1\}}$).

\paragraph{Directed spaces}

Precubical sets are natural combinatorial models for directed spaces, whose definition we recap below. 

Let $I=[0,1]$ denote the unit
segment with the topology inherited from $\Real$. 

\begin{definition}[\cite{grandisbook}]
A \look{directed topological space}, or \look{d-space}, is a pair  $(X,dX)$ consisting of a topological space $X$ equipped with
a subset $dX\subset X^I$ of continuous paths $p:I \rightarrow X$, called directed paths or
d-paths, satisfying three axioms: 
\begin{itemize}
\item every constant map $I\rightarrow X$ is directed
\item $dX$ is closed under composition with continuous non-decreasing maps from $I$ to $I$
\item $dX$ is closed under concatenation.
\end{itemize}
\end{definition}

We shall abbreviate the notation $(X,dX)$, and write $X$ when the context makes it unambiguous. 

Note that for a d-space $X$, the space of d-paths $dX\subseteq X^I$ is a topological space, equipped with the compact-open topology. We denote by $\diP{X}$ the topological space of d-paths of $X$ modulo reparametrization, and call it the \look{trace space}. We write $\diP{X}^w_v$ for the sub-topological space of $\diP{X}$ made of traces from point $v$ to point $w$ of $X$. 

A map $f: X\to Y$ between d-spaces $(X, dX)$ and $(Y, dY)$ is said to be a \look{d-map} or \look{directed map} if it is continuous and for any d-path $p\in dX$ 
the composition $f\circ p:I\to Y$ belongs to $dY$. In other words we require that $f$ preserves d-paths. We write $df~: dX \rightarrow dY$ for the induced map between directed path
spaces, as well as between trace spaces, by a slight abuse of notation. We denote by ${\cal D}$ the category of d-spaces. 

A \look{directed homeomorphism} or \look{dihomeomorphism} is a homeomorphism which is a directed map, and whose inverse is also a directed map. One of the goals of directed topology is to help classify directed spaces up to dihomeomorphisms. 

Classical examples of d-spaces arise as geometric realizations of precubical sets, as found in e.g. the semantics of concurrent and distributed systems \cite{thebook}. We will use some of these classical example to explain the directed homology theory we are designing in this article. 


Let $\I=[0,1]$ with directed paths all increasing paths in the usual ordering inherited from real numbers, and $\I^n$ the $n$th power of $\I$ (for $n > 0$) as a directed space: as ${\cal D}$ is complete (as well as co-complete), this is well defined. $\I^n$ has as directed paths all continuous paths that are increasing on each coordinate. 


\begin{definition}
The \look{geometric realization} $\mid X \mid$ of a precubical set $X$ is the Yoneda extension of the following functor $F: \ \Box \rightarrow {\cal D}$: 
\begin{itemize}
    \item $F([n])=\I^n$
    \item $F(\delta^\epsilon_i)$ is the map from $\I^{n-1}$ to $\I^n$ which sends $(x_1,\ldots,x_{n-1})\in \I^{n-1}$ to $(x_1,\ldots,x_{i},\epsilon,x_{i+1},\ldots,x_{n-1})$. 
\end{itemize}
\end{definition}

For more details, see e.g. 
\cite{Ziemanski2}. 

Finally, as we are going to construct a directed homology theory that we are going to compare with the construction of \cite{Dubut}, we recap the definition of natural homology of a directed space below. 

First, we need to define the trace category \cite{Dubut} of a directed space.   
Let $X$ be a directed space. We define the \look{trace category} $\T{X}$ as follows: 
\begin{itemize}
    \item $\T{X}$ has as objects, all points of $X$ 
    \item and as morphisms from $s$ to $t$, all traces from $s$ to $t$
     \item the composition $q\circ p$ is the concatenation $p*q$ of such traces, which is associative since this is taken modulo reparameterization: 
     $$
     p * q(s)=\left\{\begin{array}{ll}
     p(2s) & \mbox{if $s\leq \frac{1}{2}$} \\
     q(2s-1) & \mbox{if $\frac{1}{2} \leq s \leq 1$}
     \end{array}\right.
     $$
\end{itemize}
And then, we recap the notion of \look{factorization category} $F\C$, or twisted arrow category \cite{maclane} of a category $\C$: it is the category whose objects are morphisms $f$ in $\C$ and whose morphisms from $f$ to $g$ are pairs of morphisms $\langle u,v \rangle$ of $\C$ such that $g=v\circ f \circ u$. Now we can formally define natural homology: 

\begin{definition}
The \look{natural homology} of directed space $X$ is the collection of functors $HN_i(X): F\T{X}\rightarrow Ab$, $i \geq 1$, from the factorization category $F\T{X}$ of the trace category, to the category of abelian groups with: 
\begin{itemize}
\item $HN_i(X)(p)=H_i(\diP{X}^b_a)$, where $p \in \diP{X}$ is a trace from $a$ to $b \in X$, and $H_i$ is the $i$th singular homology functor,
\item $HN_i(X)(\langle u,v\rangle)(p)=[v \circ p \circ u]$ where $u$ is a trace from $a'$ to $a$, $v$ a trace from $b$ to $b'$, and $[x]$ denotes here the class in $H_i(X)(\diP{X}^{b'}_{a'})$ of $x$. 
\end{itemize}
\end{definition}



\paragraph{Cube chains and trace spaces}

Our results will apply to a particular class of precubical sets, that has been introduced in \cite{Ziemianski}: 

\begin{definition}
Let $X$ be a finite precubical set. We say that $X$ is \look{proper} if the map: 
$$
\bigcup\limits_{n\geq 0} X_n : \ c \rightarrow \{d^0(c),d^1(c)\} \in 2^{X_0}$$
\noindent is an injection, i.e., the cubes of $X$ can be distinguished by its start and end vertices. 

The \look{non-looping length covering} of $X$ is the precubical set $\tilde{X}$ with $\tilde{X}_n=X_n\times \Z$ and $d^\epsilon_i(c,k)=(d^\epsilon_i(c),k+\epsilon)$.
\end{definition}

We write ${Cub}$ for the full sub-category of $Precub$, of finite precubical sets with proper non-looping length covering. 


The following definition is essential in the calculations on trace spaces done by Krzysztof Ziemianski \cite{Ziemanski2}. We twist a little the original definition so that 1-cube chains can be empty, and can consist in that case of constant paths on vertices, this convention will prove easier for the rest of the paper. 

\begin{definition}
Let $X$ be a pre-cubical set and let $v$, $w \in X_0$ be two of its vertices. A \look{cube chain} in $X$ from $v$ to $w$ is a sequence of cubes indexed by $v$ and $w$, $c = (c_1,\ldots,c_l)_{v,w}$, where $c_k \in X_{n_k}$ and $n_k > 0$, $l\geq 0$, such that either $l=0$ and $v=w$ or: 
\begin{itemize}
    \item $d^0(c_1) = v$,
\item $d^1(c_l)=w$,
\item $d^1(c_i) = d^0(c_{i+1})$ for $i = 1,\ldots, l - 1$.
\end{itemize}
We will often leave out the index $v,w$ to cube chains when the context makes it clear what they are. 
\end{definition}

The sequence $(n_1,\ldots,n_l)$ will be called the \look{type of a cube chain} $c$, $dim(c) = n_1 + \ldots + n_l - l$ the \look{dimension} of $c$, and $n_1 + \ldots + n_l$, the \look{length} of $c$. By convention, we set $dim(()_{v,v})=0$: the dimension of constant paths on a vertex is zero. These cube chains in $X$ from $v$ to $w$ will be denoted by $Ch(X)^w_v$, and the set of cube chains of dimension equal to $m$ (resp. less than $m$, less or equal to $m$) by $Ch^{=m}(X)^w_v$ (resp. $Ch^{<m}(X)^w_v$, $Ch^{\leq m}(X)^w_v$). Note that a cube chain has dimension 0 if and only if it contains 1-cubes only, or is an empty cube chain. 

\begin{remark}
\label{rem:0cubechain}
Cube chains of dimension 0 are naturally identified with directed paths (including the constant paths on vertices) of the underlying quiver of the pre-cubical set. 
\end{remark}

\begin{remark}
\label{rem:1cubechain}
1-cube chains are necessarily of the form $(c_1,\ldots,c_l)$ with a unique $c_i$ of dimension 2, all the others being of dimension 1. Indeed, each cell of dimension strictly greater than 1 contributes to adding to the dimension of the cube chain. 
\end{remark}

\begin{remark}
\label{rem:dimension}
We note that the dimension of cube chains is additive in the following sense: let $c=(c_1,\ldots,c_l)$ and $d=(d_1,\ldots,d_k)$ be two cube chains. Their concatenation, when defined, i.e. when $d^1(c_l)=d^0(d_1)$, is the cube chain $e=(c_1,\ldots,c_l,d_1,\ldots,d_k)$. The dimension of $e$ is the sum of the dimensions of each cell $c_i$ and $d_j$ minus $k+l$, hence is equal to the sum of the dimension of $c$ with the dimension of $d$. 
\end{remark}

For a cube chain $c = (c_1,\ldots,c_l) \in Ch(X)^w_v$ of type $(n_1,\ldots,n_l)$ and dimension $i=\sum\limits_{j=1}^l n_i-l$, an integer $k \in \{1,\ldots,l\}$ and a subset $\mathcal{I} \subseteq \{0,\ldots,n_k-1\}$ having $r$ elements, where $0 < r < n_k$, define a cube chain (as done in \cite{Ziemianski}):
$$
d_{k,\mathcal{I}}(c) = (c_1,\ldots,c_{k-1},d^0_{\overline{\mathcal{I}}}(c_k),d^1_\mathcal{I}(c_k),c_{k+1},\ldots,c_l) \in Ch(X)^w_v$$ 
\noindent where $\overline{\mathcal{I}} = \{0,\ldots,n_k-1\} \backslash \mathcal{I}$. 
This defines a cell complex with cells of dimension $i$, $Ch^{=i}(X)$, being $i$-cube chains, and with faces given by the $d_{k,\mathcal{I}}$. 

For $R$ a given commutative field we define: 

\begin{definition}
\label{def:modcub}
Let $X$ be a precubical set. We call {$\look{R_{i+1}[X]}$} for $i\geq 0$ the $R$-vector space generated by all cube chains of dimension $i$. 
\end{definition}

The formula below is given in 
\cite{Ziemianski}: 

\begin{definition}
\label{def:boundaryoperator}
Define a \look{boundary map} from the 
$R_{i+1}[X]$ 
to $R_i[X]$ 
as follows: 

\begin{multline*}
\partial c = \sum\limits_{k=1}^l \sum\limits_{r=1}^{n_k-1} \sum\limits_{\mbox{\tiny $\mathcal{I}$} \subseteq\{0,\ldots,n_k-1\}: \ \mid \mbox{\tiny $\mathcal{I}$} \mid=r}
(-1)^{n_1+\ldots+n_{k-1}+k+r+1} sgn(\mbox{$\mathcal{I}$}) {d_{k,\mbox{$\mathcal{I}$}}(c)}
\end{multline*}
\noindent where
$$
sgn(\mathcal{I})=\left\{\begin{array}{ll}
1 & \mbox{if $\sum\limits_{i\in \mathcal{I}} i \equiv \sum\limits_{i=1}^r i \ mod \ 2$} \\
-1 & \mbox{otherwise}
\end{array}\right.
$$
Then $R_*[X]=(R_{i+1}[X],\partial)_{i\in \N}$ is a chain complex. The restriction to cube chains from any $v \in X_0$ to any $w \in X_0$ is a sub-chain complex of $R_*[X]$ that we write $R^w_{{v},*}[X]$.
\end{definition}

The following is a direct consequence of Theorem 1.7 of \cite{Ziemianski}: 

\begin{lemma}
\label{lem:Kris}
Let $X \in Cub$ be a finite precubical set that has proper non-looping length covering. 
Then for all $v, w \in X_0$, the chain homology of $R_{v,*}^w[X]$ is isomorphic to the homology of the trace space $\diP{\mid X \mid}^w_v$.
\end{lemma}

In order to illustrate these definitions and constructs, we give below some sample calculations, which will also be useful later. 

\begin{example}[``2 holes on diagonal/antidiagonal'']
\label{ex:2holes}
We consider the following two precubical sets: 
\begin{center}
    \begin{minipage}{5.5cm}
    \[\begin{tikzcd}
  9 \arrow[r,"i"] \arrow[d,"k"] \arrow[dr,phantom,"C"] & 8 \arrow[d,"h"] \arrow[r,"j"] & 7 \arrow[d,"l"]\\
  6 \arrow[r,"d"] \arrow[d,"c"] & 5 \arrow[d,"e"] \arrow[dr,phantom,"D"] \arrow[r,"f"] & 4 \arrow[d,"g"] \\
  3 \arrow[r,"a"] & 2 \arrow[r,"b"] & 1
\end{tikzcd}
\]
\end{minipage}
    \begin{minipage}{5.5cm}
    \[\begin{tikzcd}
  9 \arrow[r,"i"] \arrow[d,"k"]  & 8 \arrow[d,"h"] \arrow[r,"j"] \arrow[dr,phantom,"E"] & 7 \arrow[d,"l"]\\
  6 \arrow[r,"d"] \arrow[d,"c"] \arrow[dr,phantom,"F"] & 5 \arrow[d,"e"]  \arrow[r,"f"] & 4 \arrow[d,"g"] \\
  3 \arrow[r,"a"] & 2 \arrow[r,"b"] & 1
\end{tikzcd}
\]
\end{minipage}
\end{center}

We call the one on the left the ``2 holes on the anti-diagonal'', and the one on the right, ``2 holes on the diagonal''. 

Cube chains of dimension 0 are just directed paths of the underlying quiver: $(i,j,l,g)$, $(i,h,f,g)$, $(i,h,e,b)$, $(k,d,f,g)$, $(k,d,e,b)$, $(k,c,a,b)$ for both examples, with all there sub-sequences of consecutive edges.

Cube chains of dimension 1 are, for the 2 holes on the anti-diagonal example: $(C,f,g)$, $(C,e,b)$, $(i,h,D)$, $(k,d,D)$ and all sub-sequences of consecutive cubes including either $C$ or $D$. Said in a different manner, the cube chains of dimension 1 for this example are ``whiskerings'' of $C$ (resp. $D$) with any cube chains of dimension 0 which compose either on the left or on the right with $C$ (resp. $D$). 

The cube chains of dimension 1 for the 2 holes on the diagonal example are 
$(k,F,b)$, $(i,E,g)$ and suitable sub-sequences. 

Now, let us give an example calculation of the boundary maps and operator. For instance in the last example, consider the chain $(k,F,b)$ of dimension 1. It is of type $(1,2,1)$. Let us compute $d_{2,\mathcal{I}}(k,F,b)$ for all $\mathcal{I} \subseteq \{0,\ldots,1\}$ having $r$ elements, where $0 < r < 2$, i.e. having exactly one element. Hence we have to compute $d_{2,\{0\}}(k,F,b)$ and $d_{2,\{1\}}(k,F,b)$. Supposing the boundaries of 2-cells indexed by $0$ are the vertical 1-cells, and that the boundary indexed by $1$ are the horizontal ones:   
\begin{itemize}
\item $d_{2,\{0\}}(k,F,b)=(k,d^0_{\{1\}}(F),d^1_{\{0\}}(F),b)=(k,d,e,b)$
\item $d_{2,\{1\}}(k,F,b)=(k,d^0_{\{0\}}(F),d^1_{\{1\}}(F),b)=(k,c,a,b)$
\end{itemize}


Then, $$\begin{array}{rcl}
\partial(k,F,b) & = & \sum\limits_{\mathcal{I} \subseteq\{0,\ldots,1\} \mbox{ of card 1}}
(-1)^{5} sgn(\mbox{$\mathcal{I}$}) {d_{2,\mbox{$\mathcal{I}$}}(k,F,b)} \\
& = & -sgn(\{0\}) d_{2,\{0\}}(k,F,b)-sgn(\{1\}) d_{2,\{1\}}(k,F,b) \\
& = & d_{2,\{0\}}(k,F,b)-d_{2,\{1\}}(k,F,b) \\
& = & (k,d,e,b) - (k,c,a,b)
\end{array}$$
\end{example}


\begin{example}[3-cube]
\label{ex:3cube}
Consider now the full 3-cube $X$ below, with 8 vertices having 0 and 1 coordinates along the three axes (we will note them by concatenating their coordinates: 000, 001, 010, 100, 110, 101, 001, 111) and 12 1-cells $a01$, $0b1$, $a11$, $1b1$, $01c$, $00c$, $11c$, $10c$, $a00$, $0b0$, $a10$, $1b0$, with $d^0(a01)=001$, $d^1(a01)=101$, $d^0(0b1)=001$, $d^1(0b1)=011$, $d^0(a11)=011$, $d^1(a11)=111$, $d^0(1b1)=101$, $d^1(1b1)=111$, $d^0(01c)=010$, $d^1(01c)=011$, $d^0(00c)=000$, $d^1(00c)=001$, $d^0(11c)=110$, $d^1(11c)=111$, $d^0(10c)=100$, $d^1(10c)=101$, $d^0(a00)=000$, $d^1(a00)=100$, $d^0(0b0)=000$, $d^1(0b0)=010$, $d^0(a10)=010$, $d^1(a10)=110$, $d^0(1b0)=100$ and $d^1(1b0)=110$: 

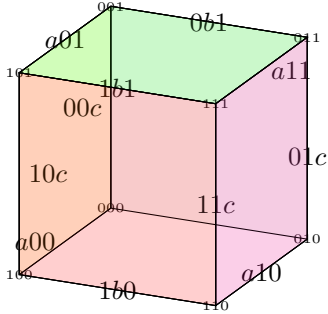
\begin{center}
\begin{tikzpicture}[scale=3.5,tdplot_main_coords]
    \coordinate (O) at (0,0,0);
    \tdplotsetcoord{P}{1.414213}{54.68636}{45}
    \draw[-,fill=green,fill opacity=0.1] (Pz) -- (Pyz) -- (P) -- (Pxz) -- cycle;
    \draw[-] (Pz) edge node {$a01$} (Pxz); 
    \draw[-] (Pxz) edge node {$1b1$} (P); 
    \draw[-] (Pz) edge node {$0b1$} (Pyz);
    \draw[-] (Pyz) edge node[above,right] {$a11$} (P);
    \draw[-] (Py) edge node {$a10$} (Pxy); 
  \node at (O) {\tiny $000$};
  \node at (P) {\tiny $111$};
  
  \node at (Px) {\tiny $100$};
  \node at (Py) {\tiny $010$};
  \node at (Pz) {\tiny $001$};
  \node at (Pxy) {\tiny $110$};
  \node at (Pyz) {\tiny $011$};
  \node at (Pxz) {\tiny $101$};

   \draw[-,fill=red,fill opacity=0.1] (Px) -- (Pxy) -- (P) -- (Pxz) -- cycle;
   \draw[-] (Pxy) edge node {$11c$} (P); 
   \draw[-] (Py) edge node[below] {$01c$} (Pyz);
    
   \draw[-,fill=magenta,fill opacity=0.1] (Py) -- (Pxy) -- (P) -- (Pyz) -- cycle;
    \draw[-] (Px) edge node {$1b0$} (Pxy); 
    \draw[-] (Px) edge node[right] {$10c$} (Pxz);

    \draw[-,fill=gray!50,fill opacity=0.1] (O) -- (Py) -- (Pyz) -- (Pz) -- cycle;
    \draw[-] (O) edge node[left] {$00c$} (Pz); 
     \draw[-] (O) edge node[left] {$a00$} (Px); 

    \draw[-,fill=yellow,fill opacity=0.1] (O) -- (Px) -- (Pxz) -- (Pz) -- cycle;

    \draw[-,fill=green,fill opacity=0.1] (Pz) -- (Pyz) -- (P) -- (Pxz) -- cycle;

    \draw[-,fill=red,fill opacity=0.1] (Px) -- (Pxy) -- (P) -- (Pxz) -- cycle;

    \draw[-,fill=magenta,fill opacity=0.1] (Py) -- (Pxy) -- (P) -- (Pyz) -- cycle;
  \end{tikzpicture}
\end{center}

The 2-cells are $A$ and $A'$ orthogonal to the vertical axis, respectively the bottom and top faces, $B$ and $B'$ which are the back and front faces and $C$ and $C'$ which are the left and right faces on the figure above. More precisely, $A$ has boundaries $d^0_0(A)=a00$, $d^0_1(A)=0b0$, $d^1_0(A)=a10$ and $d^1_1(A)=1b0$; $B$ has boundaries $d^1_0(B)=a01$, $d^0_1(B)=00c$, $d^0_0(B)=a00$ and $d^1_1(B)=10c$; $C$ has boundaries $d^1_0(C)=0b1$, $d^0_1(C)=00c$, $d^1_1(C)=01c$ and $d^0_0(C)=0b0$; $A'$ has boundaries $d^0_0(A')=a01$, $d^0_1(A')=0b1$, $d^1_0(A')=a11$ and $d^1_1(A')=1b1$; $B'$ has boundaries $d^1_0(B')=a11$, $d^0_1(B')=01c$, $d^0_0(B')=a10$ and $d^1_1(B')=11c$; $C'$ has boundaries $d^1_0(C')=1b1$, $d^1_1(C')=11c$, $d^0_0(C')=1b0$ and $d^0_0(C')=10c$. 

The unique 3-cell $S$ has as boundaries $A$, $B$, $C$, $A'$, $B'$ and $C'$ with $d^0_0(S)=C$, $d^0_1(S)=B$, $d^0_2(S)=A$, $d^1_0(S)=C'$, $d^1_1(S)=B'$, $d^1_2(S)=A'$. 

The 0-cube chains are all subpaths of the six following maximal 1-paths, $\zeta=(0b0,a10,11c)$, $\epsilon=(0b0,01c,a11)$, $\alpha=(a00,1b0,11c)$, $\beta=(a00,10c,1b1)$, $\gamma=(00c,a01,1b1)$, $\delta=(00c,0b1,a11)$, pictured below as bold dark lines, in the same order: 
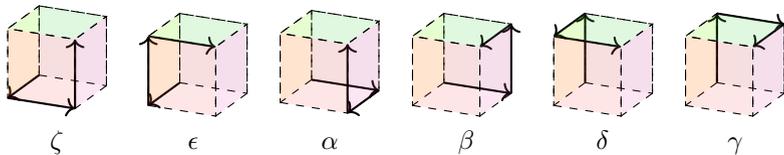
\begin{center}
\begin{tabular}{cccccc}
  \begin{tikzpicture}[scale=1.2,tdplot_main_coords]
    \coordinate (O) at (0,0,0);
    \tdplotsetcoord{P}{1.414213}{54.68636}{45}

    \draw[->,thick,fill=gray!50,fill opacity=0.3] (O) -- (Px);
    \draw[->,thick,fill=gray!50,fill opacity=0.3] (Px) -- (Pxy);
    \draw[->,thick,fill=gray!50,fill opacity=0.3] (Pxy) -- (P);

    \draw[dashed,fill=gray!50,fill opacity=0.1] (O) -- (Py) -- (Pyz) -- (Pz) -- cycle;
    \draw[dashed,fill=yellow,fill opacity=0.1] (O) -- (Px) -- (Pxz) -- (Pz) -- cycle;
    \draw[dashed,fill=green,fill opacity=0.1] (Pz) -- (Pyz) -- (P) -- (Pxz) -- cycle;
    \draw[dashed,fill=red,fill opacity=0.1] (Px) -- (Pxy) -- (P) -- (Pxz) -- cycle;
    \draw[dashed,fill=magenta,fill opacity=0.1] (Py) -- (Pxy) -- (P) -- (Pyz) -- cycle;
\end{tikzpicture}
& 
  \begin{tikzpicture}[scale=1.2,tdplot_main_coords]
    \coordinate (O) at (0,0,0);
    \tdplotsetcoord{P}{1.414213}{54.68636}{45}

    \draw[->,thick,fill=gray!50,fill opacity=0.3] (O) -- (Px);
    \draw[->,thick,fill=gray!50,fill opacity=0.3] (Px) -- (Pxz);
    \draw[->,thick,fill=gray!50,fill opacity=0.3] (Pxz) -- (P);

    \draw[dashed,fill=gray!50,fill opacity=0.1] (O) -- (Py) -- (Pyz) -- (Pz) -- cycle;
    \draw[dashed,fill=yellow,fill opacity=0.1] (O) -- (Px) -- (Pxz) -- (Pz) -- cycle;
    \draw[dashed,fill=green,fill opacity=0.1] (Pz) -- (Pyz) -- (P) -- (Pxz) -- cycle;
    \draw[dashed,fill=red,fill opacity=0.1] (Px) -- (Pxy) -- (P) -- (Pxz) -- cycle;
    \draw[dashed,fill=magenta,fill opacity=0.1] (Py) -- (Pxy) -- (P) -- (Pyz) -- cycle;
\end{tikzpicture}
&
  \begin{tikzpicture}[scale=1.2,tdplot_main_coords]
    \coordinate (O) at (0,0,0);
    \tdplotsetcoord{P}{1.414213}{54.68636}{45}

    \draw[->,thick,fill=gray!50,fill opacity=0.3] (O) -- (Py);
    \draw[->,thick,fill=gray!50,fill opacity=0.3] (Py) -- (Pxy);
    \draw[->,thick,fill=gray!50,fill opacity=0.3] (Pxy) -- (P);

    \draw[dashed,fill=gray!50,fill opacity=0.1] (O) -- (Py) -- (Pyz) -- (Pz) -- cycle;
    \draw[dashed,fill=yellow,fill opacity=0.1] (O) -- (Px) -- (Pxz) -- (Pz) -- cycle;
    \draw[dashed,fill=green,fill opacity=0.1] (Pz) -- (Pyz) -- (P) -- (Pxz) -- cycle;
    \draw[dashed,fill=red,fill opacity=0.1] (Px) -- (Pxy) -- (P) -- (Pxz) -- cycle;
    \draw[dashed,fill=magenta,fill opacity=0.1] (Py) -- (Pxy) -- (P) -- (Pyz) -- cycle;
\end{tikzpicture}
&
  \begin{tikzpicture}[scale=1.2,tdplot_main_coords]
    \coordinate (O) at (0,0,0);
    \tdplotsetcoord{P}{1.414213}{54.68636}{45}

    \draw[->,thick,fill=gray!50,fill opacity=0.3] (O) -- (Py);
    \draw[->,thick,fill=gray!50,fill opacity=0.3] (Py) -- (Pyz);
    \draw[->,thick,fill=gray!50,fill opacity=0.3] (Pyz) -- (P);

    \draw[dashed,fill=gray!50,fill opacity=0.1] (O) -- (Py) -- (Pyz) -- (Pz) -- cycle;
    \draw[dashed,fill=yellow,fill opacity=0.1] (O) -- (Px) -- (Pxz) -- (Pz) -- cycle;
    \draw[dashed,fill=green,fill opacity=0.1] (Pz) -- (Pyz) -- (P) -- (Pxz) -- cycle;
    \draw[dashed,fill=red,fill opacity=0.1] (Px) -- (Pxy) -- (P) -- (Pxz) -- cycle;
    \draw[dashed,fill=magenta,fill opacity=0.1] (Py) -- (Pxy) -- (P) -- (Pyz) -- cycle;
\end{tikzpicture}
&
  \begin{tikzpicture}[scale=1.2,tdplot_main_coords]
    \coordinate (O) at (0,0,0);
    \tdplotsetcoord{P}{1.414213}{54.68636}{45}

    \draw[->,thick,fill=gray!50,fill opacity=0.3] (O) -- (Pz);
    \draw[->,thick,fill=gray!50,fill opacity=0.3] (Pz) -- (Pxz);
    \draw[->,thick,fill=gray!50,fill opacity=0.3] (Pxz) -- (P);

    \draw[dashed,fill=gray!50,fill opacity=0.1] (O) -- (Py) -- (Pyz) -- (Pz) -- cycle;
    \draw[dashed,fill=yellow,fill opacity=0.1] (O) -- (Px) -- (Pxz) -- (Pz) -- cycle;
    \draw[dashed,fill=green,fill opacity=0.1] (Pz) -- (Pyz) -- (P) -- (Pxz) -- cycle;
    \draw[dashed,fill=red,fill opacity=0.1] (Px) -- (Pxy) -- (P) -- (Pxz) -- cycle;
    \draw[dashed,fill=magenta,fill opacity=0.1] (Py) -- (Pxy) -- (P) -- (Pyz) -- cycle;
\end{tikzpicture}
&
  \begin{tikzpicture}[scale=1.2,tdplot_main_coords]
    \coordinate (O) at (0,0,0);
    \tdplotsetcoord{P}{1.414213}{54.68636}{45}

    \draw[->,thick,fill=gray!50,fill opacity=0.3] (O) -- (Pz);
    \draw[->,thick,fill=gray!50,fill opacity=0.3] (Pz) -- (Pyz);
    \draw[->,thick,fill=gray!50,fill opacity=0.3] (Pyz) -- (P);

    \draw[dashed,fill=gray!50,fill opacity=0.1] (O) -- (Py) -- (Pyz) -- (Pz) -- cycle;
    \draw[dashed,fill=yellow,fill opacity=0.1] (O) -- (Px) -- (Pxz) -- (Pz) -- cycle;
    \draw[dashed,fill=green,fill opacity=0.1] (Pz) -- (Pyz) -- (P) -- (Pxz) -- cycle;
    \draw[dashed,fill=red,fill opacity=0.1] (Px) -- (Pxy) -- (P) -- (Pxz) -- cycle;
    \draw[dashed,fill=magenta,fill opacity=0.1] (Py) -- (Pxy) -- (P) -- (Pyz) -- cycle;
\end{tikzpicture}
\\
$\zeta$ & $\epsilon$ & $\alpha$ & $\beta$ & $\delta$ & $\gamma$ 
\end{tabular}
\end{center}
And the 1 cube chains are depicted below, with bold dark lines and more opaque 2-cells: 
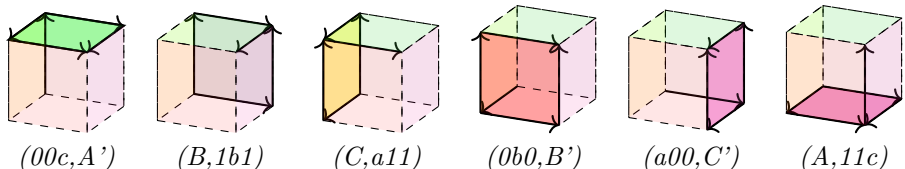
\begin{center}
\begin{tabular}{cccccc}
  \begin{tikzpicture}[scale=1.4,tdplot_main_coords]
    \coordinate (O) at (0,0,0);
    \tdplotsetcoord{P}{1.414213}{54.68636}{45}
    \draw[->,thick,fill=green,fill opacity=0.3] (Pz) -- (Pyz) -- (P) -- (Pxz) -- cycle;
\draw [->,thick] (O) -- (Pz);
    \draw[->,thick,fill=green,fill opacity=0.3] (Pz) -- (Pyz); 
    \draw[->,thick,fill=green,fill opacity=0.3] (Pyz) -- (P); 
    \draw[->,thick,fill=green,fill opacity=0.3] (Pz) -- (Pxz);
    \draw[->,thick,fill=green,fill opacity=0.3] (Pxz) -- (P);

    \draw[dashed,fill=gray!50,fill opacity=0.1] (O) -- (Py) -- (Pyz) -- (Pz) -- cycle;
    \draw[dashed,fill=yellow,fill opacity=0.1] (O) -- (Px) -- (Pxz) -- (Pz) -- cycle;
    \draw[dashed,fill=green,fill opacity=0.1] (Pz) -- (Pyz) -- (P) -- (Pxz) -- cycle;
    \draw[dashed,fill=red,fill opacity=0.1] (Px) -- (Pxy) -- (P) -- (Pxz) -- cycle;
    \draw[dashed,fill=magenta,fill opacity=0.1] (Py) -- (Pxy) -- (P) -- (Pyz) -- cycle;
  \end{tikzpicture}
&
  \begin{tikzpicture}[scale=1.4,tdplot_main_coords]
    \coordinate (O) at (0,0,0);
    \tdplotsetcoord{P}{1.414213}{54.68636}{45}

    \draw[fill=gray!50,fill opacity=0.3] (O) -- (Py) -- (Pyz) -- (Pz) -- cycle;
    \draw[->,thick,fill=gray!50,fill opacity=0.3] (O) -- (Py);
    \draw[->,thick,fill=gray!50,fill opacity=0.3] (Py) -- (Pyz);
    \draw[->,thick,fill=gray!50,fill opacity=0.3] (O) -- (Pz);
    \draw[->,thick,fill=gray!50,fill opacity=0.3] (Pz) -- (Pyz);
    \draw[->,thick,fill=gray!50,fill opacity=0.3] (Pyz) -- (P);

    \draw[dashed,fill=gray!50,fill opacity=0.1] (O) -- (Py) -- (Pyz) -- (Pz) -- cycle;
    \draw[dashed,fill=yellow,fill opacity=0.1] (O) -- (Px) -- (Pxz) -- (Pz) -- cycle;
    \draw[dashed,fill=green,fill opacity=0.1] (Pz) -- (Pyz) -- (P) -- (Pxz) -- cycle;
    \draw[dashed,fill=red,fill opacity=0.1] (Px) -- (Pxy) -- (P) -- (Pxz) -- cycle;
    \draw[dashed,fill=magenta,fill opacity=0.1] (Py) -- (Pxy) -- (P) -- (Pyz) -- cycle;
\end{tikzpicture}
&
  \begin{tikzpicture}[scale=1.4,tdplot_main_coords]
    \coordinate (O) at (0,0,0);
    \tdplotsetcoord{P}{1.414213}{54.68636}{45}

    \draw[fill=yellow,fill opacity=0.3] (O) -- (Px) -- (Pxz) -- (Pz) -- cycle;
    \draw[->,thick,fill=yellow,fill opacity=0.3] (O) -- (Px);
    \draw[->,thick,fill=yellow,fill opacity=0.3] (Px) -- (Pxz);
    \draw[->,thick,fill=yellow,fill opacity=0.3] (Pz) -- (Pxz);
    \draw[->,thick,fill=yellow,fill opacity=0.3] (O) -- (Pz);
    \draw[->,thick,fill=yellow,fill opacity=0.3] (Pxz) -- (P);

    \draw[dashed,fill=gray!50,fill opacity=0.1] (O) -- (Py) -- (Pyz) -- (Pz) -- cycle;
    \draw[dashed,fill=yellow,fill opacity=0.1] (O) -- (Px) -- (Pxz) -- (Pz) -- cycle;
    \draw[dashed,fill=green,fill opacity=0.1] (Pz) -- (Pyz) -- (P) -- (Pxz) -- cycle;
    \draw[dashed,fill=red,fill opacity=0.1] (Px) -- (Pxy) -- (P) -- (Pxz) -- cycle;
    \draw[dashed,fill=magenta,fill opacity=0.1] (Py) -- (Pxy) -- (P) -- (Pyz) -- cycle;
\end{tikzpicture}
& 
  \begin{tikzpicture}[scale=1.4,tdplot_main_coords]
    \coordinate (O) at (0,0,0);
    \tdplotsetcoord{P}{1.414213}{54.68636}{45}

    \draw[fill=red,fill opacity=0.3] (Px) -- (Pxy) -- (P) -- (Pxz) -- cycle;
    \draw[->,thick,fill=red,fill opacity=0.3] (O) -- (Px);
    \draw[->,thick,fill=red,fill opacity=0.3] (Px) -- (Pxy);
    \draw[->,thick,fill=red,fill opacity=0.3] (Pxy) -- (P);
    \draw[->,thick,fill=red,fill opacity=0.3] (Pxz) -- (P);
    \draw[->,thick,fill=red,fill opacity=0.3] (Px) -- (Pxz);

    \draw[dashed,fill=gray!50,fill opacity=0.1] (O) -- (Py) -- (Pyz) -- (Pz) -- cycle;
    \draw[dashed,fill=yellow,fill opacity=0.1] (O) -- (Px) -- (Pxz) -- (Pz) -- cycle;
    \draw[dashed,fill=green,fill opacity=0.1] (Pz) -- (Pyz) -- (P) -- (Pxz) -- cycle;
    \draw[dashed,fill=red,fill opacity=0.1] (Px) -- (Pxy) -- (P) -- (Pxz) -- cycle;
    \draw[dashed,fill=magenta,fill opacity=0.1] (Py) -- (Pxy) -- (P) -- (Pyz) -- cycle;
\end{tikzpicture}
&
  \begin{tikzpicture}[scale=1.4,tdplot_main_coords]
    \coordinate (O) at (0,0,0);
    \tdplotsetcoord{P}{1.414213}{54.68636}{45}

    \draw[fill=magenta,fill opacity=0.3] (Py) -- (Pxy) -- (P) -- (Pyz) -- cycle;
    \draw[->,thick,fill=magenta,fill opacity=0.3] (Py) -- (Pxy);
    \draw[->,thick,fill=magenta,fill opacity=0.3] (Pxy) -- (P);
    \draw[->,thick,fill=magenta,fill opacity=0.3] (Py) -- (Pyz);
    \draw[->,thick,fill=magenta,fill opacity=0.3] (Pyz) -- (P);
    \draw[->,thick,fill=magenta,fill opacity=0.3] (O) -- (Py);

    \draw[dashed,fill=gray!50,fill opacity=0.1] (O) -- (Py) -- (Pyz) -- (Pz) -- cycle;
    \draw[dashed,fill=yellow,fill opacity=0.1] (O) -- (Px) -- (Pxz) -- (Pz) -- cycle;
    \draw[dashed,fill=green,fill opacity=0.1] (Pz) -- (Pyz) -- (P) -- (Pxz) -- cycle;
    \draw[dashed,fill=red,fill opacity=0.1] (Px) -- (Pxy) -- (P) -- (Pxz) -- cycle;
    \draw[dashed,fill=magenta,fill opacity=0.1] (Py) -- (Pxy) -- (P) -- (Pyz) -- cycle;
\end{tikzpicture} 
&
  \begin{tikzpicture}[scale=1.4,tdplot_main_coords]
    \coordinate (O) at (0,0,0);
    \tdplotsetcoord{P}{1.414213}{54.68636}{45}

    \draw[fill=magenta,fill opacity=0.3] (O) -- (Py) -- (Pxy) -- (Px) -- cycle;
    \draw[->,thick,fill=magenta,fill opacity=0.3] (Pxy) -- (P);
    \draw[->,thick,fill=magenta,fill opacity=0.3] (Px) -- (Pxy);
    \draw[->,thick,fill=magenta,fill opacity=0.3] (Py) -- (Pxy);
    \draw[->,thick,fill=magenta,fill opacity=0.3] (O) -- (Px);
    \draw[->,thick,fill=magenta,fill opacity=0.3] (O) -- (Py);

    \draw[dashed,fill=gray!50,fill opacity=0.1] (O) -- (Py) -- (Pyz) -- (Pz) -- cycle;
    \draw[dashed,fill=yellow,fill opacity=0.1] (O) -- (Px) -- (Pxz) -- (Pz) -- cycle;
    \draw[dashed,fill=green,fill opacity=0.1] (Pz) -- (Pyz) -- (P) -- (Pxz) -- cycle;
    \draw[dashed,fill=red,fill opacity=0.1] (Px) -- (Pxy) -- (P) -- (Pxz) -- cycle;
    \draw[dashed,fill=magenta,fill opacity=0.1] (Py) -- (Pxy) -- (P) -- (Pyz) -- cycle;
    \draw[dashed,fill=purple,fill opacity=0.1] (O) -- (Py) -- (Pxy) -- (Px) -- cycle;
\end{tikzpicture} 
\\
(00c,A') & (B,1b1) & (C,a11) & (0b0,B') & (a00,C') & (A,11c) \\
\end{tabular}
\end{center}

As with Example \ref{ex:2holes}, we have the following calculations for $\partial$: 
$\partial(00c,A')=\gamma-\delta$, 
$\partial(B,1b1)=\beta-\gamma$, 
$\partial(C,a11)=\epsilon-\delta$, $\partial(0b0,B')=\zeta-\epsilon$, 
$\partial(a00,C')=\alpha-\beta$, 
$\partial(A,11c)=\alpha-\zeta$.

There is a unique 2 cube chain which is $(S)$. 
We now exemplify the computation of the boundary of this unique 2 cube chain below. 



$$\begin{array}{rcl}
\partial (S) & = & \sum\limits_{r=1}^{2} \sum\limits_{\mathcal{I} \subseteq\{0,1,2\} \mbox{ of card $r$}} 
(-1)^{r} sgn(\mathcal{I}) {d_{1,\mathcal{I}}((S))} \\
& = & sgn(\{0,1\}) {d_{1,\{0,1\}}((S))} + sgn(\{0,2\}) {d_{1,\{0,2\}}((S))} \\
&  & \ \ \ \ \ + sgn(\{1,2\}) {d_{1,\{1,2\}}((S))} 
- sgn(\{0\}) {d_{1,\{0\}}((S))} \\
& & \ \ \ \ \ - sgn(\{1\}) {d_{1,\{1\}}((S))}- sgn(\{2\}) {d_{1,\{2\}}((S))} \\
& = & {d_{1,\{0,1\}}((S))} -{d_{1,\{0,2\}}((S))} + {d_{1,\{1,2\}}((S))} 
- {d_{1,\{0\}}((S))} \\
& & \ \ \ \ \ \ \ \ \ \ \ \ \ \ \ + {d_{1,\{1\}}((S))}- {d_{1,\{2\}}((S))} \\
& = & (d^0_{\{2\}}(S),d^1_{\{0,1\}}(S))
- (d^0_{\{1\}}(S),d^1_{\{0,2\}}(S))
+ (d^0_{\{0\}}(S),d^1_{\{1,2\}}(S)) \\
& & \ \ \ \ \ 
- (d^0_{\{1,2\}}(S),d^1_{\{0\}}(S))
+ (d^0_{\{0,2\}}(S),d^1_{\{1\}}(S))
- (d^0_{\{0,1\}}(S),d^1_{\{2\}}(S)) \\
& = & (A,11c)
- (B,1b1)
+(C,a11)
-(a00,C')
+(0b0,B')
-(00c,A') \\
\end{array}$$
\end{example}


\section{Associative algebras and modules over algebras}

\label{sec:assoc}

We refer the reader to Appendix \ref{sec:A} for a recap of the definitions of algebras and modules over algebras. 


As is well known \cite{modabelian}, 
the categories of right and left $R$-modules, where $R$ is a commutative ring (not even a field), are abelian. 
$R$-bimodules are $R\otimes R^{op}$-modules, so they form an abelian category as well. The adaptation of the well-known results on $R$-modules to the case of $A$-modules, where $A$ is an algebra is immediate (see e.g. \cite{schiffler2014quiver}),  
in particular: 
\begin{itemize}
\item The \look{0 object} in the category of $A$-bimodules is 0 as an $R$-bimodule with the 0 action of elements of $A$. 
\item The \look{coproduct} $M\oplus N$ (isomorphic to the cartesian product) of $A$-bimodu\-les $M$ and $N$ has the coproduct of the underlying $R$-vector spaces as its underlying $R$-vector space, with action $a\bullet (m+n)\bullet b=a\bullet_M m\bullet_M b+a\bullet_N n \bullet_N b$, with $m \in M$ and $n\in N$. 
\item The \look{kernel} of $f: \ M \rightarrow N$, a morphism of $A$-bimodule is given by: 
$$
Ker \ f =\{ x \in M \ \mid \ f(x)=0\}
$$
\noindent with action $a\bullet_{Ker \ f} k { }_{Ker \ f} \bullet b=a \bullet_M k { }_M \bullet b$ (which is valid since $f(a\bullet_M k { }_M \bullet b)=a\bullet_M f(k) { }_M \bullet b=0$). 
\item The \look{cokernel} of $f: \ M \rightarrow N$, a morphism of $A$-bimodule, is given by:
$$
coKer \ f = N/f(M)
$$
\noindent with $a \bullet_{coKer \ f} [k] { }_{coKer \ f} \bullet b= [a \bullet_N k { }_N \bullet b]$. This is well defined since, first, as $f$ is a morphism of $A$-bimodules, $a \bullet_N f(m) { }_N \bullet b=f(a \bullet_M m { }_M \bullet b)$ and $f(M)$ is an $A$-submodule of $N$. Secondly, the definition of the action does not depend on the particular representative chosen for $[k]$: suppose $k'=k+f(m)$ then $a \bullet_{coKer \ f} [k'] { }_{coKer \ f} \bullet b= [a \bullet_N k' { }_N \bullet b]=[a \bullet_N k { }_N \bullet b]+[a \bullet_N f(m) { }_N \bullet b]=a \bullet_{coKer \ f} [k] { }_{coKer \ f} \bullet b+[f(a \bullet_M m { }_M \bullet b)]=a \bullet_{coKer \ f} [k] { }_{coKer \ f} \bullet b$.
\end{itemize}

\begin{ground}
In the rest of the article, although some results apply when $R$ is a commutative ring, for sake of simplicity, we suppose that $R$ is a field.
\end{ground}

Consider now the more general category of modules over possibly varying algebras:

\begin{definition}
\label{def:mod}
\look{The category of right-modules (resp. left) over algebras} has:
\begin{itemize}
\item As objects, pairs $(M,A)$ where $A$ is an algebra and $M$ is a right (resp. left) $A$-module
\item As morphisms from $(M,A)$ to $(M',A')$, pairs of morphisms $(f,g)$, where $f: \ M \rightarrow M'$ is a morphism of right (resp. left) $R$-vector spaces, $g: \ A \rightarrow A'$ is a morphism of algebras such that for all $m\in M$ and $a\in A$, $f(m\bullet a)=f(m')\bullet g(a)$
\end{itemize}
We denote this category by $Mod_R$ (resp. ${ }_R Mod$). 
\end{definition}

The definition of the category of bimodules is similar and is denoted by ${ }_R Mod_R$. 

\paragraph{Restriction/extension of scalars}


Definition \ref{def:mod} is linked to the notion of ``restriction of scalars". 
Restriction of scalars defines a functor \cite{modabelian} as follows. We spell this out in the case of bimodules, since this is the case of interest to this paper. First:  

\begin{definition}
\label{def:restr1}
Let $g: \ A \rightarrow A'$ be a morphism of algebras. $A'$ can be considered as an $A$-bimodule using the action $a_1 { }_A \bullet a'\bullet_A a_2=g(a_1)\times a'\times g(a_2)$, for all $a_1, a_2 \in A$, $a' \in A'$.    
\end{definition}

Then, we can define: 

\begin{definition}
\label{def:restrictionscalars}
The \look{restriction of scalars} functor (along $g$) is 
$$g^*: \ { }_R Mod_R \ {A'} \rightarrow { }_R Mod_{R} \ A$$ 
\noindent with 
\begin{itemize}
    \item $g^*(M)$ being $M$ as an $R$-vector space, and the action of $A$ on $m \in M$ is by definition $a_1 { }_A \bullet m \bullet_A a_2=g(a_1) { }_{A'} \bullet m\bullet_{A'} g(a_2)$. 
\item 
Let $f: \ M' \rightarrow N'$ be a morphism of $A'$-bimodules. Then $g^*(f)(m')=f(m')$. 
\end{itemize}
\end{definition}

Indeed, the definition of $g^*(f)$ is valid since $g^*(f)(a_1 { }_A \bullet m\bullet_A a_2)=f(a_1 { }_A \bullet m\bullet_A a_2)=f(g(a_1) { }_{A'}\bullet m \bullet_{A'} g(a_2)) = g(a_1) { }_{A'} \bullet f(m)\bullet_{A'} g(a_2)=a_1 { }_A \bullet f(m)\bullet_A a_2$.

Definition \ref{def:mod} of morphisms of ${ }_R Mod_{R}$ is equivalent to asking them to be $(f,g)$ with $f$ being a morphism of $A$-bimodules from $M$ to the restriction of scalars of $M'$ along $g$. 

\begin{lemma}
\label{lem:extensionscalars}
Restriction of scalars has a left adjoint \cite{modabelian}:  
$$g_!: \ { }_A Mod \rightarrow { }_{A'} Mod$$
\noindent called the \look{extension of scalars} functor along $g$ functor, which is:
\begin{itemize}
    \item $g_!(M)$ is, as an $A$-bimodule, $A' \otimes_{A} M\otimes_{A^{op}} A'$, with $A'$ considered as an $A$-bimodule as in Definition \ref{def:restr1}. 
    Hence as an $R$-vector space, it is generated by $a'_1\otimes m \otimes a'_2$ with $a'_1, a'_2 \in A'$ and $m \in M$ 
    such that, for all $a_1, a_2 \in A$: 
        \begin{multline}
        \label{eq:restrictionscalars}
    a'_1 \otimes (a_1 { }_A \bullet m \bullet_A a_2) \otimes a'_2=(a'_1 { }_A \bullet a_1) \otimes m \otimes (a_2 \bullet_A a'_2)\\=(a'_1\times g(a_1))\otimes m \otimes (g(a_2) \times a'_2)
    \end{multline}
    Finally, as a $A'$-bimodule, the action of $a'_3 \in A'$ on the left and $a'_4 \in A'$ on the right is defined as:
    \begin{equation}
    \label{eq:restrictionscalars2}
        a'_3 { }_{A'} \bullet (a'_1 \otimes m \otimes a'_2) \bullet_{A'} a'_4=(a'_3 \times a'_1) \otimes m \otimes (a'_2 \times a'_4)
        \end{equation}
\item 
Let $h: \ M \rightarrow N$ be a morphism of $A$-bimodules, then $g_!(h)(a'_1 \otimes m\otimes a'_2)=a'_1 \otimes h(m)\otimes a'_2$. 
\end{itemize}
\end{lemma}

Indeed, the action on morphisms is well defined since $g_!(h)(a'_1 \otimes (a_1 \bullet_M m { }_M \bullet a_2)\otimes a'_2)=a'_1 \otimes h(a_1 \bullet_M m{ }_M\bullet a_2)\otimes a'_2$ and $g_!(h)((a'_1 \times g(a_1)) \otimes m \otimes (g(a_2)\times a'_2))=(a'_1 \times g(a_1)) \otimes h(m)\otimes (g(a_2)\times  a'_2)=a'_1 \otimes (a_1 \bullet_N h(m){ }_N \bullet a_2) \otimes a'_2=h(a'_1 \otimes (a_1 \bullet_M m{ }_M\bullet a_2)\otimes a'_2=g_!(h)(a'_1 \otimes (a_1 \bullet_M m{ }_M \bullet a_2)\otimes a'_2)$.


\begin{remark}
\label{rem:rem4}
Suppose $g$ is an inclusion of algebras with $A \subseteq A'$, hence $g(a)=a$ for $a \in A$. 
We note that, for $a_1, a_2 \in A$ and $m \in M$, we can identify $a_1 \otimes m \otimes a_2$ in $g_!(M)$ with $a_1 { }_A \bullet m \bullet_A a_2$. This is done by identifying $1_{A'} \otimes n \otimes 1_{A'}$  with $n \in M$ and by noticing that, by Equation \ref{eq:restrictionscalars}, as $1_A=1_{A'}$:  

\begin{eqnarray*}
1_A \otimes (a_1\bullet m \bullet a_2) \otimes 1_A & = & (1_A\times g(a_1)) \otimes m \otimes (g(a_2)\times 1_A) \\
& = & (1_A \times a_1) \otimes m \otimes (a_2 \times 1_A) \\
& = & a_1 \otimes m \otimes a_2
\end{eqnarray*}
\end{remark}

\begin{remark}
Restriction of scalars also has a right adjoint called \look{co-extension of scalars}, constructed in a similar manner but with the $Hom$ functor between $A$-modules, instead of the tensor product. 
\end{remark}

\begin{remark}
An alternative definition of the category of modules would have morphisms $(f,g): \ (M,A) \rightarrow (M',A')$ with $g: \ A' \rightarrow A$ morphism of algebras and $f$ morphism of $A'$-bimodules from $g_!(M)$ to $M'$. As is well known \cite{twocatmod}, these two ways of defining categories of modules are equivalent in the case of almost finitely generated projective modules, which includes finitely generated modules, which will be the main case studied in this paper. 
\end{remark}

\section{Path algebras and modules over path algebras}

\label{sec:pathalgebras}

Let $\C$ be a small category. 

\begin{definition}
\label{def:categoryalgebra}
The \look{category algebra} or \look{convolution algebra} 
$R[\C]$ of $\C$ over $R$ is the $R$-algebra
whose underlying $R$-vector space is the free module 
$R[\C_1]$ over the set of morphisms of $\C$ and 
whose product operation is defined on basis-elements 
$f$, $g\in \C_1 \subseteq R[\C]$
to be their composition if they are composable and zero otherwise:
$$f\times g=
\left\{\begin{array}{ll}
g\circ f & \mbox{if composable} \\
0 & \mbox{otherwise} 
\end{array}
\right.$$
\end{definition}

Alternatively, the category algebra of $\C$ can be constructed from the linearization of $\C$ (its algebroid) by somehow forgetting its categorical structure, that is, by replacing composition by the algebra multiplication. 

Note that when $\C$ is a groupoid, $R[\C]$ is a star-algebra (an algebra equipped with an anti-involution). 
When $\C$ is the free category of a poset $P$, $R[\C]$ is the incidence algebra of the poset $P$. When $\C$ is the free category of a quiver $Q$, $R[\C]$ is known as the path algebra of $Q$. 
When $\C$ is a category, $R[\C]$ is a quotient of the path algebra of its underlying quiver by an ``obvious'' ideal (the one that has as relations $f\times g-g\circ f=0$). 



\begin{remark}
\label{rem:notfunctorial}
Finally, note that the category algebra is not a functorial construction. This is because there is no reason, given a functor $F$ from $\C$ to $\D$, its linearization (the obvious linear map induced by $F$ on the underlying $R$-vector space of $R[\C]$) preserves internal multiplication. 

In fact, it is easy to see that only injective-on-objects functors are mapped naturally onto algebra morphisms, see e.g. \cite{goubault2023semiabelian}. 
\end{remark}


\paragraph{Posets and incidence algebras} (see for instance \cite{assocalg})


Consider the case of a partial order $\leq$ over any $n$-element set $S$. We enumerate $S$ as $s_1,\ldots, s_n$, and in such a way that the enumeration is compatible with the order $\leq$ on $S$, that is, $s_i \leq s_j$ implies $j \leq i$, which is always possible.

Then the path algebra can be seen 
as the subalgebra $\TM_n(R)$ of the algebra of $n\times n$ matrices with coefficients in $R$, $\SM_n(R)$, together with addition, multiplication by scalars, and matrix multiplication, made of lower triangular matrices. 



Consider for instance the following partial order: 

\[
\begin{tikzcd}
n \arrow[r] & n-1 \arrow[r] & \cdots \arrow[r] & 1
\end{tikzcd}
\]

Its path algebra is exactly the algebra of lower triangular matrices, that we denote by: 
$$
\begin{pmatrix}
R & 0 & \ldots & 0 & 0 \\
R & R & \ldots & 0 & 0 \\
\ldots \\
R & R & \ldots & R & 0 \\
R & R & \ldots & R & R
\end{pmatrix}
$$

\paragraph{Idempotents, quivers and path algebras}

\label{sec:idempotents}




Let $A$ be an $R$-algebra. The following are classical notions, see e.g. \cite{assocalg}. An element $e \in A$ is called an idempotent if $e^2 = e$. The idempotent $e$ is said to be central if $a\times e = e\times a$ for all $a \in A$. The idempotents $e_1, e_2 \in A$ are called orthogonal if $e_1 \times e_2 = e_2\times e_1 = 0$. The idempotent $e$ is said to be primitive if $e$ cannot be written as a sum $e = e_1 + e_2$, where $e_1$ and $e_2$ are nonzero orthogonal idempotents of $A$.

 
 \begin{definition}[\cite{assocalg}]
 We say that $\{e_1,\ldots,e_n\}$ is a \look{complete set of primitive orthogonal idempotents} of $A$ if the $e_i$ are primitive idempotents of $A$ that are pairwise orthogonal and such that $e_1 +\ldots + e_n$ is the unit of $A$.
 \end{definition}
 

\begin{remark}
\label{rem:completeidempotents}
In the case of path algebras $R[Q]$ of some quiver $Q$, primitive idempotents are constants paths on each vertex $i$ of $Q$, $e_i$. They form a set of orthogonal primitive idempotents, as for $i\neq j$ two distinct vertices of $Q$, $e_i \times e_j=0$. They form a complete set of primitive orthogonal idempotents if $Q$ has finitely many vertices. 
\end{remark}


We refer to Appendix \ref{sec:B} for more notions on algebras that relate to path algebras. 
We now recap the following lemma, that will prove useful for practical calculations: 

\begin{lemma}[\cite{assocalg}]
\label{lem:acyclicquiver}
Let $Q$ be a connected, finite, and acyclic quiver with $Q_0 = \{1,2,\ldots,n\}$ such that, for each $i, j \in Q_0$, $j \leq i$ whenever there exists a path from $i$ to $j$ in $Q$. Then the path algebra $R[Q]$ is isomorphic to the
triangular matrix algebra
$$
\begin{pmatrix}
e_1(R[Q])e_1 & 0 & \ldots & 0 \\
e_(R[Q])e_1 & e_2(R[Q])e_2 & \ldots & 0 \\
\ldots \\
e_n(R[Q])e_1 & e_n(R[Q])e_2 & \ldots & e_n(R[Q])e_n 
\end{pmatrix}
$$
\end{lemma}

\begin{example}[\cite{assocalg})]
\label{ex:kronecker}
Consider: 
\[
\begin{tikzcd}
2 \arrow[r,bend left,"\alpha"] \arrow[r,bend right,swap,"\beta"] & 1
\end{tikzcd}
\]

In that case, the corresponding path algebra is, by Lemma \ref{lem:acyclicquiver} is the following matrix algebra: 
$$
\begin{pmatrix}
R & 0 \\
R^2 & R
\end{pmatrix}
$$
\noindent known as the Kronecker algebra. Elements of this Kronecker algebra are of the form: 
$$
\begin{pmatrix}
a & 0 \\
(b,c) & d
\end{pmatrix}
$$
\noindent with the ``obvious'' addition and external multiplication, and as internal (algebra) multiplication, the ``obvious one'' as well: 
$$
\begin{pmatrix}
a & 0 \\
(b,c) & d
\end{pmatrix}
\times
\begin{pmatrix}
a' & 0 \\
(b',c') & d'
\end{pmatrix}
=
\begin{pmatrix}
a a' & 0 \\
(ba'+db',ca'+dc') & dd'
\end{pmatrix}
$$
\end{example}

\begin{example}[Empty square]
\label{ex:emptysquare}
We apply Lemma \ref{lem:acyclicquiver} to: 
\[\begin{tikzcd}
  4 \arrow[r] \arrow[d]
    & 2 \arrow[d] \\
  3 \arrow[r]
& 1 \end{tikzcd}
\]
We get the following path algebra: 
$$
\begin{pmatrix}
R & 0 & 0 & 0 \\
R & R & 0 & 0 \\
R & 0 & R & 0 \\
R^2 & R & R & R 
\end{pmatrix}
$$
\end{example}

From a finite precubical set, we can also define a quiver and a corresponding path algebra, as follows: 

\begin{definition}
\label{def:quiveralgebra}
Let $X$ be a finite precubical set and $R$ be a commutative field. 
We define $\look{R[X]}$ to be the quiver algebra $R[X_{\leq 1}]$ of the underlying quiver $X_{\leq 1}$ of $M$. 

Generators of the quiver algebra $R[X_{\leq 1}]$, as an $R$-vector space, are 0-cube chains indexed by their start and end points $p=(p_1,\ldots,p_l)_{u',u}$, possibly empty ($l=0$), in that case $u'=u$. Empty 0-cube chains from $u$ to $u$ are also noted $e_u$ as they are going to be the idempotents of the quiver algebra. When the context makes it clear, we will drop the subscripts on generators of the path algebra. 

Let $p=(p_1,\ldots,p_l)_{u',u}$ from $u'$ to $u$ and $q=(q_1,\ldots,q_m)_{v,v'}$ from $v$ to $v'$ be 1-cube chains in $X$. 
These are the base elements of $R[X]$ as an $R$-vector space. The internal multiplication on these basis elements of $R[M]$ is: 
$$p\times q=\left\{\begin{array}{ll}
(p_1,\ldots,p_l,q_1,\ldots,q_m)_{u',v'} & \mbox{if $u=v$}\\
0 & \mbox{otherwise}
\end{array}\right.
$$
\end{definition}

The fact we are considering finite precubical sets when defining $R[X]$ allows us to make it a unital associative algebra, as is well known \cite{assocalg}. The unit of the algebra is the sum of all (finitely many) idempotents, i.e. constant paths on each vertex, as we have seen before already. 



\paragraph{Quivers and bimodules over path algebras}

Bimodules over path algebras enjoy a number of interesting properties. The first one that will be useful for representing the homology modules we are going to introduce in Section \ref{sec:homologymodules} is that bimodules over path algebras are naturally bigraded over pairs of vertices of the underlying graph, similarly to path algebras, Lemma \ref{lem:acyclicquiver}: 


\begin{lemma}
\label{lem:bimoddecomp}
Let $Q$ be a quiver with finitely many vertices and $M$ be a $R[Q]$-bimodule. Then $M$ is isomorphic, as an $R$-vector space, to the coproduct of all $R$-vector spaces $e_a \bullet M \bullet e_b$, for $a, b \in Q_0$.
\end{lemma}

\begin{proof}
As noted in Remark \ref{rem:completeidempotents}, $\{e_a \ \mid \ a \in Q_0\}$ is a complete set of primitive orthogonal idempotents of $R[Q]$, hence for any $R[Q]$-bimodule $M$, $M = 1\bullet M \bullet 1$ with $1=\sum\limits_{a \in Q_0} e_a$, so any $m\in M$ decomposes into $m= \sum\limits_{a, b \in Q_0} e_a \bullet m \bullet e_b$. 

Furthermore, each $e_a \bullet M \bullet e_b$ is distinct from one another. Indeed, consider e.g. $a' \neq a \in Q_0$ and $m\in e_a \bullet M \bullet e_b \cap e_{a'} \bullet M \bullet e_b$. We can write $m=e_a \bullet m \bullet e_b$ and $m=e_{a'}\bullet m \bullet e_b$ so $m=e_a \bullet (e_{a'} \bullet m \bullet e_b) \bullet e_b$, which, by the axioms of bimodules is equal to $e_a \times e_{a'} \bullet m \bullet e_b \times e_b$. But $e_a$ and $e_{a'}$ are orthogonal, so $e_a \times e_{a'}=0$, hence $m=0$. 

We can similarly prove that for $b \neq b' \in Q_0$, $m=e_a \bullet m \bullet e_b=e_a \bullet m \bullet e_{b'}$ is necessarily equal to 0. 

Therefore $M$ is isomorphic, as an $R$-vector space, to the coproduct of all $R$-vector spaces $e_a \bullet M \bullet e_b$, for $a, b \in Q_0$.     
\end{proof}

This leads us to the natural notation for $R[Q]$-bimodules below: 

\begin{remark} 
\label{rem:notation}
We note that for $X$ a precubical set, for all $v \in X_0$ and $w \in X_0$, and any $R[X]$-bimodule $M$, $e_v \bullet M \bullet e_w$ is an $R$-vector space. We represent any $R[X]$-bimodule $M$ as a matrix of $R$-vector spaces which $(v,w)$ entry is $e_v \bullet M \bullet e_w$, similarly to what we have been doing with path algebras in the previous paragraph. We will see the interest of this notation later, e.g. in Example \ref{ex:2holesbis}.
\end{remark}

\begin{remark}
\label{rem:notationtensor}
We also note that representing a $R[X]$-bimodule $M$ by a matrix $(M_{i,j})$ whose columns and lines are indexed by the vertices $x \in X_0$, and similarly with a $R[Y]$-bimodule $N$, represented by a matrix $(N_{i,j})$ with columns indexed by the vertices $y\in Y_0$, the $R[X]\otimes R[Y]$-bimodule $M\otimes N$ defined in Section \ref{sec:assoc} is represented using the same matrix notation as in Remark \ref{rem:notation} as a $X_0\times Y_0$ indexed matrix, with entries $((i, i'),(j, j'))$ being $M_{i,j}\otimes N_{i',j'}$.
\end{remark}

\section{Homology in modules over the path algebra}


\subsection{Definition and first basic properties}



\label{sec:homologymodules}

\label{sec:precub}





\begin{lemma}
\label{lem:free}
Let $X$ be a finite precubical set. For all $i\geq 0$, the $R$-vector space $R_{i}[X]$ ($i\geq 1$) of $(i-1)$-cube chains of Definition \ref{def:modcub} can be given the structure of a $R[X]$-bimodule, which is free for $i\geq 2$: given 0-cube chains in $X$, $p=(p_1,\ldots,p_l)_{u',u}$ from $u'$ to $u$, $q=(q_1,\ldots,q_m)_{v,v'}$ from $v$ to $v'$, the following is the $R[X]$-bimodule action of $\langle p,q\rangle$ on $r=(r_1,\ldots,r_n)$, a $(i-1)$-dimensional cube chain in $R_i[X]$: 
$$p\bullet r \bullet q=\left\{\begin{array}{ll}
(p_1,\ldots,p_l,r_1,\ldots,r_n,q_1,\ldots,q_m)_{u',v'} & \mbox{if $d^0(r_1)=d^1(p_l)$}\\
& \mbox{ \ \ and $d^1(r_n)=d^0(q_1)$} \\
0 & \mbox{otherwise}
\end{array}\right.
$$
\noindent the full action of $R[X]$ being defined by bilinearity. 
\end{lemma}

%
 
\begin{proof}
Note that for $i=1$ the bimodule operation is just the left concatenation by $p$ composed with the right concatenation by $q$. Indeed, as an $R$-vector space, $R_1[X]$ is isomorphic to $R[X]$ and 
any algebra is indeed a bimodule over itself. The unit of the algebra $1=\sum\limits_{a\in X_0} e_a$ is the unique generator of $R_1[X]$ as a bimodule over $R[X]$. Indeed, we can generate all of $R_1[X]$ as $R[X]\bullet 1 \bullet R[X]$. This is in general not a free bimodule, and this is why we will need to distinguish the cases $i=1$ from $i\geq 2$ in the sequel of the article (see e.g. Lemma \ref{lem:normalform}). 



For $i\geq 2$, obviously, 
$(p_1,\ldots,p_l,r_1,\ldots,r_n,q_1,\ldots,q_m)_{u',v'}$ is an $(i-1)$-dimensional cube chain when $p$, $q$ are 0-dimensional cube chains and $r$ is an $(i-1)$-dimensional cube chain, see Remark \ref{rem:dimension}.


Finally, we prove that it is free over $R[X]$,  generated by $(i-1)$-cube chains $(p_1,\ldots,p_n)$ with $p_1 \in X_{\geq 2}$ and $p_n \in X_{\geq 2}$ ($n$ is possibly equal to 1), forming a set of generators we denote by $G_i[X]$. 

We first prove that $G_i[X]$ forms a set of generators for $R_{i}[X]$. As $R_i[X]$ is the $R$-vector space generated by $(i-1)$-cube chains, it is generated by $(i-1)$-cube chains of the form $p=(p_1,\ldots,p_m)_{u,v}$ where $p_j \in X$, $1\leq j\leq m$. As $i\geq 2$, not all of the $p_j$ can be in $X_1$, there must be at least one of the $p_j$ belonging to $X_{\geq 2}$. Let $l$ be such that $p_{l+1}$ is the first one of these $p_j\in X_{\geq 2}$ and $n+l$ be the last one of these $p_j \in X_{\geq 2}$. Then obviously, $r=(p_1,\ldots,p_l)\in R_1[X]$, $s=(p_{n+l+1},\ldots,p_m)\in R_1[X]$, $c=(p_{l+1},\ldots,p_{l+n})\in G_i[X]$ and $p=r\bullet c\bullet s$ by definition. 


We now prove freeness of the elements in $G_i[X]$. 
Suppose that we have:
\begin{equation} 
\label{eq:free}
\sum\limits_{i\in I, j\in J} \lambda_{i,j} (p_{i,j}\bullet g_j \bullet q_{i,j})=0
\end{equation}
\noindent with $\lambda_{i,j} \in R$ and $p_{i,j}, q_{i,j} \in R[X]$. We can always suppose that \begin{equation} 
\label{eq:condfree}
p_{i',j}=p_{i,j} \mbox{ and } q_{i',j}=q_{i,j} \mbox{ implies } i=i'
\end{equation}
\noindent if not, we only need to change the coefficients $\lambda_{i,j}$ and sets $I$ and $J$ so that this property is satisfied. Now we want to prove that Equation (\ref{eq:free}) (under the condition of Equation (\ref{eq:condfree})) implies that all $\lambda_{i,j}$ are equal to 0. 
As $R_i[X]$ is the $R$-vector space generated by elements of precisely the form $p_{i,j}\bullet g_j \bullet q_{i,j}$ we need to prove that no $p_{i,j}\bullet g_j \bullet q_{i,j}$ can be equal to a $p_{i',j'}\bullet g_{j'} \bullet q_{i',j'}$ with $j'\neq j$. This amounts to proving that $R[X]\bullet g \bullet R[X] \cap R[X]\bullet g' \bullet R[X]= \{0\}$ for $g \neq g'$ two elements of $R_i[X]$. 

Suppose $R[X]\bullet g \bullet R[X] \cap R[X]\bullet g' \bullet R[X]\neq \{0\}$, with $g$ and $g'$ in $G_{i}[X]$, and consider some non null element $c$ in the intersection. Thus $c$ is an $(i-1)$-cube chain $(c_1,\ldots,c_m)$, with $(c_i,\ldots,c_j)=g$ for some indices $i$ and $j$ and $(c_k,\ldots,c_l)=g'$ for some indices $k$ and $l$. If $i\neq k$ then suppose $i<k$ for instance, and as $c_i$ is in $X_{\geq 2}$ by hypothesis, the dimension of $(c_1,\ldots,c_m)$ is at least 1 (because the cell $c_i\in X_u$ accounts for $u-1$ in the calculus of the dimension of the cube chain $c$) plus the dimension of $g'$ (by additivity of dimension, Remark \ref{rem:dimension}), which is $i-1$, so the dimension of $c$ cannot be $i-1$. Hence $i$ must be equal to $k$. Similarly, $j$ must be equal to $l$ and $g=g'$. Hence $R_i[X]$ is freely generated by $G_i[X]$ as a $R[X]$-bimodule. 

\end{proof}

\begin{lemma}
\label{lem:precubboundary}
The boundary operator $\partial: \ R_{i+1}[X] \rightarrow R_i[X]$ ($i\geq 1$) of Definition \ref{def:boundaryoperator} 
is a morphism of $R[X]$-modules, that we write $\partial_{\mid R_{i+1}[X]}$ when we need to make precise on which cube chains it applies to. 
We also set $\partial_{\mid R_1[X]}: \ R_1[X] \rightarrow 0$ to be the 0 linear map. 

Boundary operators restrict to $R$-vector space morphisms from $R_{i+1}[X](a,b)$ to $R_i[X](a,b)$ where $R_i[X](a,b)$ is the sub $R$-vector space $R^{b}_{a,i}[X]=e_a \bullet R_i[X] \bullet e_b$. 

Hence $(R_i[X],\partial)$ is a chain of bimodules on the algebra $R[X]$ and for all $a$, $b$ in $X_0$, $(e_a\bullet R_i[X] \bullet e_b,\partial)$ is a chain of $R$-vector spaces.
\end{lemma}

\begin{proof}
For a cube chain $c = (c_1,\ldots,c_l) \in Ch_w^{v}(X)$ of type $(n_1,\ldots,n_l)$ and dimension $i+1$, an integer $k \in \{1,\ldots,l\}$ and a subset $\mathcal{I} \subseteq \{0,\ldots,n_k-1\}$ having $r$ elements, where $0 < r < n_k$, we have, from Section \ref{sec:precub}: 
$$
d_{k,\mathcal{I}}(c) = (c_1,\ldots,c_{k-1},d^0_{\overline{\mathcal{I}}}(c_k),d^1_\mathcal{I}(c_k),c_{k+1},\ldots,c_l) \in Ch(X)^w_v$$ 
\noindent where $\overline{\mathcal{I}} = \{0,\ldots,n_k-1\}\backslash \mathcal{I}$. 

We first note that $c \in Ch_w^v(X)$, $d_{k,\mathcal{I}}(c) \in Ch_w^v(X)$. Because $\partial c$ is a linear combination of the $d_{k,\mathcal{I}}(c)$, see Definition \ref{def:boundaryoperator}, this implies that $\partial c$ is in $R_i[X](w,v)$. Hence, by linearity of $\partial$, $\partial$ maps $R_{i+1}[X](w,v)$ to $R_{i}[X](w,v)$. 

Now, consider $p=(p_1,\ldots,p_m) \in Ch(X)_{v'}^v$ and $q=(q_1,\ldots,q_n) \in Ch(X)^{w'}_{w}$. Therefore, $p \bullet c \bullet q=(p_1,\ldots,p_m,c_1,\ldots,c_l,q_1,\ldots,q_n)$ is of type $(s_1,\ldots,s_{m+l+n})=(1,\ldots,1,n_1,\ldots,n_l,1,\cdots,1)$. The only indices $k$ such that there exists $r$ with $0 < r < s_k$ are $k=m+1,\ldots,m+l$, and for such $k$ and $\mathcal{I} \subseteq \{0,\ldots,s_k-1\}$ having $r$ elements with $0 < r < s_k$: 
$$\begin{array}{rcl}
d_{k,\mathcal{I}}(p\bullet c \bullet q) & = &(p_1,\ldots,p_m,c_1,\ldots,c_{k-1-m},d^0_{\overline{\mathcal{I}}}(c_{k-m}), \\
& & \ \ \ \ \ d^1_\mathcal{I}(c_{k-m}),c_{k+1-m},\ldots,c_l,q_1,\ldots,q_n) \\
& = & p \bullet d_{k-m}(c) \bullet q
\end{array}
$$ 

Now, $\partial (p\bullet c \bullet q)$ is equal to: 

\begin{align*}
& \scriptstyle \sum\limits_{k=m+1}^{m+l} \sum\limits_{r=1}^{n_{k-m}-1} \sum\limits_{\mathcal{I} \subseteq\{0,\ldots,n_{k-m}-1\}: \ \mid \mathcal{I} \mid=r}
(-1)^{m+n_1+\ldots+n_{k-m-1}+k+r+1} sgn(\mathcal{I}) \ {d_{k,\mathcal{I}}(p\bullet c \bullet q)}\\
= & \scriptstyle \sum\limits_{k=m+1}^{m+l} \sum\limits_{r=1}^{n_{k-m}-1} \sum\limits_{\mathcal{I} \subseteq\{0,\ldots,n_{k-m}-1\}: \ \mid \mathcal{I} \mid=r}
(-1)^{m+n_1+\ldots+n_{k-m-1}+k+r+1} sgn(\mathcal{I}) \ {p\bullet d_{k-m,\mathcal{I}}(c)\bullet q}\\
= & \scriptstyle p \bullet \left( \sum\limits_{k=m+1}^{m+l} \sum\limits_{r=1}^{n_{k-m}-1} \sum\limits_{\mathcal{I} \subseteq\{0,\ldots,n_{k-m}-1\}: \ \mid\mathcal{I}\mid=r}
(-1)^{m+n_1+\ldots+n_{k-m-1}+k+r+1} sgn(\mathcal{I}) \ {d_{k-m,\mathcal{I}}(c)} \right) \bullet q \\
= & \scriptstyle p \bullet \left( \sum\limits_{k=1}^{l} \sum\limits_{r=1}^{n_{k-m}-1} \sum\limits_{\mathcal{I} \subseteq\{0,\ldots,n_{k}-1\}: \ \mid\mathcal{I}\mid=r}
(-1)^{2m+n_1+\ldots+n_{k-m-1}+k+r+1} sgn(\mathcal{I}) \ {d_{k-m,\mathcal{I}}(c)} \right) \bullet q \\
= & \scriptstyle p \bullet \left( \sum\limits_{k=1}^{l} \sum\limits_{r=1}^{n_{k-m}-1} \sum\limits_{\mathcal{I} \subseteq\{0,\ldots,n_{k}-1\}: \ \mid \mathcal{I}\mid=r}
(-1)^{n_1+\ldots+n_{k-m-1}+k+r+1} sgn(\mathcal{I}) \ {d_{k,\mathcal{I}}(c)} \right) \bullet q \\
= & p \bullet \partial c \bullet q 
\end{align*}
\noindent and $\partial$ is a $R[X]$-bimodule morphism.
\end{proof}


\begin{example}[2 holes on diagonal/antidiagonal]
\label{ex:2holesbis}
We consider the two precubical sets of Example \ref{ex:2holes}.
For both precubical sets, we have the following path algebra $R[X]$ (lines and columns are indexed by vertices 1 to 9 as depicted in Example \ref{ex:2holes}), as a direct application of Lemma \ref{lem:acyclicquiver}: 
$$
\begin{pmatrix}
R & 0 & 0 & 0 & 0 & 0 & 0 & 0 & 0 \\
R & R & 0 & 0 & 0 & 0 & 0 & 0 & 0 \\
R & R & R & 0 & 0 & 0 & 0 & 0 & 0 \\
R & 0 & 0 & R & 0 & 0 & 0 & 0 & 0 \\
R^2 & R & 0 & R & R & 0 & 0 & 0 & 0 \\
R^3 & R^2 & R & R & R & R & 0 & 0 & 0 \\
R & 0 & 0 & R & 0 & 0 & R & 0 & 0 \\
R^3 & R & 0 & R^2 & R & 0 & R & R & 0 \\
R^6 & R^3 & R & R^3 & R^2 & R & R & R & R 
\end{pmatrix}
$$
We note that $R_1[X]$ seen as a $R[X]$-bimodule has $\bullet$ coinciding with the algebra operation of $R[X]$. Hence, using the notation for $R[X]$-bimodules that we introduced earlier on, the $\bullet$ operation coincides with the matrix multiplication of the matrix algebra above. 
\end{example}







\begin{definition}
\label{def:hommodule}
The \look{homology modules} of a finite precubical set $X$ is defined as the homology $(HM_{i+1})_{i\geq 0}[X]$ in the abelian category of $R[X]$-bimodules, of the complex of $R[X]$-bimodules defined in Lemma \ref{lem:precubboundary}, shifted by one, i.e. 
$$
HM_{i+1}[X]=Ker \ \partial_{\mid R_{i+1}[X]}/Im \ \partial_{\mid R_{i+2}[X]}
$$
\end{definition}

\begin{remark}
The reader may rightfully wonder what $HM_0[X]$ should be. Our best guess is that it should be linked to components \cite{components}. This is left for future work. This is indeed linked to the consideration of the bimodule $R_0[X]$ of ``reachable'' pairs of states in the precubical set $X$ that we discuss in the paragraph ``digression on resolutions of associative algebras and monoids" at the end of this section. 
\end{remark}

\begin{remark}
An obvious goal, which we leave for future work, is the existence and characterization of the analogue of Moore spaces for our module homology. Indeed, given the well known results that we recapped as Theorem \ref{thm:algfromquiver}, Appendix \ref{sec:B}, given any basic and connected finite dimensional $R$-algebra $A$, there exists an admissible ideal $I$ and a quiver $Q_A$ such that $A$ is isomorphic to $R[Q_A]/I$. Given a particular class of bimodules $M$ over $A$, can we hope to complete this graph into a precubical set $X^M_A$, by adding suitable cells, so that, for some $i\geq 1$, $HM_i[X^M_A]$ is isomorphic to $M$, as an $A$-bimodule? It is shown in \cite{ZIEMIANSKI201687} that we can find a precubical set (generated by the semantics of a PV program \cite{thebook}) for which the space of dipaths between beginning and end point has a given homology. This would pave the way for at least, realizing a certain class of $A$-bimodules as the homology modules of some precubical set. 
\end{remark}

\paragraph{Structure of the homology modules of precubical sets}

In order to make the first example calculations, we describe a little more how such quotients are computed in $R[X]$-bimodules: 

\begin{lemma}
\label{lem:decompquotient}
Let $X$ be a finite precubical set. The $R[X]$-bimodule $\look{HM_{i}[X]}$ ($i\geq 1$) is, as underlying $R$-vector space, the coproduct 
$\coprod\limits_{a,b \in M_0} N_{a,b}$ where the $R$-vector spaces $N_{a,b}$ are: 
$$ N_{a,b}=Ker \ \partial_{\mid e_a \bullet R_{i}[X] \bullet e_b}/Im \ \partial_{\mid e_a \bullet R_{i+1}[X] \bullet e_b}$$
The $R[X]$-bimodule operation is defined as follows, for $u$ a path from $a'$ to $a$ in $X_0$, and $v$ a path from $b$ to $b'$ in $X_0$, and writing $[n]_{a,b}$ for the class of $n \in Ker \ \partial_{\mid e_a \bullet R_{i}[X] \bullet e_b}$ modulo $Im \ \partial_{\mid e_a \bullet R_{i+1}[X] \bullet e_b}$:
$$
u \bullet [n]_{c,d} \bullet v = \left\{\begin{array}{ll} 
[u \bullet n \bullet v]_{a',b'} & \mbox{if $c=a$ and $d=b$}\\
0 & \mbox{otherwise}
\end{array}
\right.
$$
\end{lemma}

\begin{proof}
Consider $c \in Ker \ \partial$, sub-$R[X]$ bimodule of $R_{i}[X]$. From Lemma \ref{lem:bimoddecomp}, $c$ decomposes on the coproduct $\coprod\limits_{a,b \in X_0} e_a \bullet Ker \ \partial \bullet e_b$ as $\sum\limits_{a,b \in X_0} e_a \bullet c \bullet e_b$, and $\partial c = \sum\limits_{a,b \in X_0} \partial(e_a \bullet c \bullet e_b)$. As seen in Lemma \ref{lem:precubboundary}, 
$\partial$ is a $R[X]$-bimodule morphism, so $\partial(e_a\bullet c \bullet e_b)=e_a \bullet \partial(c) \bullet e_b$ and therefore, is equal to zero, since $\partial(c)=0$. This means that $c$ can be seen as belonging to the coproduct of the $R$-vector spaces 
$Ker \ \partial_{\mid e_a \bullet R_{i}[X] \bullet e_b}$, 
for all $a, b \in X_0$. 

Similarly, as $\partial$ is a $R[X]$-bimodule morphism, $Im \ \partial_{\mid R_{i+1}[X]}$ can be identified, when it comes to its $R$-vector space structure, with the coproduct of all $R$-vector spaces 
$Im \ \partial_{\mid e_a \bullet R_{i+1}[X] \bullet e_b}$. 

Hence the quotient 
$HM_i[X]=Ker \ \partial_{\mid R_{i}[M]}/Im \ \partial_{\mid R_{i+1}[M]}$ can be identified, as an $R$-vector space, with the quotient of the underlying $R$-vector spaces (see Section \ref{sec:assoc}), which by the above can be identified with the coproduct of all $R$-vector spaces 
$Ker \ \partial_{\mid e_a \bullet R_{i}[M] \bullet e_b}/Im \ \partial_{\mid e_a \bullet R_{i+1}[M] \bullet e_b}$. Therefore $e_a \bullet HM_i(M) \bullet e_b$ is the $R$-vector space $Ker \ \partial_{\mid e_a \bullet R_{i}[M] \bullet e_b}/Im \ \partial_{\mid e_a \bullet R_{i+1}[M] \bullet e_b}$.   

Finally, as already noted in Section \ref{sec:assoc}, the $R[M]$-bimodule operation is making the projection map
$\pi: \ Ker \ \partial_{\mid R_{i}[M]} \rightarrow Ker \ \partial_{\mid R_{i}[M]}/Im \ \partial_{\mid R_{i+1}[M]}$ a $R[M]$-bimodule homomorphism, hence we have necessarily the formula of the lemma for the $R[M]$-bimodule actions. 
\end{proof}

This will allow us to use the matrix notation of Remark \ref{rem:notation} for computing and presenting the homology bimodules. 

\paragraph{Functoriality and non-functoriality}


As we have remarked already, the construction of the quiver algebra is in general non-functorial, except for precubical maps $f$ from $X$ to $Y$ that are injective on vertices (i.e. which restrict to injective maps from $X_0$ to $Y_0$), see e.g. \cite{goubault2023semiabelian}. 

In what follows, ${Cub}_I$ stands for the full subcategory of precubical sets with proper non-looping length covering, with maps which are injective on vertices. 

Let $c=(c_1,\ldots,c_l)$ be an $i$-cube chain in $X$: then, for $f: \ X \rightarrow Y$ a morphism of precubical sets, $f(c)$ is the $i$-cube chain of the same type and length (hence of the same dimension) $(f(c_1)\ldots,f(c_l))$. 
The 1-cube chains in $X$ (resp. $Y$) are identified with basis elements of $R[X]$. 
Suppose that $f$ is injective, then this transform is functorial from $R[X]$ to $R[Y]$. 

Furthermore, as $f$ is a morphism of precubical sets, it commutes with the boundary operators. Consider 
a cube chain $c = (c_1,\ldots,c_l) \in Ch(X)^w_v$ of type $(n_1,\ldots,n_l)$, an integer $k \in \{1,\ldots,l\}$ and a subset $\mathcal{I} \subseteq \{0,\ldots,n_k-1\}$ having $r$ elements, where $0 < r < n_k$, we recall that: 
$$
d_{k,\mathcal{I}}(c) = (c_1,\ldots,c_{k-1},d^0_{\overline{\mathcal{I}}}(c_k),d^1_\mathcal{I}(c_k),c_{k+1},\ldots,c_l) \in Ch(X)^w_v$$ 
Now, 
$$
\begin{array}{rcl}
d_{k,\mathcal{I}}(f(c)) & = & (f(c_1),\ldots,f(c_{k-1}),d^0_{\overline{\mathcal{I}}}(f(c_k)),d^1_\mathcal{I}(f(c_k)),f(c_{k+1}),\ldots,f(c_l)) \\
& = & (f(c_1),\ldots,f(c_{k-1}),f(d^0_{\overline{\mathcal{I}}}(c_k)),f(d^1_\mathcal{I}(c_k)),f(c_{k+1}),\ldots,f(c_l)) \\
& = & f(c_1,\ldots,c_{k-1},d^0_{\overline{\mathcal{I}}}(c_k),d^1_\mathcal{I}(c_k),c_{k+1},\ldots,c_l)
\end{array}
$$ 
Therefore $\partial(f(c))$ is equal to: 

\begin{align*}
& \sum\limits_{k=1}^l \sum\limits_{r=1}^{n_k-1} \sum\limits_{\mathcal{I} \subseteq\{1,\ldots,n_k\}: \ \mid \mathcal{I}\mid =r}
(-1)^{n_1+\ldots+n_{k-1}+k+r+1} sgn(\mathcal{I}) \ {d_{k,\mathcal{I}}(f(c))} \\
= & \sum\limits_{k=1}^l \sum\limits_{r=1}^{n_k-1} \sum\limits_{\mathcal{I} \subseteq\{1,\ldots,n_k\}: \ \mid \mathcal{I}\mid =r}
(-1)^{n_1+\ldots+n_{k-1}+k+r+1} sgn(A) \ {f(d_{k,\mathcal{I}}(c))} \\
= & f(\partial c) \\
\end{align*}
\noindent and $f$ induces a map of chain complexes from $R_{i}[X]$ to $R_{i}[Y]$ seen as $R$-vector spaces. 

Finally, as we have seen, when $f$ is injective on objects, $f$ induces an algebra map from $R[X]$ to $R[Y]$ and $f$ induces a map of chain complexes from $R_{i}[X]$ to $R_{i}[Y]$, the first one seen as $R[X]$-bimodule, and the second one, as a $R[Y]$-bimodule.
This carries over the homology construction of Section \ref{sec:precub}.  We thus have proved: 


\begin{lemma}
\label{lem:diralgfunct}
The directed homology module construction defines a functor for each $n \in \N$, $HM_n: \ Cub_I \rightarrow { }_R Mod_R$. 
\end{lemma}



\subsection{Relative homology}

Given the particular interest in injective on vertices maps, it is reasonable to ask ourselves how to reason about ``relative spaces'': 


\begin{definition}
Let $Y$ be a sub-precubical set of a finite precubical set $X$ and $M$ be a $R[Y]$-bimodule. Indeed, this inclusion of precubical sets induces an inclusion of algebras $g: \ R[Y] \rightarrow R[X]$ since inclusions are injective on objects. We define $\look{{ }^X M}$ to be 
equal to $g_!(M)$, the extension of coefficients of $M$ along $g$.
\end{definition}

We note that the elements of ${ }^X M$, for such inclusions of precubical sets $Y \subseteq X$ have simple normal forms: 

\begin{lemma}
\label{lem:normalform}
Let $M$ be a $R[Y]$-bimodule, where $Y$ is a sub precubical set of $X$. Then, $p\otimes c \otimes q\in { }^X M$ with $p, q \in Ch^{=0}(X)$, $p$ from $u$ to $v$ and $q$ from $w$ to $z$, and $c\in M$, is a normal form under Equation (\ref{eq:restrictionscalars}) if:
\begin{itemize}
    \item $c=e_v\bullet c \bullet e_w$,
    \item suppose $p$ can be decomposed as 
$p=r \times s$. Then, $s$ not a constant path implies that $s\not \in Ch^{=0}(Y)$ and,
\item suppose $q$ can be decomposed as $q=z \times t$. Then, $z$ not a constant path implies that $z \not \in Ch^{=0}(Y)$. 
\end{itemize}

This implies that if $p'\otimes c' \otimes q'$ is another element of ${ }^X M$, with $p', q' \in Ch^{=0}(X)$ and $c' \in M$, then: 
$$p\otimes c \otimes q = p'\otimes c' \otimes q'$$ in ${ }^X M$ if and only if there exists $p_1, q_1\in Ch^{=0}(X)$ and $p_2, p'_2, q_2, q'_2\in Ch^{=0}(Y)$ such that: 
\begin{eqnarray}
\label{eq:equalityextension0}
p & = & p_1 \times p_2 \\
p' & = & p_1 \times p'_2 \\
q & = & q_2 \times q_1 \\
q' & = & q'_2 \times q_1 \\
p_2 \bullet c \bullet q_2 & = & p'_2 \bullet c' \bullet q'_2
\label{eq:equalityextension}
\end{eqnarray}
\end{lemma}

\begin{proof}
Equation (\ref{eq:restrictionscalars}) can be oriented from right to left, when applied on elements of the form $p \otimes c \otimes q$, in case the algebra $A$ is $R[X]$, for some precubical set $X$. This is because the length of $p$ plus the length of $q$ is decreasing, when we go from the right hand side to the left hand side of Equation (\ref{eq:restrictionscalars}), hence a finite number of applications of this rewrite rule will converge to a normal form. This implies that $p \otimes c \otimes q$ is a normal form when $p=r \times s$ and $s$ is a path in $X$ with at least one 1-cell, then $s\not \in Ch^{=0}(Y)$ and if $q=z \times t$ with $z$ a path in $X$ with at least one 1-cell, then $z \not \in Ch^{=0}(Y)$, otherwise, indeed, one could apply Equation (\ref{eq:restrictionscalars}) once more from right to left. 

Now, suppose that $p\otimes c \otimes q=p' \otimes c' \otimes q'$ in ${ }^X M$. This means that their normal forms should be equal, syntactically. Suppose the normal form of $p\otimes c \otimes q$ is $p_1 \otimes (p_2\bullet c \bullet q_2) \otimes q_1$ after a certain number of applications of Equation (\ref{eq:restrictionscalars}) from right to left, with $p_2, q_2 \in Ch^{=0}(Y)$, and that the normal form of $p' \times c' \times q'$ is $p'_1 \otimes (p'_2 \bullet c \bullet q'_2) \otimes q'_1$ with $p'_2, q'_2 \in Ch^{=0}(Y)$. Therefore, $p_2\bullet c \bullet q_2=p'_2 \bullet c \bullet q'_2$ and $p_1=p'_1$, $q_2=q'_2$, therefore, as $p=p_1\times p_2$, $p'=p'_1\times p'_2$, $q=q_1\times q_2$ and $q'=q'_1\times q'_2$, we get the equations of the lemma.  

Conversely, suppose that $p\otimes c \otimes q$ and $p' \otimes c' \otimes q'$ obey Equations (\ref{eq:equalityextension0}--\ref{eq:equalityextension}). Then: 
$$\begin{array}{rcl}
p \otimes c \otimes q & = & (p_1 \times p_2) \otimes c \otimes (q_2 \times q_1) \\
& = & p_1 \otimes (p_2 \bullet c \bullet q_2) \otimes q_1 \\
& = & p_1 \otimes (p'_2 \bullet c' \bullet q'_2) \otimes q_1 \\
& = & (p_1 \times p'_2) \otimes c' \otimes (q'_2 \times q_1) \\
& = & p' \otimes c' \otimes q'
\end{array}
$$
\end{proof}

We can refine this characterization of the equality of the generating elements of ${ }^X M$ as follows: 

\begin{lemma}
\label{lem:refinedeq}
Let $Y$ be a sub precubical set of $X$.
Suppose $M$ is a free $R[Y]$ bimodule with basis $(m_i)_{i\in I}$. Then, for $p, q, p', q' \in Ch^{=0}(X)$, $p \otimes m_i \otimes q=p' \otimes m_j \otimes q'$ in ${ }^X M$ if and only if $p=p'$, $i=j$ and $q=q'$.
\end{lemma}

\begin{proof}
Suppose $p\otimes m_i \otimes q=p' \otimes m_{j} \otimes q'$ modulo Equation (\ref{eq:restrictionscalars}). By Lemma \ref{lem:normalform}, this would mean that $p_2 \bullet m_i \bullet q_2=p'_2\bullet m_{j}\bullet q'_2$ for some $p_2, q_2, p'_2, q'_2 \in Ch^{=0}(Y)$ and $p=p_1 \times p_2$, $q=q_2\times q_1$, $p'=p_1 \times p'_2$, $q'=q'_2\times q_1$. But as $m_j$ and $m_{j'}$ are part of a free family in a $R[Y]$-bimodule, this means that $p_2=p'_2$, $q_2=q'_2$ and $i=j$ (that is, $m_i=m_j$).
Therefore, $p=p_1 \times p_2=p_1\times p'_2=p'$, $q=q_2\times q_1=q'_2\times q_1=q'$ and $i=j$. 
\end{proof}





We can now spell out the following characterization of ${ }^X M$ for free $R[Y]$-bimodules $M$: 

\begin{proposition}
\label{rem:relativerestriction}
Let $X$ be a finite precubical set and $Y$ a sub-precubical set of $X$.  
Then, for any free $R[Y]$-bimodule $M$, freely generated by $(m_j)_{j\in J}$ as a $R[Y]$-bimodule, ${ }^X M$ is the free $R[X]$-bimodule, freely generated by $(m_j)_{j\in J}$ as a $R[X]$-bimodule. 
\end{proposition}

\begin{proof}
Let $(m_j)_{j\in J}$ be a free family of generators for $M$ as a $R[Y]$-bimodule. As $Ch^{=0}(Y)$ generates the algebra $R[Y]$, we have for each $j\in J$, a family $(p'_{i,j})_{i\in I}$ and a family $(q'_{i,j})_{i\in I}$ with $p'_{i,j}, q'_{i,j} \in Ch^{=0}(Y)$, with $p'_{i,j}=p'_{i',j}$ and $q'_{i,j}=q'_{i',j}$ implies $i=i'$, such that 
$p'_{i,j}\bullet m_j \bullet q'_{i,j}$, where $j \in J$, $i\in I$ freely generates $M$ as an $R$-vector space. 


Also, 0-cube chains of $X$ are generators of the algebra $R[X]$ as an $R$-vector space, 
we know that elements of the form 
$p_{k,i,j}\otimes (p'_{i,j} \bullet m_j \bullet q'_{i,j}) \otimes q_{k,i,j}$, $p_{k,i,j}, q_{k,i,j} \in Ch^{=0}(X)$, $k \in K$, $i\in I$ and $j \in J$ generate $g_!(M)$ as an $R$-vector space, with $p_{k',i,j}=p_{k,i,j}$ and $q_{k',i,j}=q_{k,i,j}$ implies $k=k'$. Still, they do not form a free family in $g_!(M)$ because of the relations of Lemma \ref{lem:extensionscalars}. 

Still, we argue that the elements of the form $r_{l,j}\otimes m_j \otimes s_{l,j}$ with $r_{l,j}, s_{l,j} \in Ch^{=0}(X)$ and $r_{l',j}=r_{l,j}$ and $s_{l',j}=s_{l,j}$ implies $l=l'$, generate, freely, ${ }^X M$ as an $R$-vector space, and that this implies that ${ }^X M$ is generated by the $(m_j)_{j\in J}$, as a $R[X]$-bimodule. 

First, we note that by Equation \ref{eq:restrictionscalars}, 
$$
p_{k,i,j}\otimes (p'_{i,j} \bullet m_j \bullet q'_{i,j}) \otimes q_{k,i,j} = (p_{k,i,j}\times p'_{i,j})\otimes m_j \otimes (q'_{i,j}\times q_{k,i,j}) 
$$
\noindent So writing $r_{l,j}$ to be $p_{k,i,j}\times p'_{i,j}$ (resp. $s_{l,j}$ to be $q'_{i,j}\times q_{k,i,j}$) for $l \in L$ ($L$ in bijection with $K\times I$), we get that the $r_{l,j}\otimes m_j \otimes s_{l,j}$ generate ${ }^X M$ as an $R$-vector space. 

Now we prove that the $(m_j)_{j\in J}$ form a free family of ${ }^X M$ as an $R$-vector space. 
Suppose that: 
\begin{equation} 
\label{eq:free2}
\sum\limits_{i\in I, j\in J} \lambda_{i,j} (r_{i,j}\otimes m_j \otimes s_{i,j})=0
\end{equation}
\noindent with $\lambda_{i,j} \in R$ and $r_{i,j}, s_{i,j} \in R[X]$. We can always suppose that \begin{equation} 
\label{eq:condfree2}
r_{i',j}=r_{i,j} \mbox{ and } s_{i',j}=s_{i,j} \mbox{ implies } i=i'
\end{equation}
\noindent if not, we only need to change the coefficients $\lambda_{i,j}$ and sets $I$ and $J$ so that this property is satisfied. Now we want to prove that Equation (\ref{eq:free2}) (under the condition of Equation (\ref{eq:condfree2})) implies that all $\lambda_{i,j}$ are equal to 
0. This is the case if we can prove that the $r_{i,j}\otimes m_j \otimes s_{i,j}$ are all distinct for distinct $i$ and $j$. 

${ }^X M$ is generated by $r_{i,j}\otimes m_j \otimes s_{i,j}$ modulo Equations (\ref{eq:restrictionscalars}) and (\ref{eq:restrictionscalars2}). 
Suppose $r_{i,j}\otimes m_j \otimes s_{i,j}=r_{i',j'} \otimes m_{j'} \otimes s_{i',j'}$ modulo these equations. By Lemma \ref{lem:refinedeq}, this would mean that 
$r_{i,j}=r_{i',j'}$, $j=j'$ and $s_{i,j}=s_{i',j'}$, hence by Equation (\ref{eq:condfree2}), $i=i'$ as well. 
Thus ${ }^X M$ is freely generated as an $R$-vector space by $(r_{i,j}\otimes m_j \otimes s_{i,j})_{i \in I, j\in J}$, with $r_{i',j}=r_{i,j}$ and $s_{i',j}=s_{i,j}$ implies $i=i'$, meaning that ${ }^X M$ is generated by the $(m_j)_{j\in J}$ as a $R[X]$-bimodule. 

\end{proof}

In the case of $i$-cube chain bimodules, we get even more: 

\begin{corollary}
\label{lem:relsub0}
Let $X \in Cub$ and $Y$ a sub-precubical set of $X$. 
Then ${ }^X R_i[Y]$ is a sub-$R[X]$-bimodule of $R_i[X]$, for $i\geq 2$.  
\end{corollary}

\begin{proof}
Proposition \ref{rem:relativerestriction} applies to $X$, $Y$ and $M=R_i[Y]$ with generators $G_i(Y)$ (by Lemma \ref{lem:free}). 
Thus $R_i[X]$ is the free $R[X]$-bimodule generated by $G_i(X)$. But the elements of $G_i(Y)$ which are certain $i$-cube chains of $Y$ are in particular $i$-cube chains of $X$, that are in fact elements of $G_i(X)$. Hence the result.
\end{proof}

\begin{remark}
The proof could have been made a little more pedantic. Indeed, applying $g_!$ to an $A$-bimodule $M$ is tensoring it with some algebra $A'$ (seen as an $A$-bimodule, using $g$). But $M$ being a free $A$-bimodule implies in particular that $M$ is flat, hence that the tensor product by $M$ is exact. Since the inclusion of algebras of $R[Y]$ in $R[X]$ is indeed a monomorphism, its tensor by $M$ over $A$ is a monomorphism as well. This means that, at least as an $A$-bimodule, ${ }^X M$ is a sub-bimodule of ${ }^X M$, which is $M$. It is easy to see this holds as well for the $A'$-bimodule structures. 
\end{remark}

\begin{remark}
\label{rem:moreexplicitly}
More explicitly, ${ }^X R_i[Y]$ (for $i\geq 2$) is freely generated, as an $R$-vector space, by those, among the $i$-cube chains in $X$, that are concatenations of 0-cube chains in $X$ with generating $(i-1)$-cube chains in $Y$, with 0-cube chains in $X$. Hence all $y \in { }^X R_i[Y]$ can be uniquely written as:
$$
y = \sum\limits_{i\in I, j\in J} \lambda_{i,j} p_{i,j} \otimes g_j \otimes q_{i,j}
$$
\noindent with $\lambda_{i,j} \in R$, $p_{i,j}=(c^i_1,\ldots,c^i_{m_i})$ in $Ch^{=0}(X)$, $g_{j}=(d^j_1,\ldots,d^j_{n_j})$ in $G_i(Y)$, $q_{i,j}=(e^i_1,\ldots,e^i_{l_i})$ in $Ch^{=0}(X)$ with $d^1(c^i_m)=d^0(d_1)$ and $d^1(d_n)=d^0(e^i_1)$ and the condition that $p_{i',j}=p_{i,j}$ and $q_{i',j}=q_{i,j}$ implies $i=i'$. 

These elements $y$ are indeed identified with the following element in $R[X]$:
$$ 
\sum\limits_{i\in I, j\in J} \lambda_j (c^i_1,\ldots,c^i_{m_i},d^j_1,\ldots, d^j_{n_j},e^i_1,\ldots,e^i_{l_i})
$$
\end{remark}

For ${ }^X R_1[Y]$ we need to be a bit more cautious, since $R_1[Y]$ is not free in general and the unique generator of $R_1[Y]$ is $1_Y=\sum\limits_{a \in Y_0} e_a$ which may be different than the unique generator of $R_1[X]$, at least when $Y_0 \neq X_0$, which is $1_X=\sum\limits_{b \in X_0} e_b$, so the argument of Lemma \ref{lem:relsub0} does not apply directly at least. And indeed, we need extra requirements for the lemma to apply also for the case $i=1$: 

\begin{example}
Consider the precubical set $X$ of Example \ref{ex:2holes} (the left one, or ``two holes on the diagonal''). Take $Y$ the subprecubical set of $X$ generated by edges $i$ and $f$. Then $i\otimes (h,f,g)$ is a generator (as an $R$-vector space) of ${ }^X R_1[Y]$, as well as $(i,h)\otimes (f) \otimes (g)$. These are distinct elements in ${ }^X R_1[Y]$ whereas they should represent the unique 0-cube chain $(i,h,f,g)$ in $R_1[X]$. This discrepancy forbids the identification of ${ }^X R_i[Y]$ as a sub-$R[X]$-bimodule of $R_i[X]$ in general. 
\end{example}

We thus need an extra condition on the inclusion of $Y$ into $X$ to make sense of a well behaved ${ }^X R_1[Y]$ with respect to $R_1[X]$: 


\begin{definition}
\label{def:relpair}
Let $X\in Cub$ and $Y$ a sub-precubical set of $X$. We say that $(X,Y)$ is a \look{relative pair} (of precubical sets) if for all $c=(c_1,\ldots,c_m) \in Ch^{=0}(X)$ such that $d^0(c_1) \in Y_0$ (resp. $d^1(c_m) \in Y_0$): 
   there exists an index $k$ in $\{1,\ldots,m\}$ such that for all $l> k$, $c_l \not \in Y$, and for all $l < k$, $c_l \in Y$
   (resp. there exists an index $k$ in $\{1,\ldots,m\}$ such that for all $l >k$, $c_l \in Y$, and for all $l < k$, $c_l \not \in Y$) 
    \end{definition}
    Said in a different manner, $(X,Y)$ is a relative pair if all directed paths of $X$ can only enter once, and exit once $Y$: when $d^0(c_1) \in Y_0$ (resp. $d^1(c_m) \in Y_0$), $c=c^1 c^2$ with $c^1$ possibly empty or equal to $(c_1,\ldots,c_j)$ with $c_k\in Y_1$ (resp. $c_k\in X_1 \backslash Y_1$), $k=1,\ldots,j$, and $c^2$ possibly empty or equal to $(c_{j+1},\ldots,c_m)$ with $c_{k}\in X_1\backslash Y_1$ (resp. $c_k \in Y_1$), $k=1,\ldots,m$. 
    
Now, we show that for $(X,Y)$ a relative pair of precubical sets, ${ }^X R_1[Y]$ is the sub-$R[X]$-bimodule of $R_1[X]$ of dipaths of $X$ that ``go through Y'':

\begin{lemma}
\label{lem:relsub}
Let $(X,Y)$ be a relative pair of precubical sets as in Definition \ref{def:relpair}. Then ${ }^X R_1[Y]$ is a sub-$R[X]$-bimodule of $R_1[X]$.  
More precisely, 
all $y \in { }^X R_1[Y]$ can be written in a unique manner as: 
$$
y = \sum\limits_{i\in I,j \in J} \lambda_{i,j} p_{i,j }\otimes c_j \otimes q_{i,j}
$$
\noindent with $\lambda_{i,j} \in R$, $c_j=(e_1,\ldots,e_{k_j})$ $0$-cube chains in $Y$ (possibly empty, i.e. with $k_j=0$) and $p_{i,j}=(p^{i,j}_1,\ldots,$ $p^{i,j}_{l_{i,j}})$, $q_{i,j}=(q^{i,j}_{1},\ldots,q^{i,j}_{n_{i,j}})$, 0-cube chains in $X$ with no cell in $Y$, with $d^1(p^{i,j}_{l_{i,j}})=d^0(e_1)$ and $d^1(e_{{k_j}})=d^0(q^{i,j}_{1})$, and with $p_{i',j}=p_{i,j}$ and $q_{i',j}=q_{i,j}$ implies $i=i'$.  

Furthermore, under the inclusion ${ }^X R_1[Y] \hookrightarrow R_1[X]$, $p_{i,j}\otimes c_j \otimes q_{i,j}$ is identified to $(p_1^{i,j},\ldots,p_{l_{i,j}}^{i,j},e^j_1,\ldots,e^j_{k_j},q_1^{i,j},\ldots,q_{n_{i,j}}^{i,j})$. 
\end{lemma}

\begin{proof}
As $R_{1}[Y]$ is the $R[Y]$-bimodule generated by $1_Y$, ${ }^X R_{1}[Y]$ is generated, as an $R$-vector space, 
by elements of the form $p'\otimes 1_Y \otimes q'$ with $p'=(c_1,\ldots,c_m)\in Ch^{=0}(X)$ and $q'=(d_1,\ldots,d_n)\in Ch^{=0}(X)$ since $R[X]$ is generated, as an $R$-vector space, by the elements of $Ch^{=0}(X)$. 

We first show that, necessarily, $d^1(c_m)\in Y_0$ (resp. $d^0(d_1)\in Y_0$). Suppose otherwise, suppose e.g. $p'$ has $t=d^1(p')\in X_0 \backslash Y_0$. Then $p'\otimes 1_Y \otimes q'=(p'\times e_{t})\otimes 1_Y \otimes q'=p'\otimes (e_t \bullet 1_Y) \otimes q'$, by Equation (\ref{eq:restrictionscalars}). Now, $e_t \bullet 1_Y=e_t \times 1_Y$ since $R_1[X]$ is the algebra $R[X]$ considered as a module over itself, and $e_t \times 1_Y=\sum\limits_{y \in Y_0} e_t \times e_y=0$ since the $(e_x)_{x\in X_0}$ forms an orthogonal family of idempotents in $R[X]$, hence $p'\otimes 1_Y \otimes q'=0$. 

Because $(X,Y)$ is a relative pair, and as $d^1(c_m)\in Y_0$, $d^0(d_1)\in Y_0$, $p'$ should decompose as the concatenation of $p=(c_1,\ldots,c_{j-1})$ ($c_1,\ldots,c_{j-1}$ all in $X_1\backslash Y_1$) with $(c_j,\ldots,c_m)$ ($c_j,\ldots,c_m \in Y_1$), and similarly, $q'$ should decompose as the concatenation of $q=(d_1,\ldots,d_k)$ ($d_1,\ldots,d_k \in Y_1$) with $(d_{k+1},\ldots,d_n)$ ($d_{k+1},\ldots,d_n \in X_1\backslash Y_1$).  


By Equation (\ref{eq:restrictionscalars}) of Lemma \ref{lem:extensionscalars}, $p\otimes 1_Y\otimes q$ is in that case equal to $p\otimes c \otimes q$ with $p=(c_1,\ldots,c_{j-1})$, $c=(c_j,\ldots,c_m,d_1,\ldots,d_k)$ and $q=(d_{k+1},\ldots,d_n)$. Hence the decomposition of the elements of ${ }^X R_1[Y]$ of the lemma. 

Finally, this is a canonical form for $y\in { }^X R_1[Y]$, as the generating elements $p\otimes c \otimes q$ are all distinct. Indeed, suppose $p\otimes c \otimes q=p'\otimes c'\otimes q'$. By Lemma \ref{lem:normalform}, there exists $p_1, q_1\in Ch^{=0}(X)$ and $p_2, p'_2, q_2, q'_2\in Ch^{=0}(Y)$ such that: 
\begin{eqnarray*}
p & = & p_1 \times p_2 \\
p' & = & p_1 \times p'_2 \\
q & = & q_2 \times q_1 \\
q' & = & q'_2 \times q_1 \\
p_2 \bullet c \bullet q_2 & = & p'_2 \bullet c' \bullet q'_2
\end{eqnarray*} 
But since $p$, $p'$, $q$ and $q'$ consist entirely of edges in $X\backslash Y$, necessarily, $p_2$, $p'_2$, $q_2$ and $q'_2$ are idempotents, neutral elements for the multiplication on the right for $p_1$, $p'_1$, and on the left for $q_1$, $q'_1$ respectively. These then correspond to $e_u$ and $e_v$ respectively, with $u$ being the end point of $p$ and $p'$, and $u$ being the start point of $q$ and $q'$. Hence $p=p_1=p'$ and $q=q_1=q'$ and $p_2\bullet c \bullet q_2=c$, $p'_2\bullet c' \bullet q'_2$ and $c=c'$. 
\end{proof}

\begin{remark}
\label{rem:newcanonical}
By Remark \ref{rem:moreexplicitly}, all $y \in { }^X R_i[Y]$, for $i\geq 2$ can be uniquely written as:
$$
y = \sum\limits_{i\in I, j\in J} \lambda_{i,j} p_{i,j} \otimes g_j \otimes q_{i,j}
$$
\noindent with $\lambda_{i,j} \in R$, $p_{i,j}=(c^i_1,\ldots,c^i_{m_i})$ in $Ch^{=0}(X)$, $g_{j}=(d^j_1,\ldots,d^j_{n_j})$ in $G_i(Y)$, $q_{i,j}=(e^i_1,\ldots,e^i_{l_i})$ in $Ch^{=0}(X)$ with $d^1(c^i_m)=d^0(d_1)$ and $d^1(d_n)=d^0(e^i_1)$ and the condition that $p_{i',j}=p_{i,j}$ and $q_{i',j}=q_{i,j}$ implies $i=i'$. 

By a suitable application of Equation \ref{eq:restrictionscalars}, we can write for $x\in { }^X R_i[Y]$ a similar canonical form as in Lemma \ref{lem:relsub}: 
$$
y = \sum\limits_{i\in I, j\in J} \lambda_{i,j} p'_{i,j} \otimes c_j \otimes q'_{i,j}
$$
with $c_j\in Ch^{=i-1}(Y)$ and the $p'_{i,j}$ and $q'_{i,j}$ entirely composed of 1-cells in $X\backslash Y$ and the condition that $p'_{i',j}=p'_{i,j}$ and $q'_{i',j}=q'_{i,j}$ implies $i=i'$. This is done similarly as in the proof of Lemma \ref{lem:relsub}, using the fact that $(X,Y)$ is a relative pair. 
\end{remark}


We can now define the relative homology modules for relative pairs of precubical spaces: 

\begin{definition}
\label{def:relative}
Let $(X,Y)$ be a relative pair of precubical sets. We define 
the \look{relative $i$th-homology $R[X]$-bimodule} $HM_{i}[X,Y]$ as the homology of the quotient of $R_i[X]$ with the sub-$R[X]$-bimodule ${ }^X R_i[Y]$ within the category of $R[X]$-bimodules, with boundary operator defined in the quotient as $\partial [x] = [\partial x]$, $[x]$ denoting a class with representative $x$ in $R_i[X]/{ }^X R_i[Y]$.
\end{definition}

The definition is valid as we are going to see. Consider an element $[x]$ of $R_i[X]/{ }^X R_i[Y]$. Suppose $[y]=[x]$ in ${ }^X R_i[Y]$, thus there exists $z \in { }^X R_i[Y]$ with $y=x+z$. Thus $\partial [y]=[\partial y]=[\partial x + \partial z]=[\partial x]+[\partial z]=\partial [x]$, and the class of $\partial([x])$ in $R_i[X]/{ }^X R_i[Y]$ does not depend on the particular representative chosen for $[x]$. 






\subsection{Examples}

Let us now examplify the homology modules we have been defining on a few classical examples coming from concurrency theory \cite{thebook}. We postpone examples on relative homology to Section \ref{sec:exact}, Example \ref{ex:relativehomology}. 

\begin{example}
For the ``2 holes on the anti-diagonal" precubical set, Example \ref{ex:2holes} (considered also in Example \ref{ex:2holesbis}), we quotient the $R[X]$-bimodule $R_1[X]$ by $Im \ \partial$ which is the sub-$R[X]$-bimodule of $R_1[M]$ generated by $ih-kd$ and $fg-eb$, i.e. this is the $R$-vector space generated by
$(ih-kd)fg\in e_9 \bullet R_1[M]\bullet e_1$, 
$(ih-kd)f \in e_9 \bullet R_1[M]\bullet e_4$, 
$(ih-kd)e \in e_9 \bullet R_1[M]\bullet e_2$, 
$(ih-kd)eb \in e_9 \bullet R_1[M]\bullet e_1$, 
$ih-kd \in e_9 \bullet R_1[M]\bullet e_5$, 
$fg-eb \in e_5 \bullet R_1[M]\bullet e_1$, 
$h(fg-eb) \in e_8 \bullet R_1[M]\bullet e_1$, 
$ih(fg-eb) \in e_9 \bullet R_1[M]\bullet e_1$,
$d(fg-eb) \in e_6 \bullet R_1[M]\bullet e_1$, 
$kd(fg-eb) \in e_9 \bullet R_1[M]\bullet e_1$. 

Indeed, within $e_9 \bullet R_1[M]\bullet e_1$, 
$x_1=(ih-kd)fg$, $x_2=ih(fg-eb)$, $x_3=kd(fg-eb)$, and $x_4=(ih-kd)eb$ are $R$-linearly dependent: $x_1-x_4=x_2-x_3$ and there are no other (independent) dependencies. This means that we have to quotient (using our notations, see Remark \ref{rem:notation}) entry $(9,1)$ of the matrix of $R$-vector spaces of Example \ref{ex:2holes} by $R^3$, leading to entry $R^3$ in position $(9,1)$ below, which represents the $R[X]$-bimodule, and the other modified entries are $(9,4)$, $(9,2)$, $(9,5)$, $(5,1)$, $(8,1)$, and $(6,1)$:

$$
\begin{pmatrix}
R & 0 & 0 & 0 & 0 & 0 & 0 & 0 & 0 \\
R & R & 0 & 0 & 0 & 0 & 0 & 0 & 0 \\
R & R & R & 0 & 0 & 0 & 0 & 0 & 0 \\
R & 0 & 0 & R & 0 & 0 & 0 & 0 & 0 \\
R & R & 0 & R & R & 0 & 0 & 0 & 0 \\
R^2 & R^2 & R & R & R & R & 0 & 0 & 0 \\
R & 0 & 0 & R & 0 & 0 & R & 0 & 0 \\
R^2 & R & 0 & R^2 & R & 0 & R & R & 0 \\
R^3 & R^2 & R & R^2 & R & R & R & R & R 
\end{pmatrix}
$$
Let us exemplify the $R[X]$-bimodule action now . The entry $(9,1)$ is $R^3$, generated, as an $R$-vector space, by $[ihfg]_{9,1}$, $[ijlg]_{9,1}$ and $[kcab]_{9,1}$. Consider entry $(9,5)$, which is $R$. It is generated as an $R$-vector space by $[kd]_{9,5}$. The right action of $fg$ is $[kdfg]_{9,1}=[ihfg]_{9,1}$. 
\end{example}

\begin{example}
    For the ``two holes on the diagonal''  cubical set of Example \ref{ex:2holes}, we quotient the matrix of $R$-vector spaces by $Im \ \partial$ which is the $R[X]$-bimodule generated by $jl-hf$ and $de-ca$, giving respectively the matrix algebras, by the same argument as above: 
$$
\begin{pmatrix}
R & 0 & 0 & 0 & 0 & 0 & 0 & 0 & 0 \\
R & R & 0 & 0 & 0 & 0 & 0 & 0 & 0 \\
R & R & R & 0 & 0 & 0 & 0 & 0 & 0 \\
R & 0 & 0 & R & 0 & 0 & 0 & 0 & 0 \\
R^2 & R & 0 & R & R & 0 & 0 & 0 & 0 \\
R^2 & R & R & R & R & R & 0 & 0 & 0 \\
R & 0 & 0 & R & 0 & 0 & R & 0 & 0 \\
R^2 & R & 0 & R & R & 0 & R & R & 0 \\
R^4 & R^2 & R & R^2 & R^2 & R & R & R & R 
\end{pmatrix}
$$
which is indeed a non-isomorphic $R[X]$-bimodule to the one obtained for $HM_1$ for Example \ref{ex:2holes}. 
\end{example}

\begin{example}[Empty cube]
\label{ex:emptycube}
We consider here the boundary of the 3-cube of Example \ref{ex:3cube}: 


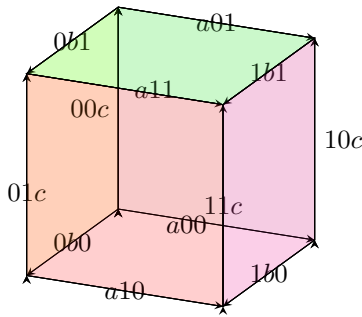
\begin{center}
\begin{tikzpicture}[scale=3.5,tdplot_main_coords]
    \coordinate (O) at (0,0,0);
    \tdplotsetcoord{P}{1.414213}{54.68636}{45}
    \draw[fill=green,fill opacity=0.1] (Pz) -- (Pyz) -- (P) -- (Pxz) -- cycle;
\draw [-stealth] (O) -- (Pz);
    \draw[-stealth] (Pz) edge node {$a01$} (Pyz); 
    \draw[-stealth] (Pyz) edge node {$1b1$} (P); 
    \draw[-stealth] (Pz) edge node {$0b1$} (Pxz);
    \draw[-stealth] (Pxz) edge node[above,right] {$a11$} (P);

   \draw[fill=red,fill opacity=0.1] (Px) -- (Pxy) -- (P) -- (Pxz) -- cycle;
    \draw[-stealth] (Px) edge node {$a10$} (Pxy); 
    \draw[-stealth] (Pxy) edge node {$11c$} (P); 
    \draw[-stealth] (Px) edge node[below] {$01c$} (Pxz);
    
    \draw[fill=magenta,fill opacity=0.1] (Py) -- (Pxy) -- (P) -- (Pyz) -- cycle;
    \draw[-stealth] (Py) edge node {$1b0$} (Pxy); 
    \draw[-stealth] (O) edge node {$0b0$} (Px);
    \draw[-stealth] (Py) edge node[right] {$10c$} (Pyz);
    
    \draw[fill=gray!50,fill opacity=0.1] (O) -- (Py) -- (Pyz) -- (Pz) -- cycle;
    \draw[-stealth] (O) edge node[left] {$00c$} (Pz); 
        \draw[-stealth] (O) edge node[left] {$a00$} (Py); 

    \draw[fill=yellow,fill opacity=0.1] (O) -- (Px) -- (Pxz) -- (Pz) -- cycle;

    \draw[fill=green,fill opacity=0.1] (Pz) -- (Pyz) -- (P) -- (Pxz) -- cycle;

    \draw[fill=red,fill opacity=0.1] (Px) -- (Pxy) -- (P) -- (Pxz) -- cycle;

    \draw[fill=magenta,fill opacity=0.1] (Py) -- (Pxy) -- (P) -- (Pyz) -- cycle;
  \end{tikzpicture}
\end{center}

 By Lemma \ref{lem:acyclicquiver}, $R[X]$ is the matrix algebra: 
 $$
\left(\begin{array}{l|cccccccc}
& 111 & 110 & 101 & 100 & 011 & 010 & 001 & 000 \\
\hline
111 & R & 0 & 0 & 0 & 0 & 0 & 0 & 0 \\
110 & R & R & 0 & 0 & 0 & 0 & 0 & 0  \\
101 & R & 0 & R & 0 & 0 & 0 & 0 & 0  \\
100 & R^2 & R & R & R & 0 & 0 & 0 & 0  \\
011 & R & 0 & 0 & 0 & R & 0 & 0 & 0 \\
010 & R^2 & R & 0 & 0 & R & R & 0 & 0 \\
001 & R^2 & 0 & R & 0 & R & 0 & R & 0 \\
000 & R^6 & R^2 & R^2 & R & R^2 & R & R & R \\
\end{array}\right)
$$
\noindent and has the same matrix representation as a $R[X]$-bimodule as of Remark \ref{rem:notation}. 
We have seen in Example \ref{ex:3cube} that the 1-cube chains are generated by the following elements, plus sub-1-cube chains (made of just one 2-cell): 
\begin{center}
\begin{tabular}{cccccc}
  \begin{tikzpicture}[scale=1.4,tdplot_main_coords]
    \coordinate (O) at (0,0,0);
    \tdplotsetcoord{P}{1.414213}{54.68636}{45}
    \draw[->,thick,fill=green,fill opacity=0.3] (Pz) -- (Pyz) -- (P) -- (Pxz) -- cycle;
\draw [->,thick] (O) -- (Pz);
    \draw[->,thick,fill=green,fill opacity=0.3] (Pz) -- (Pyz); 
    \draw[->,thick,fill=green,fill opacity=0.3] (Pyz) -- (P); 
    \draw[->,thick,fill=green,fill opacity=0.3] (Pz) -- (Pxz);
    \draw[->,thick,fill=green,fill opacity=0.3] (Pxz) -- (P);

    \draw[dashed,fill=gray!50,fill opacity=0.1] (O) -- (Py) -- (Pyz) -- (Pz) -- cycle;
    \draw[dashed,fill=yellow,fill opacity=0.1] (O) -- (Px) -- (Pxz) -- (Pz) -- cycle;
    \draw[dashed,fill=green,fill opacity=0.1] (Pz) -- (Pyz) -- (P) -- (Pxz) -- cycle;
    \draw[dashed,fill=red,fill opacity=0.1] (Px) -- (Pxy) -- (P) -- (Pxz) -- cycle;
    \draw[dashed,fill=magenta,fill opacity=0.1] (Py) -- (Pxy) -- (P) -- (Pyz) -- cycle;
  \end{tikzpicture}
&
  \begin{tikzpicture}[scale=1.4,tdplot_main_coords]
    \coordinate (O) at (0,0,0);
    \tdplotsetcoord{P}{1.414213}{54.68636}{45}

    \draw[fill=gray!50,fill opacity=0.3] (O) -- (Py) -- (Pyz) -- (Pz) -- cycle;
    \draw[->,thick,fill=gray!50,fill opacity=0.3] (O) -- (Py);
    \draw[->,thick,fill=gray!50,fill opacity=0.3] (Py) -- (Pyz);
    \draw[->,thick,fill=gray!50,fill opacity=0.3] (O) -- (Pz);
    \draw[->,thick,fill=gray!50,fill opacity=0.3] (Pz) -- (Pyz);
    \draw[->,thick,fill=gray!50,fill opacity=0.3] (Pyz) -- (P);

    \draw[dashed,fill=gray!50,fill opacity=0.1] (O) -- (Py) -- (Pyz) -- (Pz) -- cycle;
    \draw[dashed,fill=yellow,fill opacity=0.1] (O) -- (Px) -- (Pxz) -- (Pz) -- cycle;
    \draw[dashed,fill=green,fill opacity=0.1] (Pz) -- (Pyz) -- (P) -- (Pxz) -- cycle;
    \draw[dashed,fill=red,fill opacity=0.1] (Px) -- (Pxy) -- (P) -- (Pxz) -- cycle;
    \draw[dashed,fill=magenta,fill opacity=0.1] (Py) -- (Pxy) -- (P) -- (Pyz) -- cycle;
\end{tikzpicture}
&
  \begin{tikzpicture}[scale=1.4,tdplot_main_coords]
    \coordinate (O) at (0,0,0);
    \tdplotsetcoord{P}{1.414213}{54.68636}{45}

    \draw[fill=yellow,fill opacity=0.3] (O) -- (Px) -- (Pxz) -- (Pz) -- cycle;
    \draw[->,thick,fill=yellow,fill opacity=0.3] (O) -- (Px);
    \draw[->,thick,fill=yellow,fill opacity=0.3] (Px) -- (Pxz);
    \draw[->,thick,fill=yellow,fill opacity=0.3] (Pz) -- (Pxz);
    \draw[->,thick,fill=yellow,fill opacity=0.3] (O) -- (Pz);
    \draw[->,thick,fill=yellow,fill opacity=0.3] (Pxz) -- (P);

    \draw[dashed,fill=gray!50,fill opacity=0.1] (O) -- (Py) -- (Pyz) -- (Pz) -- cycle;
    \draw[dashed,fill=yellow,fill opacity=0.1] (O) -- (Px) -- (Pxz) -- (Pz) -- cycle;
    \draw[dashed,fill=green,fill opacity=0.1] (Pz) -- (Pyz) -- (P) -- (Pxz) -- cycle;
    \draw[dashed,fill=red,fill opacity=0.1] (Px) -- (Pxy) -- (P) -- (Pxz) -- cycle;
    \draw[dashed,fill=magenta,fill opacity=0.1] (Py) -- (Pxy) -- (P) -- (Pyz) -- cycle;
\end{tikzpicture}
& 
  \begin{tikzpicture}[scale=1.4,tdplot_main_coords]
    \coordinate (O) at (0,0,0);
    \tdplotsetcoord{P}{1.414213}{54.68636}{45}

    \draw[fill=red,fill opacity=0.3] (Px) -- (Pxy) -- (P) -- (Pxz) -- cycle;
    \draw[->,thick,fill=red,fill opacity=0.3] (O) -- (Px);
    \draw[->,thick,fill=red,fill opacity=0.3] (Px) -- (Pxy);
    \draw[->,thick,fill=red,fill opacity=0.3] (Pxy) -- (P);
    \draw[->,thick,fill=red,fill opacity=0.3] (Pxz) -- (P);
    \draw[->,thick,fill=red,fill opacity=0.3] (Px) -- (Pxz);

    \draw[dashed,fill=gray!50,fill opacity=0.1] (O) -- (Py) -- (Pyz) -- (Pz) -- cycle;
    \draw[dashed,fill=yellow,fill opacity=0.1] (O) -- (Px) -- (Pxz) -- (Pz) -- cycle;
    \draw[dashed,fill=green,fill opacity=0.1] (Pz) -- (Pyz) -- (P) -- (Pxz) -- cycle;
    \draw[dashed,fill=red,fill opacity=0.1] (Px) -- (Pxy) -- (P) -- (Pxz) -- cycle;
    \draw[dashed,fill=magenta,fill opacity=0.1] (Py) -- (Pxy) -- (P) -- (Pyz) -- cycle;
\end{tikzpicture}
&
  \begin{tikzpicture}[scale=1.4,tdplot_main_coords]
    \coordinate (O) at (0,0,0);
    \tdplotsetcoord{P}{1.414213}{54.68636}{45}

    \draw[fill=magenta,fill opacity=0.3] (Py) -- (Pxy) -- (P) -- (Pyz) -- cycle;
    \draw[->,thick,fill=magenta,fill opacity=0.3] (Py) -- (Pxy);
    \draw[->,thick,fill=magenta,fill opacity=0.3] (Pxy) -- (P);
    \draw[->,thick,fill=magenta,fill opacity=0.3] (Py) -- (Pyz);
    \draw[->,thick,fill=magenta,fill opacity=0.3] (Pyz) -- (P);
    \draw[->,thick,fill=magenta,fill opacity=0.3] (O) -- (Py);

    \draw[dashed,fill=gray!50,fill opacity=0.1] (O) -- (Py) -- (Pyz) -- (Pz) -- cycle;
    \draw[dashed,fill=yellow,fill opacity=0.1] (O) -- (Px) -- (Pxz) -- (Pz) -- cycle;
    \draw[dashed,fill=green,fill opacity=0.1] (Pz) -- (Pyz) -- (P) -- (Pxz) -- cycle;
    \draw[dashed,fill=red,fill opacity=0.1] (Px) -- (Pxy) -- (P) -- (Pxz) -- cycle;
    \draw[dashed,fill=magenta,fill opacity=0.1] (Py) -- (Pxy) -- (P) -- (Pyz) -- cycle;
\end{tikzpicture} 
&
  \begin{tikzpicture}[scale=1.4,tdplot_main_coords]
    \coordinate (O) at (0,0,0);
    \tdplotsetcoord{P}{1.414213}{54.68636}{45}

    \draw[fill=magenta,fill opacity=0.3] (O) -- (Py) -- (Pxy) -- (Px) -- cycle;
    \draw[->,thick,fill=magenta,fill opacity=0.3] (Pxy) -- (P);
    \draw[->,thick,fill=magenta,fill opacity=0.3] (Px) -- (Pxy);
    \draw[->,thick,fill=magenta,fill opacity=0.3] (Py) -- (Pxy);
    \draw[->,thick,fill=magenta,fill opacity=0.3] (O) -- (Px);
    \draw[->,thick,fill=magenta,fill opacity=0.3] (O) -- (Py);

    \draw[dashed,fill=gray!50,fill opacity=0.1] (O) -- (Py) -- (Pyz) -- (Pz) -- cycle;
    \draw[dashed,fill=yellow,fill opacity=0.1] (O) -- (Px) -- (Pxz) -- (Pz) -- cycle;
    \draw[dashed,fill=green,fill opacity=0.1] (Pz) -- (Pyz) -- (P) -- (Pxz) -- cycle;
    \draw[dashed,fill=red,fill opacity=0.1] (Px) -- (Pxy) -- (P) -- (Pxz) -- cycle;
    \draw[dashed,fill=magenta,fill opacity=0.1] (Py) -- (Pxy) -- (P) -- (Pyz) -- cycle;
    \draw[dashed,fill=purple,fill opacity=0.1] (O) -- (Py) -- (Pxy) -- (Px) -- cycle;
\end{tikzpicture} 
\\
(00c,A') & (B,1b1) & (C,a11) & (0b0,B') & (a00,C') & (A,11c) \\
\end{tabular}
\end{center}

Therefore, $R_2[X]$ is: 
$$
\left(\begin{array}{l|cccccccc}
& 111 & 110 & 101 & 100 & 011 & 010 & 001 & 000 \\
\hline
111 & 0 & 0 & 0 & 0 & 0 & 0 & 0 & 0 \\
110 & 0 & 0 & 0 & 0 & 0 & 0 & 0 & 0  \\
101 & 0 & 0 & 0 & 0 & 0 & 0 & 0 & 0  \\
100 & R & 0 & 0 & 0 & 0 & 0 & 0 & 0  \\
011 & 0 & 0 & 0 & 0 & 0 & 0 & 0 & 0 \\
010 & R & 0 & 0 & 0 & 0 & 0 & 0 & 0 \\
001 & R & 0 & 0 & 0 & 0 & 0 & 0 & 0 \\
000 & R^6 & R & R & 0 & R & 0 & 0 & 0 \\
\end{array}\right)
$$

Finally, $H_1[X]=R_1[X]/Im \ \partial_{R_2[X] \rightarrow R_1[X]}$ hence: 
 $$
\left(\begin{array}{lcccccccc}
& 111 & 110 & 101 & 100 & 011 & 010 & 001 & 000 \\
\hline
111 & R & 0 & 0 & 0 & 0 & 0 & 0 & 0 \\
110 & R & R & 0 & 0 & 0 & 0 & 0 & 0  \\
101 & R & 0 & R & 0 & 0 & 0 & 0 & 0  \\
100 & R & R & R & R & 0 & 0 & 0 & 0  \\
011 & R & 0 & 0 & 0 & R & 0 & 0 & 0 \\
010 & R & R & 0 & 0 & R & R & 0 & 0 \\
001 & R & 0 & R & 0 & R & 0 & R & 0 \\
000 & R & R & R & R & R & R & R & R \\
\end{array}\right)
$$
Now $Ker \ \partial_{\mid R_2[X]\rightarrow R_1[X]}$ is easily seen to be generated by $(00c,A')-(C,a11)-(0b0,B')-(A,11c)+(a00,C')+(B,1b1)$ which is the generator of the $R$-vector space $e_{000}\bullet R_2[X] \bullet e_{111}$, 
thanks to the calculation of $\partial$ on $R_2[X]$ in Example \ref{ex:3cube}. 
As $R_3[X]=0$, $HM_2[X]$ is:
$$
\left(\begin{array}{lcccccccc}
& 111 & 110 & 101 & 100 & 011 & 010 & 001 & 000 \\
\hline
111 & 0 & 0 & 0 & 0 & 0 & 0 & 0 & 0 \\
110 & 0 & 0 & 0 & 0 & 0 & 0 & 0 & 0  \\
101 & 0 & 0 & 0 & 0 & 0 & 0 & 0 & 0  \\
100 & 0 & 0 & 0 & 0 & 0 & 0 & 0 & 0  \\
011 & 0 & 0 & 0 & 0 & 0 & 0 & 0 & 0 \\
010 & 0 & 0 & 0 & 0 & 0 & 0 & 0 & 0 \\
001 & 0 & 0 & 0 & 0 & 0 & 0 & 0 & 0 \\
000 & R & 0 & 0 & 0 & 0 & 0 & 0 & 0 \\
\end{array}\right)
$$

In the case of the full $3$-cube, $R_3[X]$ is the following $R[X]$-bimodule: 
$$
\left(\begin{array}{lcccccccc}
& 111 & 110 & 101 & 100 & 011 & 010 & 001 & 000 \\
\hline
111 & 0 & 0 & 0 & 0 & 0 & 0 & 0 & 0 \\
110 & 0 & 0 & 0 & 0 & 0 & 0 & 0 & 0  \\
101 & 0 & 0 & 0 & 0 & 0 & 0 & 0 & 0  \\
100 & 0 & 0 & 0 & 0 & 0 & 0 & 0 & 0  \\
011 & 0 & 0 & 0 & 0 & 0 & 0 & 0 & 0 \\
010 & 0 & 0 & 0 & 0 & 0 & 0 & 0 & 0 \\
001 & 0 & 0 & 0 & 0 & 0 & 0 & 0 & 0 \\
000 & R & 0 & 0 & 0 & 0 & 0 & 0 & 0 \\
\end{array}\right)
$$
\noindent with, as only generator (in entry ($000$,$111$)) the 3-cell $S$, with boundary 
$(A,11c)
- (B,1b1)
+(C,a11)
-(a00,C')
+(0b0,B')
-(00c,A')$ as computed in Example \ref{ex:3cube} and $HM_2[X]=0$.
\end{example}

\begin{example}[Matchbox example]
Let us now consider Fahrenberg's matchbox 
example \cite{fahdihom} which is the empty cube of Example \ref{ex:emptycube} minus the lower face $A$. 
The path algebra $R[X]$ is the same as for the empty cube, but $R_2[X]$ is slightly different, in that it does not include, as generators (as an $R$-vector space) the 1-cube path $(A,11c)$ from $000$ to $111$ and the 1-cube path $(A)$ from 
$000$ to $110$. Hence $R_2[X]$ is the $R[X]$-bimodule: 
$$
\left(\begin{array}{l|cccccccc}
& 111 & 110 & 101 & 100 & 011 & 010 & 001 & 000 \\
\hline
111 & 0 & 0 & 0 & 0 & 0 & 0 & 0 & 0 \\
110 & 0 & 0 & 0 & 0 & 0 & 0 & 0 & 0  \\
101 & 0 & 0 & 0 & 0 & 0 & 0 & 0 & 0  \\
100 & R & 0 & 0 & 0 & 0 & 0 & 0 & 0  \\
011 & 0 & 0 & 0 & 0 & 0 & 0 & 0 & 0 \\
010 & R & 0 & 0 & 0 & 0 & 0 & 0 & 0 \\
001 & R & 0 & 0 & 0 & 0 & 0 & 0 & 0 \\
000 & R^5 & 0 & R & 0 & R & 0 & 0 & 0 \\
\end{array}\right)
$$

Finally, $H_1[X]=R_1[X]/Im \ \partial_{R_2[X] \rightarrow R_1[X]}$ hence: 
 $$
\left(\begin{array}{lcccccccc}
& 111 & 110 & 101 & 100 & 011 & 010 & 001 & 000 \\
\hline
111 & R & 0 & 0 & 0 & 0 & 0 & 0 & 0 \\
110 & R & R & 0 & 0 & 0 & 0 & 0 & 0  \\
101 & R & 0 & R & 0 & 0 & 0 & 0 & 0  \\
100 & R & R & R & R & 0 & 0 & 0 & 0  \\
011 & R & 0 & 0 & 0 & R & 0 & 0 & 0 \\
010 & R & R & 0 & 0 & R & R & 0 & 0 \\
001 & R & 0 & R & 0 & R & 0 & R & 0 \\
000 & R & R^2 & R & R & R & R & R & R \\
\end{array}\right)
$$

Finally, $R_3[X]=0$ and $Ker \ \partial_{\mid  R_2[X] \rightarrow R_1[X]}=0$ so $HM_2[X]=0$.
\end{example}

\section{Homology modules and persistence modules}

\label{sec:persistencemod}
\label{sec:representations}

The objective of this section is to understand and represent effectively the previous homological constructions in the category of bimodules over the associative algebra of directed paths. Due to classical Morita equivalence between quivers and path algebras, our homology modules $HM$ can be considered as representations of a particular quiver, 
showing a link between our homology algebras and persistence over the underlying quiver of a directed space (the ``space of parameters''). We review below the elements of this equivalence.





\paragraph{Quiver representations}

\begin{definition}[\cite{assocalg}]
Let $Q$ be a finite quiver. An $R$-linear representation or, more briefly, a \look{representation} $M$ of 
$Q$ is defined by the following data:
\begin{itemize}
\item To each point $a$ in $Q_0$ is associated an $R$-vector space $M_a$.
\item To each arrow $\alpha: \ a \rightarrow b \in Q_1$ is associated an $R$-linear map $\phi_\alpha : \ M_a \rightarrow M_b$.
\end{itemize}
Such a representation is denoted as $M = (M_a , \phi_\alpha)_{a \in Q_0, \alpha \in Q_1}$ , or simply $M = (M_a, \phi_\alpha)$. It is called finite dimensional if each vector space $M_a$ is finite dimensional.
\end{definition}

\begin{definition}[\cite{assocalg}]
Let $M = (M_a,\phi_\alpha)$ and $M' = (M_a',\phi_\alpha')$ be two representations of $Q$. A \look{morphism of representations} $f : \ M \rightarrow M'$ is a family $f = (f_a)_{a \in Q_0}$ of $R$-linear maps $(f_a : M_a \rightarrow M_a')_{a\in Q_0}$ that are compatible with the structure maps $\phi_\alpha$, that is, for each arrow $\alpha : \ a \rightarrow b$, we have $\phi_\alpha' f_a = f_b \phi_\alpha$ or, equivalently, the following square is commutative:
$$
\begin{tikzcd}
  M_a \arrow[r,"\phi_\alpha"] \arrow[d,"f_a"]
    & M_b \arrow[d,"f_b"] \\
  {M'}_a \arrow[r,"{\phi_\alpha}'"]
& {M'}_b \\
\end{tikzcd}
$$
\end{definition}

We have thus defined a category $Rep(Q)$ of $R$-linear representations of $Q$. We denote by $rep(Q)$ the full subcategory of $Rep(Q)$ consisting of the finite dimensional representations.

Let $(M,(\phi)_\alpha)$ be a representation of some quiver $Q$. Consider a path $p=(p_1,\ldots,p_l)$ in $Q$, from $a$ to $b$. The evaluation, see \cite{assocalg}, $eval_{(M,(\phi_\alpha))}(p)$ of $p$ in representation $(M,(\phi)_\alpha)$ is the $R$-linear map from $M_a$ to $M_b$ which is $\phi_{p_l}\circ \ldots \phi_{p_1}$. 

\begin{definition} (\cite{assocalg})
\look{A representation of a bound quiver} $(Q,I)$ (where $I$ is an admissible ideal of $Q$, see Definition \ref{def:boundquiver})
is a representation $(M,(\phi)_\alpha)$ of $Q$ for which, given any element $\sum\limits_{i=1}^k \alpha_i p_i \in I$, $\alpha_i\in R$, and $p_i$ being paths in $Q$, 
$\sum\limits_{i=1}^k \alpha_i eval_{(M,(\phi_\alpha))}(p_i)=0$.
\end{definition}


As is well known \cite{assocalg}, when the algebra $A$ is $R[Q]/I$ for some finite and connected quiver $Q$, and $I$ an admissible ideal of $R[Q]$,  
the category $mod \ A$ of (left) $A$-modules is equivalent to the category $rep_R (Q, I)$. 


This can easily be generalized to $A$-bimodules as follows: 


\begin{definition}
\label{lem:frombimodtorep}
Let $A=R[Q]/I$ be an algebra, with $Q$ a finite connected quiver and $I$ an admissible ideal of $R[Q]$. Construct the {graph $\look{FQ}$} as follows: 
\begin{itemize}
    \item vertices of $FQ$ are pairs of vertices $(x,y)$ of $Q$ such that there exists a path from $x$ to $y$ in $Q$
    \item arrows from $(x,y)$ to $(x',y')$ in $FQ$ are pairs of arrows $(u,v)$ in $Q$ where $u$ goes from $x'$ to $x$ and $v$ goes from $y$ to $y'$
\end{itemize}

Then given $M$ an $A$-bimodule, we construct a representation $F(M)$ of $FQ$ bound by $I$ as follows: 
\begin{itemize}
    \item to each vertex $(a,b) \in FQ$, define the $R$-vector space $M_{a,b}$ to be $(e_a+I) \bullet M\bullet (e_b+I)$,
\item and define $\phi^M_{u,v}: \ M_{a,b} \rightarrow M_{a',b'}$ for $u$ an arrow from $a'$ to $a$ and $v$ an arrow from $b$ to $b'$ in $Q$, to be the map which associates to each $x \in M_{a,b}$, $u \bullet x \bullet v=e_{a'} u \bullet x \bullet v e_{b'}=e_{a'}\bullet (u\bullet x \bullet v) \bullet e_{b'} \in M_b \in M_{a',b'}$.
\end{itemize}
\end{definition}



This 
implies the following equivalence of categories: 


\begin{lemma}
\label{lem:equivrep}
Let $A=R[Q]/I$, where $Q$ is a finite connected quiver and $I$ an admissible ideal of $R[Q]$. Then there exists an equivalence of categories $F: \ {}_R mod_R \ A \rightarrow rep_R(FQ,I)$
\end{lemma}


This is a direct consequence of Theorem 1.6 in Chapter 3 of \cite{assocalg} and of the fact that $A$-bimodules are  $A\otimes A^{op}$-left modules. We are a bit more explicit how the equivalence works on objects below: 

For $M$ an $A$-bimodule, define $F(M)$ as in Definition \ref{lem:frombimodtorep}. Now, consider $f: \ M \rightarrow N$ a morphism of $A$-bimodules and $(a,b)$ a vertex in $FQ$. Let $x \in M_{a,b}$, which therefore has the form $(e_a+i)\bullet m \bullet (e_b+j)$ where $m\in M$, $i \in I$ and $j \in I$. Define $F(f)_{a,b}(x)$ to be $(e_a+i)\bullet f(m) \bullet (e_b+j) \in N_{a,b}$. We have to check now that for all arrows $u$ from $a'$ to $a$ and $v$, arrow from $b$ to $b'$ in $Q$, $\phi^N_{u,v} F(f)_{a,b}=F(f)_{a',b'} \phi^M_{u,v}$. We compute, for $x \in M_{a,b}$ of the form $x=(e_a+i)\bullet m \bullet (e_b+j)$ for some $m\in M$, $i\in I$ and $j\in I$: 
$$
\begin{array}{rcl}
\phi^N_{u,v} F(f)_{a,b}(x) & = & \phi^N_{u,v}((e_a+i)\bullet f(m) \bullet (e_b+j))\\
& = & u\bullet ((e_a+i) \bullet f(m) \bullet (e_b+j)) \bullet v \\
& = & u(e_a+i) \bullet f(m) \bullet (e_b+j) v \\
& = & u \bullet f(m) \bullet v 
\end{array}
$$
\noindent since $u$ (resp. $v$) is an arrow from $a'$ to $a$ (resp. from $b$ to $b'$) in $Q$, making $u(e_a+i)=u$ (resp. $(e_b+j)v=v$) in the quiver algebra $R[Q]/I$ by $I$, whereas: 
$$
\begin{array}{rcl}
F(f)_{a',b'} \phi^M_{u,v}(x) & = & F(f)_{a',b'}(u\bullet x \bullet v)\\
& = & F(f)_{a',b'}(e_{a'}u\bullet m \bullet v e_{b'}) \\
& = & e_{a'} f(u\bullet m \bullet v) e_{b'}\\
& = & e_{a'} u \bullet f(m) \bullet v e_{b'} \\
& = & u \bullet f(m) \bullet v \\
& = & \phi^N_{u,v} F(f)_{a,b}(x)
\end{array}
$$
Now, define the following transform $G: \ rep_R(FQ,I) \rightarrow {}_R mod_R \ A$. 

Let $(M_{a,b},\phi_{u,v})$ be a representation of $FQ$. Define $G(M_{a,b},\phi_{u,v})$ to be the $A$-bimodule which is, as an $R$-vector space, $\coprod\limits_{(a,b)\in FQ} M_{a,b}$, with the following left and right action of elements $a$ and $b$ in $A$. 

As $A=R[Q]/I$, $a$ (resp. $b$) is of the form $u_1 u_2\ldots u_m+i$ (resp. $v_1 v_2\ldots v_n+j$) where $i \in I$ (resp. $j \in I$) and $u_k$ (resp. $v_l$) is an arrow in $Q$ from $x_k$ to $x_{k+1}$ (resp. from $y_l$ to $y_{l+1}$), with $x_1=a'$ and $x_{m+1}=a$ (resp. $y_1=b$ and $y_{n+1}=b'$).

Define the action $\bullet$ of $u$ and $v$ on $m \in M_{a,b}$
$$\begin{array}{rcl}
u\bullet m & = & \phi_{u_1,e_b} \phi_{u_2,e_b} \ldots \phi_{u_m,e_b} (m) \\
m \bullet v & = & \phi_{e_a,v_1} \phi_{e_a,v_2} \ldots
\phi_{e_a,v_n}(m)
\end{array}
$$
\noindent hence $u \bullet m \bullet v = eval_{(M_{a,b},(\phi_{u,v}))}(u,v)$. 

Now we see that, for any $A$-bimodule $M$, $G(F(M))$ is the $R$-vector space $$\coprod\limits_{(a,b)\in FQ} (e_a+I)\bullet M \bullet (e_b+I)$$ 
\noindent which is isomorphic to $M$, as $e_a+I$, $e_b+I$, for $a \in Q$ and $b\in Q$ form a complete set of idempotents of $A$.  

Similarly, for any representation $(M_{a,b},\phi^M_{u,v})$ of $FQ$ bound by $I$, $F(G(M_{a,b},$ $\phi_{u,v}))$ is the representation $(N,\phi^N_{u,v})$ with $N_{a,b}$ being $(e_a+I)\bullet \coprod\limits_{(a,b)\in FQ} M_{a,b} \bullet (e_b+I)$ which is isomorphic to $M_{a,b}$.


This is of course akin to the construction of the factorization category (on the category of traces) and of natural systems on this factorization category, see next paragraph. Also, Lemma \ref{lem:equivrep} portrays our homology module over the path algebra, as persistent homology of the total path space $\mid X\mid ^{\I}$ ($X$ is a precubical set) along the decomposition into path spaces from $a$ to $b$, $a$ and $b$ varying in $X_0$. 




\begin{example}
\label{ex:repres}
Consider again the following precubical set $X$ of Example \ref{ex:emptysquare}, on the left below (an empty square):

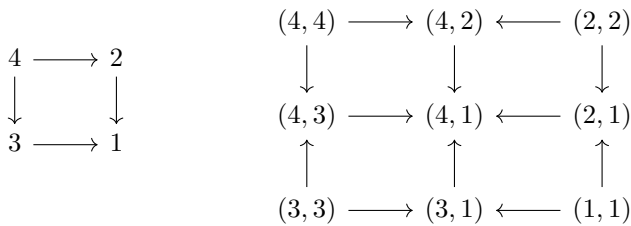
\begin{center}
\begin{minipage}{5cm}
\[\begin{tikzcd}
  4 \arrow[r] \arrow[d]
    & 2 \arrow[d] \\
  3 \arrow[r]
& 1 \end{tikzcd}\]
\end{minipage}
    \begin{minipage}{5cm}
    \[\begin{tikzcd}
  (4,4) \arrow[r] \arrow[d] & (4,2) \arrow[d]  & (2,2) \arrow[l]\arrow[d]\\
  (4,3) \arrow[r] & (4,1) & (2,1) \arrow[l] \\
  (3,3) \arrow[r]\arrow[u] & (3,1)\arrow[u] & (1,1)\arrow[l]\arrow[u]
\end{tikzcd}
\]
\end{minipage}
\end{center}
On the right hand side, we represented the quiver $FX_{\leq 1}$ (which is $FX$ since $X$ is actually a quiver). The representation of $FX$ which corresponds to $HM_1[X]$ is: 
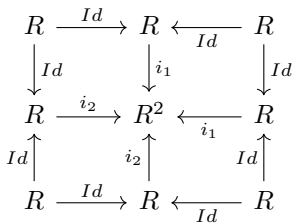
\[\begin{tikzcd}
  R \arrow[r,"Id"] \arrow[d,"Id"] & R \arrow[d,"i_1"]  & R \arrow[l,"Id"]\arrow[d,"Id"]\\
  R \arrow[r,"i_2"] & R^2 & R \arrow[l,"i_1"] \\
  R \arrow[r,"Id"]\arrow[u,"Id"] & R\arrow[u,"i_2"] & R\arrow[l,"Id"]\arrow[u,"Id"]
\end{tikzcd}
\]
\noindent where $i_1$ is the inclusion of $R$ into the first component of $R^2$ and $i_2$ is the inclusion of $R$ into the second component of $R^2$. 
\end{example}

In general, there are infinitely many indecomposable $FQ$-modules, when $FQ$ is not very particular (e.g. of the form of Dynkin diagram $A_n$, as for one-parameter persistence), and in order to give tractable invariants, we need to resort to simpler characterizations, such as rank invariants, that we briefly detail in the next paragraph. Still, we will claim in Appendix \ref{sec:tameness} that our homology modules, for an interesting class of precubical sets, gives rise to tame bimodules of some sort. 


\paragraph{Rank invariants}

Let $M$ be an $A$-bimodule where $A$ is a basic, connected finite dimensional algebra. By Theorem \ref{thm:algfromquiver}, Appendix \ref{sec:B}, $A$ is isomorphic to $R[Q]/I$ where $Q$ is some finite connected quiver (the one of Definition \ref{def:ordquiver}), and by Definition \ref{lem:frombimodtorep}, we get a representation of $FQ$ bound by $I$ $(M_\alpha,\phi_\alpha)$ ($\alpha \in FQ$). We define now a categorical view on representations and modules, that will be useful in this section and in next section, Section \ref{sec:naturalhomology}: 

\begin{definition}
\label{def:categoricalrepresentation} 
The functor $\mathcal{F}_M$ from 
$FQ$, seen as the free category on $FQ$, to the category of $R$-vector spaces, which has 
$\mathcal{F}_M(x,y)=M_{x,y}$ and $\mathcal{F}_M(u,v)=\phi_{u,v}$ where $(x,y)$ is a vertex of $FQ$ and $(u,v)$ is an arrow in $FQ$, is called the \look{categorical presentation} of $M$. 
\end{definition}

This categorical presentation of modules is well-known in category theory and used also in persistence, e.g. in the framework \cite{miller2020modules}. 

Rank invariants or generalized rank invariants have been defined as computable invariants for multi-dimensional, poset or even DAG persistence modules as in e.g. \cite{Kim_2021}. 

Here we will consider only simple rank invariants. Consider all intervals $\mathcal{I}$ within $FQ$. Now, the categorical presentation of $M$, $\mathcal{F}_M$, can be restricted on any of these intervals, to give the functor ${\mathcal{F}_M}_{\mid \mathcal{I}}$, which can be viewed as a diagram in the complete and co-complete category of $R$-vector spaces. 

\begin{remark}
Note that any interval $\mathcal{I}$ in $FQ$ corresponds to the maximal chains of the trace poset of a partially-order space, as defined in \cite{calk2023persistent}. These are the ones that define one-parameter persistence modules within natural homology (that we are going to recap in next Section). 
%
\end{remark}

We can thus consider the canonical map $\Phi_{\mathcal{I}}: \ \lim\limits_{\leftarrow} {\mathcal{F}_M}_{\mid \mathcal{I}} \rightarrow \lim\limits_{\rightarrow} {\mathcal{F}_M}_{\mid \mathcal{I}}$ from the limit of the diagram ${\mathcal{F}_M}_{\mid \mathcal{I}}$ to the colimit of the same diagram. The dimension of $Im \ \Phi_{\mathcal{I}}$, for all $\mathcal{I}$ intervals in $FQ$ defines the rank invariant of $M$. 

\begin{example}
Intervals in $FQ$ for Example \ref{ex:repres} are: $(4,4)$, $(3,3)$, $(2,2)$, $(1,1)$, $(4,4)\leq (4,2)$, $(4,4) \leq (4,3)$, $\ldots$, $(4,4)\leq (4,2),(4,1)$ and correspond to a sequence of extensions within some given maximal path. 

Indeed the rank invariant corresponding to singletons $q=(x,y) \in (FQ)_0$ are just the ranks of the corresponding modules $M_q=e_x \bullet M \bullet y$. The rank of $R \mathop{\rightarrow}\limits^{i_1} R^2$ is indeed the rank of the image of $i_1$, which is 1, and similarly for $i_2: R \rightarrow R^2$. The only interesting case is for maximal intervals (of length 3), e.g. $(4,4)\leq (4,2) \leq (4,1)$. In that case, $\lim\limits_{\leftarrow} {F_M}_{\mid \mathcal{I}}$ is 0 (as this is also the pullback of $i_1$ and $i_2$ from $R$ to $R^2$) and $\lim\limits_{\rightarrow} {F_M}_{\mid \mathcal{I}}$ is $R$ (and the universal co-cone has all identity maps and map $R^2\rightarrow R$ which to $(a,b)$ associates $a+b$). And the rank of the corresponding $\Phi_\mathcal{I}$ is 0. 
\end{example}

\begin{remark}
This should allow us also to define barcodes (a map from the set of intervals of $Q$ to (relative) integers) by a M\"obius inversion formula as in \cite{Kim_2021}, for precubical sets whose geometric realization is a partially-ordered space. In this case, $FQ$ is actually a poset, see \cite{calk2023persistent}. More generally speaking, this should have relations with the approach of \cite{calk2023persistent}. This will be investigated and developed elsewhere, for  practical applications. 
\end{remark}

\section{Homology modules and natural homology} 

\label{sec:naturalhomology}

We begin by showing that the homology module framework we have been setting up measures, at least locally, the homology of the path space of precubical spaces, between two endpoints: 

\begin{lemma}
\label{lem:homtrace}
Let $X$ be a finite precubical set
that is such that is has proper non-looping length covering, $a, b \in X_0$, $R$ a field. Then
the $R$-vector space $e_a \bullet HM_n[X] \bullet e_b$, $n \geq 1$, is the standard $n$th homology of the trace space $\diP{\mid X \mid}^b_a$ from $a$ to $b$. 
\end{lemma}


\begin{proof}
By Lemma \ref{lem:decompquotient}, 
the quotient 
$HM_i[X]=Ker \ \partial_{\mid R_{i}[X]}/Im \ \partial_{\mid R_{i+1}[X]}$ can be identified, as an $R$-vector space, with the coproduct of all $R$-vector spaces 
$$Ker \ \partial_{\mid e_a \bullet R_{i}[X] \bullet e_b}/Im \ \partial_{\mid e_a \bullet R_{i+1}[X] \bullet e_b}$$ 
\noindent with bimodule action being defined by $u\bullet [n]_{c,d} \bullet v=[u\bullet n\bullet v]_{a',b'}$ if $c=a$ and $d=b$, 0 otherwise. Hence
$e_a \bullet HM_i[X] \bullet e_b$ is $Ker \ \partial_{\mid e_a \bullet R_{i}[X] \bullet e_b}/Im \ \partial_{\mid e_a \bullet R_{i+1}[X] \bullet e_b}$ as an $R$-vector space. 


This last quotient is the homology of the chain complex $Ch(X)^b_a$, and by Lemma \ref{lem:Kris}, this is isomorphic to the (singular) homology of the trace space $\diP{\mid X \mid}^b_a$.
\end{proof}

We are now going to make the link between our homology modules and natural homology, at least for a certain class of precubical sets, that has been used in 
\cite{Dubut}:
\begin{definition}[\cite{Dubut}]
\label{def:cubicalcomplex}
A
($d$-dimensional) \look{{cubical complex}} $X$ is a finite set of
\emph{cubes} $(D, \vec x)$, where $D \subseteq \{1, 2, \cdots, d\}$
and $\vec x \in \Int^d$, which is closed under taking past and future
faces (to be defined below). The cardinality of $D$ is the
\emph{dimension} of the cube $(D, \vec x)$.  Let $\vec 1_k$ be the
$d$-tuple whose $k$th component is $1$, all others being $0$.  Each
cube $(D, \vec x)$ is {realized} as the geometric cube $\iota (D,
\vec x) = I_1 \times I_2 \times \cdots \times I_d$ where $I_k = [x_k,
x_k+1]$ if $k \in D$, $I_k = [x_k, x_k]$ otherwise.  

When $card \ D=n$, we write $D [i]$ for the $i$th element of $D$.  For
example, if $D = \{3, 4, 7\}$, then $D[1]=3$, $D[2]=4$, $D[3]=7$.  We
also write $\partial_i D$ for $D$ minus $D[i]$.  Every $n$-dimensional
cube $(D, \vec x)$ has $n$ {past faces} $\partial_i^0 (D, \vec
x)$, defined as $(\partial_i D, x)$, and $n$ {future faces}
$\partial_i^1 (D, \vec x)$, defined as $(\partial_i D, x + \vec
1_{D[i]})$, $1\leq i \leq n$.
\end{definition}

Note that a cubical complex is in particular a precubical set such that it has proper non-looping length covering. 



The comparison between natural homology and homology modules will be made through bisimulation between categorical diagrams \cite{Dubut}, that we particularize to diagrams in the category of $R$-vector spaces $Vect$, and that we recall below:

\begin{definition}[\cite{Dubut}]
\label{def:alternatedefbisim}
A \look{bisimulation} $\mathcal B$ between two diagrams $\map{F}{\C}{Vect}$ and $\map{G}{\D}{Vect}$ is a set of triples $(c,f,d)$ where $c$ is an object of $\C$, $d$ is an object of $\D$ and $\map{f}{F(c)}{G(d)}$ is an isomorphism of $\M$ such that for all $(c,f,d)$ in $R$:
\begin{itemize}
	\item if there exists $\map{i}{c}{c'} \in \C$ then there exists $\map{j}{d}{d'} \in \D$ and $\map{g}{F(c')}{G(d')} \in Vect$ such that $g\circ F(i) = G(j) \circ f$ and $(c',g,d') \in \mathcal{B}$
	\item if there exists $\map{j}{d}{d'}\in \D$ then there exists $\map{i}{c}{c'} \in \C$ and $\map{g}{F(c')}{G(d')} \in Vect$ such that $g\circ F(i) = G(j) \circ f$ and $(c',g,d') \in \mathcal{B}$
\end{itemize}
\begin{center}
    \begin{tikzpicture}[scale=.7]
    \node (x') at (0,0) {$c'$};
    \node (x) at (0,1.5) {$c$};
    \node (Fx') at (1,0) {$Fc'$};
    \node (Fx) at (1,1.5) {$Fc$};
    \node (Gy') at (4,0) {$Gd'$};
    \node (Gy) at (4,1.5) {$Gd$};
    \node (y') at (5,0) {$d'$};
    \node (y) at (5,1.5) {$d$};
    \draw[->] (x) -- (x');
    \draw[->] (y) -- (y');
    \draw[->] (Fx) -- (Fx');
    \draw[->] (Fx) -- (Gy);
    \draw[->] (Fx') -- (Gy');
    \draw[->] (Gy) -- (Gy');
    \node (i) at (-0.2,0.75) {$i$};
    \node (j) at (5.2,0.75) {$j$};
    \node (Fi) at (0.6,0.75) {$Fi$};
    \node (Gj) at (4.4,0.75) {$Gj$};
    \node (eta) at (2.5,1.7) {$f$};
    \node (eta') at (2.5,-0.3) {$g$};
  \end{tikzpicture}
  \end{center}
and such that:
\begin{itemize}
	\item for all $c\in\C$, there exists $d$ and $f$ such that $(c,f,d)\in \mathcal{B}$
	\item for all $d\in\D$, there exists $c$ and $f$ such that $(c,f,s)\in \mathcal{B}$
\end{itemize}
We say that two diagrams $\map{F}{\C}{Vect}$ and $\map{G}{\D}{Vect}$ are \look{bisimulation equivalent} if there exists a bisimulation between them. 
\end{definition}

 
Now we can prove, identifying $HM_i[X]$ with its categorical presentation $\mathcal{F}_{HM_i[X]}$: 

\begin{theorem}
\label{thm:bisimequivnathom}
 Let $X$ be a cubical complex and $Subd(X)$ be its barycentric subdivision. Then $HM_i[Subd(X)]$ is bisimulation equivalent to the natural homology $HN_i[\mid X\mid ]$, for all $i\geq 1$.    
\end{theorem}

%

\begin{proof}
The proof relies on the fact that the combinatorial natural homology of $X$ is bisimulation equivalent to the natural homology of its geometric realization $\mid X\mid $ when $X$ is a cubical complex. And that the homology modules of this paper, $HM$, are bisimulation equivalent to combinatorial natural homology under the same hypotheses, up to barycentric subdivision, as we will see. We first need to recap some definitions from \cite{Dubut}. 

Let $x=(D_x,\vec x)$ and $y=(D_y,\vec y) \in X$. We say that $x$ is a {future boundary} (resp. a {past boundary}) of $y$ if there exist $k \geq 0$ and $i_0$, ..., $i_k$ such that $x = \partial_{i_k}^1\circ\cdots\circ\partial_{i_0}^1(y)$ (resp. $x = \partial_{i_k}^0\circ\cdots\circ\partial_{i_0}^0(y)$). We write $x \preceq y$ when:
\begin{itemize}
	\item either $x$ is a past boundary of $y$
	\item either $y$ is a future boundary of $x$
\end{itemize}

A \emph{discrete trace} from $x$ to $y$ in $X$ is a sequence $c_0$, ..., $c_n$ of cubes in $X$ (with $n\geq 0$) such that $c_0 = x$, $c_n = y$ and for all $i\in\{1, ..., n\}$ $c_{i-1} \preceq c_i$. An extension of a discrete trace $c_0$, ..., $c_n$ is a pair $((d_0, ..., d_m), (e_0, ..., e_p))$ of discrete traces such that $d_m = c_0$ and $c_n = e_0$.
Define the category $\Trace^d_X$ to be:
\begin{itemize}
	\item objects are discrete traces of $X$
	\item morphisms from $c_0$, ..., $c_n$ (discrete trace from $x$ to $y$) to $d_0$, ..., $d_m$ (discrete trace from $x'$ to $y'$) are pairs of discrete traces $((e_0, ..., e_p),(f_0,...,f_k))$ from $(x',y)$ to $(x,y')$ (so $c_0 = e_{p} = x$ and $c_n = f_0 = y$) such that $d_0,...,d_m = e_0,...,e_{p-1},c_0,...,c_n,f_1,...,f_k$.
\end{itemize}

When $X$ is a cubical complex, a discrete trace can be realized as a trace in the geometric realization $\mid X\mid $. 

For $z=(D_z,\vec z)$ in $X$, write $\hat z$ for the point $[z,\star]$ of $\mid X\mid$ where $\star = (\frac{1}{2},\ldots,\frac{1}{2})\in \Real^m$ ($m=card \ D_z$).

For $x$ and $y$ two cubes in $X$, when $x$ is a past boundary of $y$, i.e. $x = \partial_{i_k}^0\circ\cdots\circ\partial_{i_0}^0(y)$, write $\widehat{xy}$ for the path in $\mid X\mid $ from $\hat x$ to $\hat y$ defined by $$\widehat{xy}(t) = [y, t\star_y + (1-t)\rho_{i_0}^0\circ\ldots\circ\rho_{i_k}^0(\star_x)]$$
\noindent where $\rho^0_{i}$ is the map which associates to each $a=(a_1,a_2,\ldots,a_l)\in \Real^l$, 
the element 
$\rho^0_i(a)=(a_1,a_2,\ldots ,a_{i-1}, 0, a_i, \ldots, a_l)$. 

When $y=(D_y,\vec y)$ is a future boundary of $x$, i.e. $y = \partial_{i_k}^1\circ\cdots\circ\partial_{i_0}^1(x)$, write $\widehat{xy}$ for the path in $\mid X\mid $ from $\hat x$ to $\hat y$ defined by $$\widehat{xy}(t) = [x, (1-t)\star_x + t\rho_{i_0}^1\circ\ldots\circ\rho_{i_k}^1(\star_y)]$$
\noindent where $\rho^1_{i}$ is the map which associates to each $a = (a_1,a_2,\ldots,a_l)\in \Real^l$, the element 
$\rho^1_i(a)=(a_1,a_2,\ldots , a_{i-1},1, a_i,\ldots, a_l)$. 

Note that $\widehat{xy}$ is uniquely defined when $X$ is a non-self linked pre-cubical set \cite{algtopandconcur} (which is the case for cubical complexes of Definition \ref{def:cubicalcomplex})  
and is a dipath also in that case. 

Then, a discrete trace $c_0$, ..., $c_n$ can be realized as the concatenation of the traces $\widehat{c_0c_1},\ldots$ with $\widehat{c_{n-1}c_n}$ in $\mid X\mid $, noted $\widehat{c_0, ..., c_n}$ and $\widehat{~~~}$ extends to a functor from $\Trace^d_X$ to $\Trace_{\mid {X}\mid }$. 


For $n\geq 1$, the $i$-th natural homology of $X$ or discrete natural homology $\map{\syshd{i}{X}}{{\cal T}^d_X}{Mod_R}$ as $$\syshd{i}{X} = HN_i({\mid X\mid }) \circ \widehat{~~~}$$

Now, we are going to construct a set $\mathcal{B}$ of triples  $((x,y),f,(c_0,\ldots,c_n))$ where $(c_0,\ldots,c_n)$ is a discrete trace in $\Trace^d_X$ from $x$ to $y$, and $(x,y)\in F(Subd(X))_{\leq 1}$ (see Definition \ref{lem:frombimodtorep}). Indeed, vertices of $Subd(X)$ are in bijection with cubes in $X$, see e.g. \cite{Jardine}. This will form a bisimulation equivalence between $HM_i[Subd(X)]$, seen as its categorical presentation (Definition \ref{def:categoricalrepresentation}), and $\syshd{i}{X}$ as we are going to see.  

First, as a consequence of Lemma \ref{lem:homtrace}, 
$\mathcal{F}_{HM_i[Subd(X)]}(x,y)=e_x \bullet HM_i[Subd(X)] \bullet e_y$, $i \geq 1$, is isomorphic to the standard $(i-1)$th homology of the trace space $\diP{\mid Subd(X) \mid}^y_x$ from $x$ to $y$, which is isomorphic to 
$\diP{\mid X\mid}^y_x$.
But $\syshd{i}{X}(x,y)$ is also isomorphic to the $(i-1)$th homology of $\diP{X}^y_x$, and we set $f$ to be the composition of the two isomorphisms. 

Now, for each morphism from $(c_0,\ldots, c_n)$ (discrete trace from $x$ to $y$) to $(d_0, \cdots, d_m)$ (discrete trace from $x'$ to $y'$), which is a pair of discrete traces $((e_0, ..., e_p),(f_0,...,f_k))$ from $(x',y)$ to $(x,y')$ (so $c_0 = e_{p} = x$ and $c_n = f_0 = y$) such that $(d_0,...,d_m) = (e_0,...,e_{p-1},c_0,...,c_n,f_1,...,f_k)$, consider any morphim 
$(u,v)$ from $(x,y)$ to $(x',y')$ in $F(Subd(X))_{\leq 1}$, which is made of any 0-cube chain $u$ at the boundary of $(c_0,\ldots, c_n)$ (resp. $v$ at the boundary of $(d_0, \cdots, d_m)$). As the geometric realization of $(c_0,\ldots, c_n)$ is dihomotopic to $u$ (resp. of $(d_0, \cdots, d_m)$, to $v$), the pre and postcomposition by $u$ and $v$ is, in homology, equal to the pre and post composition by $(c_0,\ldots, c_n)$ and $(d_0, \cdots, d_m)$, hence generate the same map. 

Conversely, given any edge $(u,v)$ in $F(Subd(X))_{\leq 1}$, this is also a morphism in $\Trace^X_d$ and they generate the same map in homology. Hence $\mathcal{B}$ is a bisimulation between $HM_i[Subd(X)]$ and $\syshd{i}{X}$
\end{proof}

\begin{remark}
This, and what we saw in the previous section, should be linked to \cite{calk2023persistent}, which constructs natural homology as a certain colimit of persistence over directed paths of a directed space. Indeed, for $X$ a finite precubical set with underlying graph which is a DAG, $\mid X\mid $ is a partially-ordered space, for which the construction of \cite{calk2023persistent} applies. It is shown there that the natural homology of $\mid X\mid $ is the colimit of the one-parameter persistent homologies based on the restricted categorical representations  ${\mathcal{F}_{HM_i[X]}}_{\mid \mathcal{I}}$, over all interval inclusions of $\mathcal{I}$ within $FX_{\leq 1}$. 
\end{remark}

An obvious goal now is whether we have a counterpart of module homologies that we defined on combinatorial structures (precubical sets), for continuous structures (a form of directed space), and whether we have a form of tameness in that case, as we would like in any persistence oriented theory. The result of first investigations are shown in Appendix \ref{sec:homdir}. 


\section{Exact sequences and isomorphisms}

\label{sec:exact}

\paragraph{Invariance under dihomeomorphism}

Let $X$ and $Y$ be two cubical complexes. Then we have:

\begin{theorem}
\label{thm:invariance}
Suppose $\mid X\mid $ and $\mid Y\mid $ are dihomeomorphic, i.e. isomorphic in the category of directed spaces. Then for all $i\geq 1$, $HM_i[Subd(X)]$ is bisimulation equivalent to $HM_i[Subd(Y)]$.
\end{theorem}

%

\begin{proof}
By Theorem \ref{thm:bisimequivnathom}, for all $i \geq 1$, $HM_i[Subd(X)]$ (resp. $HM_i[Subd(Y)]$) is bisimulation equivalent to the natural homology $HN_i$ of $\mid X\mid $ (resp. of $\mid Y\mid $). As $\mid X\mid $ and $\mid Y\mid $ are dihomeomorphic, and natural homology is an invariant under dihomeomorphism \cite{goubault2023semiabelian}, $HN_i(\mid X\mid )$ and $HN_i(\mid Y\mid )$ are isomorphic, as natural systems of $R$-vector spaces. 

From this and the bisimulation above, it is straightforward to construct a bisimulation directly between $HM_i[Subd(X)]$ and $HM_i[Subd(Y)]$.
\end{proof}

\paragraph{Dicontractibility}

A po-space is called dicontractible, as in \cite{dirtop} (Theorem 1), if the dipath space map $\chi: \ \diP{X} \rightarrow X \times X$, which associates to each path (or trace) $p$ its pair of endpoints: $\chi(p)=(p(0),p(1))$, has a global section, from the subspace in $X\times X$ of reachable pair of points $\Gamma_X=\{ (x,y)\in X\times X \ \mid  \ \exists p \in dX, \ p(0)=x, \ p(1)=y \ \}$.

Then we have:

\begin{lemma}
A po-space, which is a geometric realization of some $X \in Cub$, is dicontractible implies $HM_1[X]=R$, with the trivial $R[X]$-bimodule structure, and $HM_i[X]=0$, for all $i\geq 2$.
\end{lemma}

\begin{proof}
By Theorem 1 of \cite{dirtop}, $\mid X\mid $ is such that 
all dipath spaces $\diP{X}^b_a$ are contractible, for all
$(a,b) \in \Gamma_X$, hence, by Lemma \ref{lem:homtrace}, 
$e_a\bullet HM_1 \bullet e_b=R$ and $e_a \bullet HM_i \bullet e_b=0$, for $i\geq 2$. Therefore the 
homology bimodules are all constant, either with value $R$ or with value 0. 
\end{proof}

\begin{remark}
It
is a simple exercise to see that the relation between the categorification (Definition \ref{def:categoricalrepresentation}) of $HM_1[X]$ (resp. $HM_i[X]$, $i\geq 2$) and the categorification of $R$ seen as an $A$-bimodule with $A$ the ring algebra of $R$ with trivial (identity) action, 
which 
relates all its objects to the only object 1 of $R$ is a bisimulation equivalence. 
\end{remark}


\paragraph{Change of coefficients}

Restricting the algebra, which acts on the homology bimodules (i.e. considering the restriction of scalars functor of Definition \ref{def:restrictionscalars}) can be seen as a well behaved filtering of the homological information present originally, in two important ways. In what follows, $X$ is a finite precubical set with proper length covering. 

\begin{lemma}
\label{lem:rel1}
Given $Y$ a sub-precubical set of $X \in Cub$, inducing, as inclusions are injective on objects, a map of algebras $g: \ R[Y]\rightarrow R[X]$. Let ${ }_Y R_*[X]$ be the chain of $R[Y]$-bimodules which is, in dimension $i$, equal to $g^*(R_i[X])$. Its homology, $H_i({ }_Y R_*[X])$ is isomorphic to the the $R[Y]$-bimodule ${ }_Y HM_i[X]=g^*(HM_i[X])$.
\end{lemma}

\begin{proof}
The  $i$th homology of the chain of $R[Y]$-bimodules ${ }_Y R_*[X]$ is $Ker \ g^* \partial^i/$ $Im \ g^* \partial^i$, which is the co-kernel of $g^* f$ where $f$ is the map ${ }_Y R_{i+1}[X] \rightarrow Ker \ \partial^i$ induced canonically by $g^* \partial^{i+1}: \ { }_Y R_{i+1}[X] \rightarrow { }_Y R_i[X]$ in the kernel diagram. 

As the restriction of scalars functor $g^*$ has a left and a right adjoint, it commutes with limits and colimits. Furthermore, as $g^*(0)=0$, it commutes with kernels and co-kernels. 

Hence $Ker \ g^* \partial^i/Im \ g^* \partial^{i+1}$ is isomorphic to $g^*(Ker \ \partial^i/Im \ \partial^{i+1})$, which is ${ }_Y HM_i[X]$.
\end{proof}

Indeed, the left (resp. right) action of $y \in R[Y]$ on ${ }_Y M$, for $M$ any $R[X]$-bimodule, is the action of $y$ seen as an element of $R[X]$. 

\begin{remark}
\label{rem:extrem}
Consider the extension of scalar $g_!$ induced by $g$ of Lemma \ref{lem:rel1}. By definition, ${ }^X R_i[Y]=g_!(R_i[Y])$ and ${ }^X HM_i[Y]=g_!(HM_i[Y])$. 
By Lemma \ref{lem:relsub}, 
As $g_!$ is a left adjoint, it preserves co-limits, and as $g_!(0)=0$, it preserves co-kernels, but not necessarily all kernels, a priori.  

Still, in the case of $\partial$, from the free $R[Y]$-bimodule $R_{i+1}[Y]$ to the free $R[Y]$-bimodule $R_i[Y]$, we have enough properties to get a similar result as Lemma \ref{lem:rel1}. 
\end{remark}



\begin{proposition}
\label{lem:lem18}
Given $X$ and $Y$ in $Cub$ such that $(X,Y)$ a relative pair of precubical sets, 
${ }^X HM_i[Y]$ is isomorphic to $H_i({ }^X R_*[Y])$ for all $i\in \N$, $i\geq 1$.    
\end{proposition}

\begin{proof}
Because of Remark \ref{rem:extrem}, we only have to check that the extension of scalar $g_!$ preserves the kernel of $\partial$ (we know that that it preserves all co-kernels). 

When $i=1$, $HM_i[Y]$ is the co-kernel of $\partial: \ R_2[Y] \rightarrow R_1[Y]$ and we do not have to compute any kernel; the result of the lemma holds, trivially. 

Now, suppose $i\geq 2$. It is straightforward to show that ${ }^X Ker \ \partial_{\mid R_{i+1}[Y]\rightarrow R_i[Y]}$ is a submodule of $Ker \ { }^X \partial_{\mid R_{i+1}[Y]\rightarrow R_i[Y]}$. Indeed, the $p\otimes c\otimes q$ with $p, q \in Ch^{=0}(X)$ and $c \in Ch^{=i}(Y)$ with $\partial(c)=0$, generate ${ }^X Ker \ \partial$ and are also such that $\partial(p \otimes c \otimes q)=p\otimes \partial(c) \otimes q=0$, hence, are in $Ker { }^X \partial$. 



We now prove that all elements of $Ker \ { }^X \partial_{\mid  R_{i+1}[Y] \rightarrow R_i[Y]}$ are elements of ${ }^X Ker \ \partial_{\mid  R_{i+1}[Y] \rightarrow R_i[Y]}$. 

Consider $x \in Ker { }^X \partial \subseteq { }^X R_i[Y]$, with $i\geq 2$. 
Then by Remark \ref{rem:newcanonical}, $x$ can be written canonically as: 
$$
x = \sum\limits_{i\in I, j\in J} \lambda_{i,j} p'_{i,j} \otimes c_j \otimes q'_{i,j}
$$
with $c_j\in Ch^{=i-1}(Y)$ and the $p'_{i,j}$ and $q'_{i,j}$ entirely composed of 1-cells in $X\backslash Y$ 
and the condition that $p_{i',j}=p_{i,j}$ and $q_{i',j}=q_{i,j}$ implies $i=i'$. 

Let us now write $\partial(c_j)$ as an $R$-linear combination of $(i-2)$-cube chains in $Y$: $\partial(c_j)=\sum\limits_{k\in K} \mu_{k,j} d_{k,j}$ with $d_{k,j} \in Ch^{=i-2}(Y)$ and $\mu_{k,j}\in R$. 
Then 
$\partial(x)=\sum\limits_{i\in I, j\in J,k\in K} \lambda_{i,j} \mu_{k,j} p'_{i,j} \otimes d_{k,j} \otimes q'_{i,j}$. We now organize all the $d_{k,j}$, which are $(i-2)$-cube chains in $Y$, on a basis of the $R$-vector space of $(i-2)$-cube chains in $Y$, $(d_u)_{u\in U}$: $d_{k,j}=d_{u(k,j)}$. Now, 
$$\begin{array}{rcl}
\partial(x) & = & \sum\limits_{i\in I, j\in J,k\in K} \lambda_{i,j} \mu_{k,j} p'_{i,j} \otimes d_{u(k,j)} \otimes q'_{i,j}\\
& =& \sum\limits_{i\in I, j\in J, u \in U} \sum\limits_{k \in K, u=u(k,j)} \lambda_{i,j} \mu_{k,j} p'_{i,j} \otimes d_{u} \otimes q'_{i,j}\\
& =& \sum\limits_{i\in I, j\in J, u \in U} \sum\limits_{k \in K, u=u(k,j)} \lambda_{i,j} \mu_{k,j} p'_{i,j} \otimes d_{u} \otimes q'_{i,j}
\end{array}$$
As this is the canonical decomposition of elements of ${ }^X R_{i-1}[Y]$, see Lemma \ref{lem:relsub} and Remark \ref{rem:newcanonical}, this implies that $\sum\limits_{k\in K, u=u(k,j)} \mu_{k,j}=0$. But,
$$\begin{array}{rcl}
\partial(c_j) & = & \sum\limits_{k\in K} \mu_{k,j} d_{k,j} \\
& = & \sum\limits_{u \in U} (\sum\limits_{k\in K, u=u(k,j)} \mu_{k,j}) d_u \\
& = & 0
\end{array}$$
Hence $x \in { }^X Ker \ \partial$. 
    \label{eq:flat}
\end{proof}

\begin{remark}
Indeed, the proof above is typical of the flat nature of the (free) bimodules of $i$-cube chains we consider, for $i\geq 1$. 
\end{remark}


There are other change of coefficients of interest in our case, for which we only briefly describe one, since it will not be of use in the rest of the paper: 

\begin{lemma}
\label{lem:rel2}
Given $Y$ a sub-precubical set of $X$. We have:
\begin{itemize}
\item $R[Y]$ is isomorphic to the algebra $R[X]$ quotiented by the two-sided ideal $I^X_Y=R[X](X_{\leq 1}\backslash Y_{\leq 1} \cup X_{0}\backslash Y_0)R[X]$ (where $X_0\backslash Y_0$ denotes the set of idempotents $e_{x}$ with $x\in X_0\backslash Y_0$ in $R[X]$)
\item The canonical map of algebras $h: \ R[X]\rightarrow R[X]/I^X_Y$ induces, by restriction of scalars, a functor $h^*: { }_{R[Y]} Mod_{R[Y]} \rightarrow { }_{R[X]} Mod_{R[X]}$ 
\item Calling $R^X_*[Y]$ the chain of $R[X]$-bimodules which is, in dimension $i$, equal to $h^*(R_i[Y])$, its homology $H_i(R^X_*[Y])$ is isomorphic to the $R[X]$-bimodule $HM^X_i[Y]=h^*(HM_i[Y])$. 
\end{itemize}    
\end{lemma}

\begin{proof}
Consider the following map $R[Y] \rightarrow R[X]/I^X_Y$ which sends paths $q$ in $Y_{\leq 1}$ to the class [q] of $q$ in $R[X]/I^X_Y$. This is well defined since a path in $Y_{\leq_1}$ is in particular a path in $X_{\leq 1}$. Now consider the map from $R[X]/I^X_Y$ to $R[Y]$ which sends paths classes $[p]$ of paths $p$ in $X_{\leq 1}$ modulo $I^X_Y$ to 0 if $p$ is a path in $X_{\leq 1}$ that is not a path in $Y_{\leq 1}$, or to $p$ otherwise. In the first case, this means that it is either a constant path $p=e_a$ with $a \in X_0\backslash Y_0$ or a path $p=(p_1,\ldots,p_m)$ with at least one $p_i \in X_1 \backslash Y_1$: in both case $p$ is zero modulo $I^X_Y$. These two maps are inverse to one another, which proves the isomorphism between $R[Y]$ and $R[X]/I^X_Y$. 

The second statement of the lemma comes from the definition of the restriction of scalars functor, Section \ref{sec:assoc}. The last statement is proven in the same way as for Lemma \ref{lem:rel1}. 

Indeed, the action of paths of $X_{\leq 1}$ on $c \in HM ^X_i[Y]$ which are not paths in $Y_{\leq 1}$ is 0. 
\end{proof}

\begin{remark}
For intervals $\mathcal{I}$ in $F_{X_{\leq 1}}$ as in Section \ref{sec:persistencemod}, the categorical presentation of $HM^X_i[\mathcal{I}]$ is the restriction $\mathcal{F}_{HM_i[\mathcal{I}]_{\mid \mathcal{I}}}$. 


As $X_{\leq 1}$ is the direct limit of the diagram made of the inclusion of its (not necessarily maximal) paths (or intervals), $R[X]$ is the direct limit of $R[\mathcal{I}]$ over the diagram of path algebra inclusions of paths into the quiver $X_{\leq 1}$, by Lemma 2.5 of \cite{leavitt}.
We believe that $HM_i[X]$ is also the colimit of all $HM^X_i[\mathcal{I}]$, the direct limit of the one-parameter homologies on all paths of $X_{\leq 1}$, as proved, for natural homology, in \cite{calk2023persistent}. 
\end{remark}

\begin{remark}
In general, we do not have inclusions between $R^X_i[Y]$ and $R_i[X]$. Indeed, if $Y_{\leq 1}$ is stricly included in $X_{\leq 1}$ then, for e.g. $x \in X_1\backslash Y_1$, we will have for all $m \in R^X_i[Y]$, $x \bullet m=0$, whereas $x \bullet m$ may not be zero if we see $m$ as an element of $R_i[X]$ through the inclusion of $Y$ in $X$. 
\end{remark}



\paragraph{K\"unneth formula}


K\"unneth formula is an important tool in algebraic topology \cite{Massey}, and is logical to look at now since we examined in previous section, the effect of simple change of coefficients. 

Such formulas have only relatively recently been considered in the context of persistent homology, see e.g. 
\cite{bubkunneth,gakhar2019kunneth}. In these papers, the K\"unneth theorems are for the tensor product between persistence modules on the same underlying category. In our work, we are considering a K\"unneth theorem for the tensor product between cubical complexes, which involves computing persistence modules over different underlying categories. More related to our work is \cite{CARLSSON2020106244}, where tensor products of persistence modules consider different underlying categories (or filtrations), but these filtrations are induced by a metric, and are very different from the filtrations induced by the tensor product of cubical complexes (corresponding to the cartesian product of their directed geometric realizations) that we are considering here. 


Here we consider instead the problem of linking the homology modules of a product space, with the tensor product of the respective homology modules. This is a deeply multiparameter persistence type of result, since for this, we need to consider the product of the path algebras, i.e. some form of tensor product of the underlying quivers, as a multiparameter persistence parameter space.


\begin{definition}[\look{Tensor product of precubical sets}]
\label{def:tensorcub}
Let $X$ and $Y$ be two precubical sets. Their tensor product is the precubical set $X \otimes Y$ with: 
\begin{itemize}
\item $(X \otimes Y)_n=\coprod\limits_{i+j=n} X_i \times Y_j$
\item for $(c,d)\in X_i \times Y_j$ with $i+j=n$, $d^{\epsilon}_k(c,d)=\left\{\begin{array}{ll}
(d^{\epsilon}_k(c),d) & \mbox{if $0\leq k \leq i-1$} \\
(c,d^{\epsilon}_{k-i}(d)) & \mbox{if $i \leq k \leq n-1$}
\end{array}\right.$
\end{itemize}
\end{definition}

This definition actually extends to the tensor product of CW-complexes, see e.g. Chapter VI of \cite{Massey}, and in particular Theorem 2.1. 
Note that the chain complex generated by the tensor product of two cell complexes is the (chain) tensor product of the underlying corresponding chain complexes, as defined in Definition 2.1 of \cite{Massey} and recapped below, for chain complexes of $R$-vector spaces and then extended to chain complexes of $A$-bimodules ($A$ being an algebra) in Definition \ref{def:tensoralgmod}: 

\begin{definition}[\look{Tensor product of chain complexes}]
\label{def:tensoralgmod}
Let $C=(C_*,\partial_C)$ and $(D_*,\partial_D)$ be two chain complexes of $R$-vector spaces. Their tensor product $C \otimes D$ is defined as the chain complex with:
\begin{itemize}
\item $(C\otimes D)_{n}=\mathop{\oplus}\limits_{i+j=n} C_i \times D_j$
\item $\partial^n_{C \otimes D}: \ (C\otimes D)_n \rightarrow (C\otimes D)_{n-1}$ is such that $\partial^n_{C\otimes D}(c\otimes d)=\partial^i_C(c)\otimes d +(-1)^i c \otimes \partial^j_D(d)$ where $c \in C_i$ and $d \in D_j$ with $i+j=n$
\end{itemize}
\end{definition}

For chain complexes of bimodules over algebras, each $R$-vector space $(C\otimes D)_n$ is given the structure of a $A\otimes B$-bimodule with (as seen in Section \ref{sec:assoc}): 

$$
(a\otimes b)\bullet c \otimes d \bullet (a' \otimes b') = (a\bullet c \bullet a')\otimes (b \bullet d \bullet b')
$$
\noindent for $a$, $a'$ in $A$ and $b$, $b'$ in $B$. 

Consider now two precubical sets $X\in Cub$ and $Y\in Cub$, and their directed geometric realization $\mid X\mid $ and $\mid Y\mid $. 
Let $x$ and $x'$ (resp. $y$ and $y'$) be two points in $\mid X\mid $ (resp. in $\mid Y\mid $) and $p$ (resp. $q$) be a dipath from $x$ to $x'$ (resp. from $y$ to $y'$) in $\mid X\mid $ (resp. in $\mid Y\mid $). As is well known \cite{Massey}, the geometric realization $\mid X \otimes Y\mid $ of $X\otimes Y$ is isomorphic to the cartesian product of the geometric realizations. This easily extends to the directed geometric realization $\mid X\mid \times \mid Y\mid $. 


Consider the $i$th natural homology $HN_i(\mid X\otimes Y\mid )(p,q)$ of $\mid X \otimes Y\mid =\mid X\mid \times \mid Y\mid $ at $(p,q)$. It is equal to the $(i-1)$th homology of the path space (or equivalently the trace space) from $(x,y)$ to $(x',y')$ in $\mid X\mid  \times \mid Y\mid $. 

Dipaths in $\mid X\mid \times \mid Y\mid $ are directed maps from $\I$ to $\mid X\mid \times \mid Y\mid $, i.e. are pairs of directed maps from $\I$ to $\mid X\mid $ and from $\I$ to $\mid Y\mid $. This means dipaths in $\mid X \otimes Y\mid $ can be identified with pairs of dipaths in $d\mid X\mid $ and in $d\mid Y\mid $ and dipaths from $(x,y)$ to $(x',y')$ in $\mid X\otimes Y\mid $, $d\mid X \otimes Y\mid $ are pairs of dipaths, from $x$ to $x'$ and from $y$ to $y'$, i.e. are the elements of $d\mid X\mid (x,x') \times d\mid Y\mid (y,y')$. Thus $H_{i-1}(d\mid X\otimes Y\mid ((x,y),(x',y')))$ is isomorphic to 
$H_{i-1}(d\mid X\mid (x,x')\times d\mid Y\mid (y,y'))$, which, by K\"unneth formula \cite{Massey}, for field coefficients, is isomorphic to $(H_*(d\mid X\mid (x,x'))\otimes H_*(d\mid Y\mid (y,y')))_{i-1}$. 

Therefore we proved that: 
$$
HN_*(\mid X\times Y\mid )(p,q)=HN_*(\mid X\mid )(p)\boxtimes HN_*(\mid Y\mid )(q)
$$
\noindent where $\boxtimes$ is a shifted version of the tensor product of chain complexes of Definition \ref{def:tensoralgmod} as follows: $(C_*\boxtimes D_*)_{n+1}=s(C_*)\otimes s(D_*)$ where $s(C_*)_{i+1}=C_i$ for $i \in \N$. This is due to the fact that our homology bimodules have shifted indices with respect to the (singular) homology of the path spaces, as defined in Definition \ref{def:hommodule}. 

Suppose now that $X$ and $Y$ are cubical complexes, as in Definition 
\ref{def:cubicalcomplex}. It is straightforward to see that $X\otimes Y$ is a cubical complex as well. 

By Lemma \ref{lem:homtrace}, we know that $e_{(x,y)} \bullet HM_i[X\otimes Y] \bullet e_{(x',y')}$ is isomorphic to $HN_i(\mid X\otimes Y\mid )(p,q)$, for any $p$ (resp. $q$) dipath from $x$ to $x'$ in $\mid X\mid $ (resp. from $y$ to $y'$ in $\mid Y\mid $), which is isomorphic to $HN_*(\mid X\mid )(p)\otimes HN_*(\mid Y\mid )$ by the above. 

By Lemma \ref{lem:homtrace} again, we know that $HN_i(\mid X\mid )(p)$ (resp. $HN_i(\mid Y\mid )(q)$) is isomorphic to $e_x \bullet HM_i[X]\bullet e_y$ (resp. $e_{x'}\bullet HM_i[Y] \bullet e_{y'}$). As this is true for all $x$, $x'$, $y$, $y'$, and as, as an $R$-vector space, $HM_i[X]$ (resp. $HM_i[Y]$) is the coproduct of all the $e_x \bullet HM_i[X] \bullet e_{x'}$ (resp. of all the $e_{y} \bullet HM_i[Y] \bullet e_{y'}$), we get that, as a chain of $R$-vector spaces, $HM_*[X\otimes Y]$ is isomorphic to $HM_*[X]\boxtimes HM_*[Y]$. The tensor product in the latter formula is the chain complex tensor product of Definition \ref{def:tensoralgmod}, where $HM_*[X]$ and $HM_*[Y]$ are considered as chain complexes with zero differential. 

Now this isomorphism is actually an isomorphism of $R[X\otimes Y]$-bimodules. First, $HM_*[X]\boxtimes HM_*[Y]$ defines not only a chain of $R$-vector spaces but a chain of $R[X]\otimes R[Y]$-bimodules, since $HM_*[X]$ (resp. $HM_*[Y]$) is a chain of $R[X]$-bimodules (resp. $R[Y]$-bimodules), using Definition \ref{def:tensoralgmod}. 

The algebra $R[X]\otimes R[Y]$ is not isomorphic to $R[X\otimes Y]$ in general, but is rather a (generally strict) sub-algebra of $R[X\otimes Y]$. Indeed, generators of $R[X]\otimes R[Y]$ are $p\otimes q$, with $p=(p_1,\ldots,p_k)_{x,x'}$ path from $x$ to $x'$ and $q=(q_1,\ldots,q_l)_{y,y'}$ from $y$ to $y'$, and generators of $R[X\otimes Y]$ are \look{shuffles} of such paths, i.e. are of the form $(r_1,\ldots, r_{k+l})_{(x,y),(x',y')}$ with $r_i$ equal to some $p_{\alpha(i)}\otimes e_{b}$ or $e_a \otimes q_{\beta(i)}$ where $\alpha$ (resp. $\beta$) is an increasing map from $\{1,\ldots,k+l\}$ to $\{1,\ldots,k\}$ (resp. $\{1,\ldots,l\}$) and $b$ is the boundary of some $q_j$, $a$, of some $p_i$. This is made more explicit in the lemma below: 

\begin{lemma}
\label{lem:tensorpathalg}
Let $X$ and $Y$ be two precubical sets. The path algebra $R[X \otimes Y]$ is generated by elements of the form $(u,e_b)$ and $(e_a,v)$ where $u \in X_1$, $b\in Y_0$, $a \in X_0$ and $v \in Y_1$, with the following multiplication:
$$
\begin{array}{rcl}
(u,e_b)(e_a,v) & = & 
0 \mbox{, if $d^1(u)\neq a$ or $d^0(v)\neq b$} \\
(e_a,v)(u,e_b) & = & 
0 \mbox{, if $d^0(u)\neq a$ or $d^1(v)\neq b$} \\
(e_a,v)(e_{c},z) & = & \left\{\begin{array}{ll}
(e_a,vz) & \mbox{if $a=c$} \\
0 & \mbox{otherwise}
\end{array}\right. \\
(u,e_b)(w,e_d) & = & \left\{\begin{array}{ll}
(uw,e_b) & \mbox{if $b=d$} \\
0 & \mbox{otherwise}
\end{array}\right.
\end{array}
$$
\end{lemma}

\begin{proof}
This comes directly from the definition of the product of two precubical sets, Definition \ref{def:tensorcub}. Elements of dimension one in $X\otimes Y$ are or the form $(a,d)$ or $(c,b)$ with $a \in X_0$, $b\in Y_0$, $c\in X_1$ and $d \in Y_1$.
\end{proof}

Overall, this proves that ${ }_{R[X]\otimes R[Y]} HM_*[X\otimes Y]=HM_*[X]\boxtimes HM_*[Y]$ (the change of coefficient being applied on each element of the chain complex of bimodules), but we can in fact prove that there is a natural action of $R[X\otimes Y]$ on $HM[X]\boxtimes HM[Y]$ which makes it isomorphic, as a $R[X\otimes Y]$-bimodule to $HM_*[X \otimes Y]$. 

Indeed, as we saw, all shuffles of $p$ and $q$ go from $(x,y)$ to $(x',y')$, as does $p\otimes q$. But the action of any of these shuffles (on the left, as well as on the right) on any element of $HM_i[X\otimes Y]$ depends only only on $(x,y)$ and $(x',y')$ since all shuffles from $(x,y)$ and $(x',y')$ are dihomotopic in $\mid X\otimes Y\mid $. Hence we can make $HM_*[X]\boxtimes HM_*[Y]$ into a $R[X\otimes Y]$-bimodule, isomorphic to $HM_*[X\otimes Y]$, by setting the left (resp. right) action of $(r_1,\ldots,r_{k+l})$ on $x\otimes y$ to be equal to $(p\bullet x)\otimes (q \bullet y)$ (resp. $(x\bullet p)\otimes (y \bullet q)$). 


In summary, we proved a version of K\"unneth's formula in the field case, and for cubical complexes: 

\begin{theorem}
\label{thm:kunneth}
Let $X$ and $Y$ be cubical complexes and $R$ a field. Then 
$$
HM_*[X\otimes Y]=HM_*[X]\boxtimes HM_*[Y]
$$
\noindent with the action of shuffles of paths $p$ of $X$ with $q$ of $Y$ on $HM_*[X]\boxtimes HM_*[Y]$ to be the action of $p$ on the first component, and action of $q$ on the second component.
\end{theorem}

\begin{remark}
Indeed, the $R[X\otimes Y]$-bimodule structure that we defined on the $R[X]\otimes R[Y]$-bimodule $(HM_*[X]\boxtimes HM_*[Y])_n$ is derived from the algebra map: 
$$
\begin{array}{llcl}
h: & R[X \otimes Y] & \rightarrow & R[X]\otimes R[Y] \\
\end{array}
$$
\noindent which to any generator $(u,v)$ of $R[X\otimes Y]$ (see Lemma \ref{lem:tensorpathalg}) associates $(u,v)$, this time ruled by the algebra multiplication of $R[X]\otimes R[Y]$, which is $(u,v)(u',v')=(uu',vv')$. 

It is easy to prove that the structure of $R[X\otimes Y]$-bimodule we put on $(HM_*[X]\boxtimes HM_*[Y])_n$ in Theorem \ref{thm:kunneth} is the restriction of scalars $h^*(HM_*[X]\boxtimes HM_*[Y])_n$.
\end{remark}

\begin{remark}
Although in order to prove the theorem, we went through the geometric realization functor, the isomorphism is simple to describe, directly in terms of cubical complexes. 

Let $i$, $j$ and $n$ be integers such that $n=i+j-1$. There is a natural map: 
$$
\begin{array}{llcl}
& HM_i[X]\times HM_j[Y] & \rightarrow & HM_n[X \otimes Y] \\
f: \ & ([(c_1,\ldots,c_k)_{x,x'}],[(d_1,\ldots,d_l)_{y,y'}]) & \rightarrow & [(c_1\otimes e_y,\ldots, c_k\otimes e_y,\\
& & & e_{x'}\otimes d_1,\ldots,e_{x'}\otimes d_l)_{(x,y),(x',y')}]
\end{array}
$$

The map is well defined since, obviously, as $(c_1\otimes e_y,\ldots,c_k\otimes e_y)$ is an $(i-1)$-chain and $(e_{x'}\otimes d_1,\ldots,e_{x'}\otimes d_l)$ is an $(j-1)$-chain, $(c_1\otimes e_y,\ldots,c_k\otimes e_y,e_{x'}\otimes d_1,\ldots,e_{x'}\otimes d_l)$ is an $(i+j-2)$-cube chain by Remark \ref{rem:dimension}, i.e. an $(n-1)$-cube chain. 

We have to check that the definition does not depend on the class chosen for $c$ and $d$. Indeed, if $c'=c+\partial \gamma$ and $d'=d+\partial \zeta$, where $\gamma=(\gamma_1,\ldots,\gamma_k)$ and $\zeta=(\zeta_1,\ldots,\zeta_l)$, then $(c'_1\otimes e_y,\ldots,c'_k\otimes e_y,e_{x'}\otimes d'_1,\ldots,e_{x'}\otimes d'_l)=(c_1\otimes e_y,\ldots,c_k\otimes e_y,e_{x'}\otimes d_1,\ldots,e_{x'}\otimes d_l)+\partial(\gamma_1\otimes e_y,\ldots,\gamma_k\otimes e_y,e_{x'}\otimes \zeta_1,\ldots,e_{x'}\otimes \zeta_l)$ since, by Definition \ref{def:tensorcub}, all boundaries of $e_a \otimes u$ or $u\otimes e_a$ are $e_a$ tensor (on the left or on the right) $u$ (the boundaries of any 0-cube chain, including the empty one, are zero). 


Showing map $f$ is injective is simple. $f([c],[d])=f([c'],[d'])$ implies $(c'_1\otimes e_y,\ldots,$ $c'_k\otimes e_y,e_{x'}\otimes d'_1,\ldots,e_{x'}\otimes d'_l)=(c_1\otimes e_y,\ldots,c_k\otimes e_y,e_{x'}\otimes d_1,\ldots,e_{x'}\otimes d_l)+\partial(\gamma_1\otimes e_y,\ldots,\gamma_k\otimes e_y,\zeta_1,\ldots,\zeta_l)$, which in turn implies
$c'=c+\partial \gamma$ and $d'=d+\partial \zeta$, where $\gamma=(\gamma_1,\ldots,\gamma_k)$ and $\zeta=(\zeta_1,\ldots,\zeta_l)$. Hence $[c']=[c]$ and $[d']=[d]$. 

Surjectivity of $f$ is more complicated to prove though, without resorting to the geometric realization and the classical K\"unneth theorem as we did. 
\end{remark}





    

We proved the K\"unneth formula for cubical complexes, but we believe this should be true for all precubical sets that we are considering here, that is, finite precubical sets with proper length covering, at least. 

\begin{example}
Consider again 
Example \ref{ex:emptysquare}, $\diCub$: 
\[\begin{tikzcd}
  4 \arrow[r] \arrow[d]
    & 2 \arrow[d] \\
  3 \arrow[r]
& 1 \end{tikzcd}
\]
We get the following path algebra: 
$$
\begin{pmatrix}
R & 0 & 0 & 0 \\
R & R & 0 & 0 \\
R & 0 & R & 0 \\
R^2 & R & R & R 
\end{pmatrix}
$$

\noindent and indeed $HM_1[\diCub]$ the bimodule $R[\diCub]$ with the same matrix representation as the one for $R[\diCub]$ above, whereas $HM_i[\diCub]$ is 0 for all $i\geq 2$. 

The tensor product $\diCub\otimes \diCub$ is, up to classical homotopy, a torus. 

Now, $HM_1[\diCub\otimes \diCub]$ is $(HM_*[\diCub]\boxtimes HM_*[\diCub])_1$, that is, $HM_1[\diCub]\otimes HM_1[\diCub]$, using the notation, by block first, of Remark \ref{rem:notationtensor}:
$$
\begin{pmatrix}
\scriptstyle R \otimes \begin{pmatrix}
\scriptstyle R & \scriptstyle 0 & \scriptstyle 0 & \scriptstyle 0 \\
\scriptstyle R & \scriptstyle R & \scriptstyle 0 & \scriptstyle 0 \\
\scriptstyle R & \scriptstyle 0 & \scriptstyle R & \scriptstyle 0 \\
\scriptstyle R^2 & \scriptstyle R & \scriptstyle R & \scriptstyle R 
\end{pmatrix} & \scriptstyle 0 & \scriptstyle 0 & \scriptstyle 0 \\
\scriptstyle R \otimes \begin{pmatrix}
\scriptstyle R & \scriptstyle 0 & \scriptstyle 0 & \scriptstyle 0 \\
\scriptstyle R & \scriptstyle R & \scriptstyle 0 & \scriptstyle 0 \\
\scriptstyle R & \scriptstyle 0 & \scriptstyle  R & \scriptstyle 0 \\
\scriptstyle R^2 & \scriptstyle R & \scriptstyle R & \scriptstyle R 
\end{pmatrix} & \scriptstyle R \otimes \begin{pmatrix}
\scriptstyle R & \scriptstyle 0 & \scriptstyle 0 & \scriptstyle 0 \\
\scriptstyle R & \scriptstyle R & \scriptstyle 0 & \scriptstyle 0 \\
\scriptstyle R & \scriptstyle 0 & \scriptstyle R & \scriptstyle 0 \\
\scriptstyle R^2 & \scriptstyle R & \scriptstyle R & \scriptstyle R 
\end{pmatrix} & \scriptstyle 0 & \scriptstyle 0 \\
\scriptstyle R \otimes \begin{pmatrix}
\scriptstyle R & \scriptstyle 0 & \scriptstyle 0 & \scriptstyle 0 \\
\scriptstyle R & \scriptstyle R & \scriptstyle 0 & \scriptstyle 0 \\
\scriptstyle R & \scriptstyle 0 & \scriptstyle R & \scriptstyle 0 \\
\scriptstyle R^2 & \scriptstyle R & \scriptstyle R & \scriptstyle R 
\end{pmatrix} & \scriptstyle 0 & \scriptstyle R \otimes \begin{pmatrix}
\scriptstyle R & \scriptstyle 0 & \scriptstyle 0 & \scriptstyle 0 \\
\scriptstyle R & \scriptstyle R & \scriptstyle 0 & \scriptstyle 0 \\
\scriptstyle R & \scriptstyle 0 & \scriptstyle R & \scriptstyle 0 \\
\scriptstyle R^2 & \scriptstyle R & \scriptstyle R & \scriptstyle R 
\end{pmatrix} & \scriptstyle 0 \\
\scriptstyle R^2 \otimes \begin{pmatrix}
\scriptstyle R & \scriptstyle 0 & \scriptstyle 0 & \scriptstyle 0 \\
\scriptstyle R & \scriptstyle R & \scriptstyle 0 & \scriptstyle 0 \\
\scriptstyle R & \scriptstyle 0 & \scriptstyle R & \scriptstyle 0 \\
\scriptstyle R^2 & \scriptstyle R & \scriptstyle R & \scriptstyle R 
\end{pmatrix} & \scriptstyle R \otimes \begin{pmatrix}
\scriptstyle R & \scriptstyle 0 & \scriptstyle 0 & \scriptstyle 0 \\
\scriptstyle R & \scriptstyle R & \scriptstyle 0 & \scriptstyle  0 \\
\scriptstyle R & \scriptstyle 0 & \scriptstyle R & \scriptstyle 0 \\
\scriptstyle R^2 & \scriptstyle R & \scriptstyle R & \scriptstyle R 
\end{pmatrix} & \scriptstyle R \otimes \begin{pmatrix}
\scriptstyle R & \scriptstyle 0 & \scriptstyle 0 & \scriptstyle 0 \\
\scriptstyle R & \scriptstyle R & \scriptstyle 0 & \scriptstyle 0 \\
\scriptstyle R & \scriptstyle 0 & \scriptstyle R & \scriptstyle 0 \\
\scriptstyle R^2 & \scriptstyle R & \scriptstyle R & \scriptstyle R 
\end{pmatrix} & \scriptstyle R \otimes \begin{pmatrix}
\scriptstyle R & \scriptstyle 0 & \scriptstyle 0 & \scriptstyle 0 \\
\scriptstyle R & \scriptstyle R & \scriptstyle 0 & \scriptstyle 0 \\
\scriptstyle R & \scriptstyle 0 & \scriptstyle R & \scriptstyle 0 \\
\scriptstyle R^2 & \scriptstyle R & \scriptstyle R & \scriptstyle R 
\end{pmatrix}
\end{pmatrix}
$$
\noindent which, on the base $4\otimes 4'$, $4\otimes 3'$, $4\otimes 2'$, $4\otimes 1'$, 
$3\otimes 4'$, $3\otimes 3'$, $3\otimes 2'$, $3\otimes 1'$, 
$2\otimes 4'$, $2\otimes 3'$, $2\otimes 2'$, $2\otimes 1'$, 
$1\otimes 4'$, $1\otimes 3'$, $1\otimes 2'$, $1\otimes 1'$ (where the ' are given to the vertex names of the second copy of $\diCub$) in that order, is:
$$
\left(\begin{array}{cccccccccccccccc}
\scriptstyle R & \scriptstyle 0 & \scriptstyle 0 & \scriptstyle 0 & \scriptstyle 0 & \scriptstyle 0 & \scriptstyle 0 & \scriptstyle 0 & \scriptstyle 0 & \scriptstyle 0 & \scriptstyle 0 & \scriptstyle 0 & \scriptstyle 0 & \scriptstyle 0 & \scriptstyle 0 & \scriptstyle 0\\
\scriptstyle R & \scriptstyle R & \scriptstyle 0 & \scriptstyle 0 & \scriptstyle 0 & \scriptstyle 0 & \scriptstyle 0 & \scriptstyle 0 & \scriptstyle 0 & \scriptstyle 0 & \scriptstyle 0 & \scriptstyle 0 & \scriptstyle 0 & \scriptstyle 0 & \scriptstyle 0 & \scriptstyle 0\\
\scriptstyle R & \scriptstyle 0 & \scriptstyle R & \scriptstyle 0 & \scriptstyle 0 & \scriptstyle 0 & \scriptstyle 0 & \scriptstyle 0 & \scriptstyle 0 & \scriptstyle 0 & \scriptstyle 0 & \scriptstyle 0 & \scriptstyle 0 & \scriptstyle 0 & \scriptstyle 0 & \scriptstyle 0\\
\scriptstyle R^2 & \scriptstyle R & \scriptstyle R & \scriptstyle R & \scriptstyle 0 & \scriptstyle 0 & \scriptstyle 0 & \scriptstyle 0 & \scriptstyle 0 & \scriptstyle 0& \scriptstyle 0 & \scriptstyle 0 & \scriptstyle 0& \scriptstyle 0 & \scriptstyle 0 & \scriptstyle 0 \\
\scriptstyle R & \scriptstyle 0 & \scriptstyle 0 & \scriptstyle 0&  
\scriptstyle R & \scriptstyle 0 & \scriptstyle 0 & \scriptstyle 0 & \scriptstyle 0 & \scriptstyle 0& \scriptstyle 0 & \scriptstyle 0
& \scriptstyle 0 & \scriptstyle 0& \scriptstyle 0 & \scriptstyle 0 \\
\scriptstyle R & \scriptstyle R & \scriptstyle 0 & \scriptstyle 0 &
\scriptstyle R & \scriptstyle R & \scriptstyle 0 & \scriptstyle 0& \scriptstyle 0 & \scriptstyle 0& \scriptstyle 0 & \scriptstyle 0
& \scriptstyle 0 & \scriptstyle 0& \scriptstyle 0 & \scriptstyle 0 \\
\scriptstyle R & \scriptstyle 0 & \scriptstyle R & \scriptstyle 0 &
\scriptstyle R & \scriptstyle 0 & \scriptstyle  R & \scriptstyle 0 & \scriptstyle 0 & \scriptstyle 0& \scriptstyle 0 & \scriptstyle 0
& \scriptstyle 0 & \scriptstyle 0& \scriptstyle 0 & \scriptstyle 0 \\
\scriptstyle R^2 & \scriptstyle R & \scriptstyle R & \scriptstyle R & 
\scriptstyle R^2 & \scriptstyle R & \scriptstyle R & \scriptstyle R & \scriptstyle 0 & \scriptstyle 0& \scriptstyle 0 & \scriptstyle 0
& \scriptstyle 0 & \scriptstyle 0& \scriptstyle 0 & \scriptstyle 0 \\
\scriptstyle R & \scriptstyle 0 & \scriptstyle 0 & \scriptstyle 0  & \scriptstyle 0  & \scriptstyle 0  & \scriptstyle 0  & \scriptstyle 0 & \scriptstyle R & \scriptstyle 0 & \scriptstyle 0 & \scriptstyle 0 & \scriptstyle 0 & \scriptstyle 0& \scriptstyle 0 & \scriptstyle 0 \\
\scriptstyle R & \scriptstyle R & \scriptstyle 0 & \scriptstyle 0 & \scriptstyle 0  & \scriptstyle 0  & \scriptstyle 0  & \scriptstyle 0 & \scriptstyle R & \scriptstyle R & \scriptstyle 0 & \scriptstyle 0 & \scriptstyle 0 & \scriptstyle 0& \scriptstyle 0 & \scriptstyle 0 \\
\scriptstyle R & \scriptstyle 0 & \scriptstyle R & \scriptstyle 0 & \scriptstyle 0  & \scriptstyle 0  & \scriptstyle 0  & \scriptstyle 0 & \scriptstyle R & \scriptstyle 0 & \scriptstyle R & \scriptstyle 0 & \scriptstyle 0 & \scriptstyle 0& \scriptstyle 0 & \scriptstyle 0 \\
\scriptstyle R^2 & \scriptstyle R & \scriptstyle R & \scriptstyle R & \scriptstyle 0  & \scriptstyle 0  & \scriptstyle 0  & \scriptstyle 0 & \scriptstyle R^2 & \scriptstyle R & \scriptstyle R & \scriptstyle R & \scriptstyle 0 & \scriptstyle 0& \scriptstyle 0 & \scriptstyle 0 \\
\scriptstyle R^2 & \scriptstyle 0 & \scriptstyle 0 & \scriptstyle 0 &\scriptstyle R & \scriptstyle 0 & \scriptstyle 0 & \scriptstyle 0 &\scriptstyle R & \scriptstyle 0 & \scriptstyle 0 & \scriptstyle 0& \scriptstyle R & \scriptstyle 0 & \scriptstyle 0 & \scriptstyle 0 \\
\scriptstyle R^2 & \scriptstyle R^2 & \scriptstyle 0 & \scriptstyle 0 &\scriptstyle R & \scriptstyle R & \scriptstyle 0 & \scriptstyle  0& \scriptstyle R & \scriptstyle R & \scriptstyle 0 & \scriptstyle 0&\scriptstyle R & \scriptstyle R & \scriptstyle 0 & \scriptstyle 0 \\
\scriptstyle R^2 & \scriptstyle 0 & \scriptstyle R^2 & \scriptstyle 0 &\scriptstyle R & \scriptstyle 0 & \scriptstyle R & \scriptstyle 0& \scriptstyle R & \scriptstyle 0 & \scriptstyle R & \scriptstyle 0&\scriptstyle R & \scriptstyle 0 & \scriptstyle R & \scriptstyle 0 \\
\scriptstyle R^4 & \scriptstyle R^2 & \scriptstyle R^2 & \scriptstyle R^2 & \scriptstyle R^2 & \scriptstyle R & \scriptstyle R & \scriptstyle R& \scriptstyle R^2 & \scriptstyle R & \scriptstyle R & \scriptstyle R &\scriptstyle R^2 & \scriptstyle R & \scriptstyle R & \scriptstyle R \\
\end{array}
\right)
$$

In particular, there are four dipaths modulo homology from $4\otimes 4'$ to $1\otimes 1'$. Indeed, there are two connected components for the dipath space from $4$ to $1$ in $\diCub$ and two connected components for the dipath space from $4'$ to $1'$ in the second copy of $\diCub$, hence $2\times 2=4$ components for the product dipath space.

Similarly, by K\"unneth theorem, $HM_2[\diCub\times \diCub]=HM_2[\diCub]\otimes HM_1[\diCub]\oplus HM_1[\diCub]\otimes HM_2[\diCub]=0$.

One may wonder why we do not have a non trivial second homology module, since classically, $H_2$ of a torus is $R$. This is due to the fact that the directed homology theory we have been constructing is highly non-abelian. We exemplify this on the product of the two directed graphs of Example \ref{ex:kronecker}, $\diS \otimes \diS$. 
We recall that $\diS$ is: 
\[
\begin{tikzcd}
1 \arrow[r,bend left,"\alpha"] \arrow[r,bend right,swap,"\beta"] & 2
\end{tikzcd}
\]

The tensor product $\diS\otimes \diS$ is: 
\[
\begin{tikzcd}[column sep=1.5cm,row sep=1.5cm]
1\otimes 2' \arrow[r,bend left,"\alpha\otimes 2'"] \arrow[r,bend right,swap,"\beta\otimes 2'"] & 2\otimes 2' \\
1\otimes 1' \arrow[u,bend left,"1 \otimes \alpha'"] \arrow[u,bend right, swap,"1 \otimes \beta'"] \arrow[r,bend left,"\alpha\otimes 1'"] \arrow[r,bend right,swap,"\beta\otimes 1'"] & 2\otimes 1' \arrow[u,bend left,"2 \otimes \alpha'"] \arrow[u,bend right, swap,"2 \otimes \beta'"]
\end{tikzcd}
\]
\noindent plus the four 2-cells $\alpha\otimes \alpha'$, $\alpha \otimes \beta'$, $\beta\otimes \alpha'$ and $\beta\otimes \beta'$, which generate the $R$-vector space of 1-cube paths from $1\otimes 1'$ to $2\otimes 2'$. As $\diS$ is not a cubical complex we have to compute the homology modules by hand, without resorting to our K\"unneth formula, which we only proved so far for cubical complexes. 

We compute their boundaries: 
\begin{itemize}
\item $\partial ((\alpha\otimes \alpha'))=(1\otimes \alpha',\alpha\otimes 2')-(\alpha\otimes 1',2\otimes \alpha')$
\item $\partial ((\alpha\otimes \beta'))=(1\otimes \beta',\alpha\otimes 2')-(\alpha\otimes 1',2\otimes \beta')$
\item $\partial ((\beta\otimes \alpha'))=(1\otimes \alpha',\beta\otimes 2')-(\beta\otimes 1',2\otimes \alpha')$
\item $\partial ((\beta\otimes \beta'))=(1\otimes \beta',\beta\otimes 2')-(\beta\otimes 1',2\otimes \beta')$
\end{itemize}
And there is no linear combination of these four 2-cells whose boundary sums up to zero, hence $Ker \ \partial_1=0$ and $HM_2[\diS\otimes \diS]=0$. Abelianizing the cube paths to interpret them in the $R$-vector space generated by 1-cells, as we would do in ordinary homology, we would get: 
\begin{itemize}
\item $\partial (\alpha\otimes \alpha')=1\otimes \alpha'+\alpha\otimes 2'-\alpha\otimes 1'-2\otimes \alpha'$
\item $\partial (\alpha\otimes \beta')=1\otimes \beta'+\alpha\otimes 2'-\alpha\otimes 1'-2\otimes \beta'$
\item $\partial (\beta\otimes \alpha')=1\otimes \alpha'+\beta\otimes 2'-\beta\otimes 1'-2\otimes \alpha'$
\item $\partial (\beta\otimes \beta')=1\otimes \beta'+\beta\otimes 2'-\beta\otimes 1'-2\otimes \beta'$
\end{itemize}
And indeed, $\partial(\alpha\otimes \alpha'-\alpha\otimes \beta'-\beta\otimes \alpha'+\beta\otimes \beta')=0$ and $\alpha\otimes \alpha'-\alpha\otimes \beta'-\beta\otimes \alpha'+\beta\otimes \beta'$
is the generator of $H_2(\diS\otimes \diS)$.
\end{example}






\paragraph{Exact sequence of relative homology}



The category ${ }_R Mod_{R}$ is not abelian but we have the following exact sequence, which is close to classical relative homology sequences for simplicial or singular homology theories: 








\begin{theorem}
\label{prop:relhomology}
Suppose $(X,Y)$ is a relative pair of precubical sets. We have the following relative homology sequence: 

\begin{center}
    \begin{tikzcd}[arrow style=math font,cells={nodes={text height=2ex,text depth=0.75ex}}]
    & & 0 \arrow[draw=none]{d}[name=X,shape=coordinate]{} \\
       HM_1[X,Y] \arrow[curarrow=X]{urr}{} 
       & HM_{1}[X] \arrow[l] \arrow[draw=none]{d}[name=Y, shape=coordinate]{} & \arrow[l] \cdots \\
       HM_{i}[X,Y] \arrow[curarrow=Y]{urr}{} & HM_{i}[X] \arrow[l] \arrow[draw=none]{d}[name=Z,shape=coordinate]{} & { }^X HM_i[Y] \arrow[l] \\
       HM_{i+1}[X,Y] \arrow[curarrow=Z]{urr}{} & HM_{i+1}[X] \arrow[l] & \cdots \arrow[l]
   \end{tikzcd}
\end{center}
\end{theorem}

\begin{proof}
As we have seen already, for all $i \in \N$, ${ }^X R_i[Y]$ is a sub-$R[X]$-bimodule of $R_i[X]$,  the inclusion map $f_i: \ { }^X R_i[Y] \rightarrow R_i[X]$ being indeed injective, hence a monomorphism in the category of $R[X]$-bimodules.


We also have the projection map $\pi_i$ from the $R[X]$-bimodule $R_i[X]$ onto $R_i[X,Y]=R_i[X]/{ }^X R_i[Y]$, which is surjective by construction. 
Finally, 
the kernel of $\pi_i$ is the image of $f_i$ by construction as well, giving us the following short exact sequence of $R[X]$-bimodules: 

\begin{center}
\begin{tikzcd}
0 \arrow[r] & { }^X R_i[Y] \arrow[r,"f_i"] & R_i[X] \arrow[r,"\pi_i"] & R_i[X,Y] \arrow[r] & 0
\end{tikzcd}
\end{center}

Consider $\partial$, which we have seen already is a morphism of $R[X]$-bimodules from $R_i[X]$ to $R_{i-1}[X]$. As $Y$ is a sub-precubical set of $X$,
the boundary operator $\partial$ induces a morphism of $R[Y]$-bimodules from $R_i[Y]$ to $R_{i-1}[Y]$ since 
all maps $d_{k,B}: Ch^{=i-1}(X)^w_v \rightarrow Ch^{=i-2}(X)^w_v$ restrict to 
$d_{k,B}: Ch^{=i-1}(Y)^w_v \rightarrow Ch^{=i-2}(Y)^w_v$. 
Therefore, the extension of scalar functor $g_!$ where $g$ is the inclusion of algebras $R[Y]$ in $R[X]$, applied to morphism $\partial$ seen as a morphism of $R[Y]$-bimodules is a morphism of $R[X]$-bimodules from ${ }^X R_{i}[Y]$ to ${ }^X R_{i-1}[Y]$, which is the restriction of $\partial$ defined on $R_i[X]$, on its sub-$R[X]$-bimodule ${ }^X R_i[Y]$. Therefore $\partial$ 
induces also a map from $R_i[X,Y]$ to $R_{i-1}[X,Y]$. 

Hence the short exact sequence of $R[X]$-bimodules above is actually a short exact sequence of chains of $R[X]$-bimodules. The category of $R[X]$-bimodules being an abelian category, this induces the following long exact sequence of $R[X]$-bimodules:
\begin{center}
    \begin{tikzcd}[arrow style=math font,cells={nodes={text height=2ex,text depth=0.75ex}}]
    & & 0 \arrow[draw=none]{d}[name=X,shape=coordinate]{} \\
       HM_1[X,Y] \arrow[curarrow=X]{urr}{} 
       & HM_{1}[X] \arrow[l] \arrow[draw=none]{d}[name=Y, shape=coordinate]{} & \arrow[l] \cdots \\
       HM_{i}[X,Y] \arrow[curarrow=Y]{urr}{} & HM_{i}[X] \arrow[l] \arrow[draw=none]{d}[name=Z,shape=coordinate]{} & H_i({ }^X R_*[Y]) \arrow[l] \\
       HM_{i+1}[X,Y] \arrow[curarrow=Z]{urr}{} & HM_{i+1}[X] \arrow[l] & \cdots \arrow[l]
   \end{tikzcd}
\end{center}
But by Lemma \ref{lem:lem18} we know that $H_i({ }^X R_*[Y])={ }^X HM_i[Y]$ hence the exact sequence 
of the proposition.

\end{proof}




Let us now look more in detail at elements of $HM_{i}[X,Y]$. 
First, a cycle $[c]_{{ }^X R_i[Y]} \in Ker \ \partial_{\mid R_i[X]/{ }^X R_i[Y]\rightarrow R_{i-1}[X]/{ }^X R_{i-1}[Y]}$ is a class modulo ${ }^X R_i[Y]$ of $c \in R_i[X]$ such that $\partial c \in { }^X R_{i-1}[Y]$, and in that sense, it is called, as in the classical case, a $Y$-relative cycle. 

Such a $Y$-relative cycle is trivial in $HM_i[X,Y]$ if it is a $Y$-relative boundary, i.e. is such that $c=\partial b+a$ with $b \in R_{i+1}[X]$ and $a \in { }^X R_{i}[Y]$. 

As in the classical case of singular or simplicial homology, the connecting homomorphism 
$\partial^*$ 
sends an element $[c] \in HM_{i+1}[X,Y]$
represented by an 
$Y$-relative cycle 
$c\in R_{i+1}[X]$, to the class represented by the boundary 
$\partial c \in 
{ }^X R_i[Y] \subseteq  
R_i[X]$.

\begin{example}
\label{ex:relativehomology}
Consider $\overrightarrow{D}^n=\K^n$ the tensor product of $n$ copies of the precubical set $\K$ that has two vertices $0$ and $1$, and a unique 1-cell $a$ with $d^0(a)=0$ and $d^1(a)$. $\K$ ``represents'' a directed segment, and $\overrightarrow{D}^n$ is a precubical version of a directed $n$-disc. We write $\overrightarrow{S}^{n-1}$ for the boundary of $\overrightarrow{D}^n$, which is a version of a directed $(n-1)$-sphere, or ``hollow $n$-hypercube".  
As for the cube example, the vertices of $\overrightarrow{D}^n$ are numbered using a binary word of length $n$. 

Let us illlustrate the relative homology sequence for $(\overrightarrow{D}^n,\overrightarrow{S}^{n-1})$, when $n=2$. 
Consider first $X=\overrightarrow{D}^2$:
\[\begin{tikzcd}
  00 \arrow[r,"a0"] \arrow[d,left,"0b"] \arrow[dr,phantom,"C"]
    & 01 \arrow[d,"a1"] \\
  10 \arrow[r,"1b"]
& 11 \end{tikzcd}
\]

The algebra $R[X]$ is the matrix algebra: 
$$
\left(\begin{array}{l|cccc}
& 11 & 10 & 01 & 00 \\
\hline
11 & R & 0 & 0 & 0 \\
10 & R & R & 0 & 0 \\
01 & R & 0 & R & 0 \\
00 & R^2 & R & R & R
\end{array}\right)
$$

Now, $HM_1[\overrightarrow{D}^2]$ can be represented by the matrix: 
$$
\left(\begin{array}{l|cccc}
& 11 & 10 & 01 & 00 \\
\hline
11 & R & 0 & 0 & 0 \\
10 & R & R & 0 & 0 \\
01 & R & 0 & R & 0 \\
00 & R & R & R & R
\end{array}\right)
$$
\noindent and $HM_2[\overrightarrow{D^2}]$ can be represented by: 
$$
\left(\begin{array}{l|cccc}
& 11 & 10 & 01 & 00 \\
\hline
11 & 0 & 0 & 0 & 0 \\
10 & 0 & 0 & 0 & 0 \\
01 & 0 & 0 & 0 & 0 \\
00 & R & 0 & 0 & 0
\end{array}\right)
$$
\noindent all other $HM_i[\overrightarrow{D}^2]$, $i\geq 3$ being zero. 

We now compute the homology modules of $\overrightarrow{S}^1$. First, 
${ }^{\overrightarrow{D}^2} HM_1({\overrightarrow{S}^1})$ is:
$$
\left(\begin{array}{l|cccc}
& 11 & 10 & 01 & 00 \\
\hline
11 & R & 0 & 0 & 0 \\
10 & R & R & 0 & 0 \\
01 & R & 0 & R & 0 \\
00 & R^2 & R & R & R
\end{array}\right)
$$
\noindent since $R[\overrightarrow{D}^2]$ is actually the same as $R[\overrightarrow{S}^1]$. It is easy to see as well, that $HM_i({ }^{\overrightarrow{D}^2}  R_i({\overrightarrow{S}^1}))$ is zero for $i\geq 2$.

So, the long exact sequence in relative homology implies that all ${ }^{\overrightarrow{D}^2} HM_i[X,Y]=0$ for $i\geq 3$ and boils down to: 
\begin{center}
\begin{tikzcd}[column sep=small]
\scriptstyle
0 \arrow[r] & \scriptstyle HM_2[\overrightarrow{D}^2] \arrow[r] & \scriptstyle HM_{2}[\overrightarrow{D}^2,\overrightarrow{S}^1] \arrow[r] & \scriptstyle HM_1[\overrightarrow{S}^1] \arrow[r] & \scriptstyle HM_1[\overrightarrow{D}^2]\arrow[r] & \scriptstyle HM_{1}[\overrightarrow{D}^2,\overrightarrow{S}^1] \arrow[r] & 0
\end{tikzcd}
\end{center}


We can easily check now that 
$HM_2[\overrightarrow{D}^2,\overrightarrow{S}^1]$ to be: 
$$
\left(\begin{array}{l|cccc}
& 11 & 10 & 01 & 00 \\
\hline
11 & 0 & 0 & 0 & 0 \\
10 & 0 & 0 & 0 & 0 \\
01 & 0 & 0 & 0 & 0 \\
00 & R & 0 & 0 & 0
\end{array}\right)
$$
\noindent and $HM_1[\overrightarrow{D}^2,\overrightarrow{S}^1]=0$
(which can be computed directly as well). The interesting entry of this exact sequence of matrices of modules is entry corresponding to line $00$ and column $11$: 
\begin{center}
\begin{tikzcd}
0 \arrow[r] & R \arrow[r] & R \arrow[r] & R^2 \arrow[r] & R\arrow[r] & 0 \arrow[r] & 0
\end{tikzcd}
\end{center}
\end{example}

\begin{example}
Consider the {$X$ empty cube (``$\overrightarrow{S}^2$"), and the relative pair made of $(X,Y)$ with $Y$ the square in red (``$\overrightarrow{S}^1$") as in the figure below:}
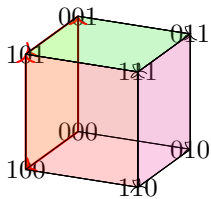
\begin{center}
\begin{tikzpicture}[scale=2,tdplot_main_coords]
    \coordinate (O) at (0,0,0);
    \tdplotsetcoord{P}{1.414213}{54.68636}{45}
    \draw[-,fill=green,fill opacity=0.1] (Pz) -- (Pyz) -- (P) -- (Pxz) -- cycle;
\draw [->] (O) -- (Pz);
    \draw[->] (Pz) -- (Pyz); 
    \draw[->] (Pyz) -- (P); 
    \draw[->,thick,red] (Pz) -- (Pxz);
    \draw[->] (Pxz) -- (P);
\node at (O) {$000$};
\node at (Px) {$100$};
\node at (Py) {$010$};
\node at (Pz) {$001$};
\node at (Pxy) {$110$};
\node at (Pxz) {$101$};
\node at (Pyz) {$011$};
\node at (P) {$111$};

   \draw[-,fill=red,fill opacity=0.1] (Px) -- (Pxy) -- (P) -- (Pxz) -- cycle;
    \draw[->] (Px) -- (Pxy); 
    \draw[->] (Pxy) -- (P); 
    \draw[->,thick,red] (Px) -- (Pxz);
    
    \draw[-,fill=magenta,fill opacity=0.1] (Py) -- (Pxy) -- (P) -- (Pyz) -- cycle;
    \draw[->] (Py) -- (Pxy); 
    \draw[->,thick,red] (O) -- (Px);
    \draw[->] (Py) -- (Pyz);
    
    \draw[->,fill=gray!50,fill opacity=0.1] (O) -- (Py) -- (Pyz) -- (Pz) -- cycle;
    \draw[->,thick,red] (O) -- (Pz); 
        \draw[->] (O) -- (Py); 

    \draw[->,fill=yellow,fill opacity=0.1] (O) -- (Px) -- (Pxz) -- (Pz) -- cycle;

    \draw[->,fill=green,fill opacity=0.1] (Pz) -- (Pyz) -- (P) -- (Pxz) -- cycle;

    \draw[->,fill=red,fill opacity=0.1] (Px) -- (Pxy) -- (P) -- (Pxz) -- cycle;

    \draw[->,fill=magenta,fill opacity=0.1] (Py) -- (Pxy) -- (P) -- (Pyz) -- cycle;
  \end{tikzpicture}
\end{center}

The 
{relative homology sequence} reads: 
\[
\begin{tikzcd}[column sep=small]
0 \arrow[r] & HM_2[X] \arrow[r] & HM_{2}[X,Y] \arrow[r] & { }^X HM_1[Y] \arrow[r] &  HM_1[X]\arrow[r] &  HM_{1}[X,Y] \arrow[r] &  0
\end{tikzcd}
\]
\noindent Since there are no 2-cube chains in $X$ (no 3-cell, no concatenation of at least 2 2-cells) in $X$, and there are no 1-cube chains in $Y$ (only 0-cube chains). 

Now we compute: 
\begin{itemize}
\item 
{$HM_1[X,Y](000,111)=0$: }
${ }^X R_1[Y]$ is 
the sub-$R$-bimodule of $R_1[X]$ of dipaths going through $Y$. 
All dipaths from $000$ to $111$ go through $Y$ so $R_1[X]/{ }^X R_1[Y](000,111)=0$ and $HM_1[X,Y]=0$. 
\item 
{$HM_1[X](000,111)=R$: }
as already seen (all dipaths from start to end, are homologous)
\item 
{${ }^X HM_1[Y](000,111)=R^{6}$:}
$HM_1[Y]$ is $R_1[Y]$ so ${ }^X HM_1[Y](000,111)={ }^X R_1[Y]=R[X](000,111)=R^{6}$
\item 
{$HM_2[X,Y](000,111)=R^{6}$:}
$R_2[Y]=0$ so ${ }^X R_2[Y]=0$ and $R_2[X]/{ }^X R_2[Y]=R_2[X]$. As $R_1[X]/{ }^X R_1[Y](000,111)=0$, $Ker \ \partial=R_2[X](000,111)$; now $R_3[X](000,111)=0$, $R_3[Y](000,111)=0$, so $HM_2[X,Y](000,111)=R_2[X](000,111)=R^6$. 
\end{itemize}

Therefore, the 
{relative homology sequence between start and end} is: 
\[
\begin{tikzcd}[column sep=small]
0 \arrow[r] & HM_2[X](000,111) \arrow[r] & R^{6} \arrow[r] & R^{6} \arrow[r] & R \arrow[r] & 0 \arrow[r] & 0
\end{tikzcd}
\]
\noindent 
So we find $HM_2[X](000,111)=R$ as already seen (a sequence of $R$-vector spaces is split so we just count dimensions).

More generally, 
by induction on $n$, we can prove that the empty $n$-cube $X_n$ (``$\overrightarrow{S}^{n-1}"$) has:
\begin{itemize}
\item $HM_{n-1}[X_n](0^n,1^n)=R$ and 
\item $HM_i[X_n](0^n,1^n)=R$, $1 < i< n-1$ 
(using the relative pair $(X_n,Y_n)$, with $Y_n$ the sub ``initial" $(n-1)$-cube, with first coordinate $0$)
\end{itemize}
Similarly to the classical relative homology sequence, we can in fact characterize all of $HM_*[X_n]$ using a similar calculation. 
\end{example}

\paragraph{Mayer-Vietoris sequence}

Consider a precubical set $X$ and sub-precubical sets $X^1$ and $X^2$ such that $X=X^1\cup X^2$, and such that $(X,X^1)$, $(X,X^2)$ are relative pairs. 
Furthermore, we ask that $(X^1,X^2)$ is a {good cover} of $X$ in the following sense: 

\begin{definition}
Let $X^1$ and $X^2$ be two sub precubical sets of $X$ that cover $X$, i.e. $X=X^1\cup X^2$. This pair of precubical sets $(X^1,X^2)$ is called a \look{good cover} if all $i$-cube chains of $X$, $c=(c_1,\ldots,c_k)$ are such that there exists indexes $l$ and $m$, $l < m$ such that
$c_1,\ldots, c_l \in X_1$, $c_m,\ldots,c_k \in X_1$ and for all $j$ with $l < j < m$, $c_j \in X^1$ or for all $j$ with $l < j < m$, $c_j \in X^2$. 
\end{definition}

In short, $(X^1,X^2)$ is a good cover of $X$ if $i$-cube chains of $X$ are whiskerings (by 0-cube chains) of either a $(i-1)$-cube chain of $X^1$ or a $(i-1)$-cube chain of $X^2$. Note that this is somehow a form of ``relative pair" condition, but for all $i$-cube chains, not only the 0-cube chains. What the condition of a good cover implies is that precisely, every $i$-cube chain in $X$ enters $X_1$ and $X_2$ once, and exits it once, as an $i$-cube chain in $X_1$, and in $X_2$. 

We first prove:

\begin{lemma}
\label{lem:shortMayer}
For all $i\geq 1$, we have a short exact sequence of $R[X]$-bimodules, when $(X^1,X^2)$ is a good cover of $X$: 

\begin{equation}
    \label{eq:shortMayer}
\begin{tikzcd}[column sep=small]
0 \arrow[r] & { }^X R_i[X^1\cap X^2] \arrow[r,"\alpha"] & { }^X R_i[X^1]\oplus { }^X R_i[X^2] \arrow[r,"\beta"] & R_i[X] \arrow[r] & 0
\end{tikzcd}
\end{equation}
\noindent with $\alpha(x)=(x,-x)$ and $\beta(x,y)=x+y$. 
\end{lemma}

For proving this, we first observe: 

\begin{lemma}
\label{lem:relpairintersect}
Let $(X,X^1)$ and $(X,X^2)$ be two relative pairs. Then $(X^1,X^1\cap X^2)$ and $(X^2,X^1\cap X^2)$ are relative pairs.     
\end{lemma}

\begin{proof}
We prove that $(X^1,X^1 \cap X^2)$ is a relative pair, the other case being symmetric. 

For this, consider $c=(c_1,\ldots,c_m) \in Ch^{=0}(X^1)$ such that $d^0(c_1) \in X^1_0\cap X^2_0$ (resp. $d^1(c_m)\in X^1_0\cap X^2_0$), then, in particular, $c$ belongs as well to $Ch^{=0}(X)$ such that $d^0(c_1) \in X^u_0$ (resp. $d^1(c_m)\in X^u_0$) for $u=1$ and $u=2$. As $(X,X^u)$ for $u=1$ and $u=2$ is a relative pair: 

\begin{itemize}
    \item there exists an index $k^u \in \{1,\ldots,m\}$ such that, 
    \item $c_l \in X^u$ for $l<k^u$ (resp. $c_l \in X^u$ for $l>k^u$) and 
    \item $c_l \not \in X^u$ for $l>k^u$ (resp. $c_l \not \in X^u$ for $l<k^u$).
    \end{itemize}
Let us write $j=min(k^1,k^2)$ (resp. $j=max(k^1,k^2)$, then:
\begin{itemize}
    \item $c_l \in X^1\cap X^2$ for $l<k^u$ (resp. $c_l \in X^1\cap X^2$ for $l>k^u$) and 
    \item $c_l \not \in X^1\cap X^2$ for $l>k^u$ (resp. $c_l \not \in X^1 \cap X^2$ for $l<k^u$).
    \end{itemize}
    This proves that $(X^1,X^1\cap X^2)$ is a relative pair. 
\end{proof}

\begin{lemma}
\label{lem:reltrans}
Let $(X,Y)$ and $(Y,Z)$ be relative pairs and $M$ be a free $R[Z]$-bimodule. Then ${ }^X ({ }^Y M)$ is isomorphic to ${ }^X M$ as a $R[X]$-bimodule. 
\end{lemma}

\begin{proof}
As $M$ is a free $R[Z]$-bimodule, it is generated by a set of generators that we call $G$, as a $R[Z]$-bimodule. Then, we know by Lemma \ref{rem:relativerestriction} that ${ }^Y M$ is generated by $G$ as well as a $R[Y]$-bimodule, and ${ }^X ({ }^Y M)$ is, by the same Lemma, generated by $G$ as a $R[X]$-bimodule. Finally, ${ }^X M$ is generated by $G$ as a $R[X]$-bimodule, once again by Lemma \ref{rem:relativerestriction}, so is isomorphic to ${ }^X({ }^Y M)$.   
\end{proof}


We can now prove Lemma \ref{lem:shortMayer}. 

We note that as $(X^1,X^1\cap X^2)$ is a relative pair by Lemma \ref{lem:relpairintersect}, ${ }^{X^1} R_i[X^1\cap X^2]$ is a sub-$R[X^1]$ bimodule of $R_i[X^1]$ and as $(X^2,X^1\cap X^2)$ is a relative pair by the same Lemma, ${ }^{X^2} R[X^1\cap X^2]$ is a sub-$R[X^2]$ bimodule of $R_i[X^2]$. 

By Lemma \ref{lem:reltrans}, this implies that ${ }^X({ }^{X^1} R_i[X^1\cap X^2])\sim { }^{X} R_i[X^1\cap X^2]$ is a sub-$R[X]$ bimodule of ${ }^X R_i[X^1]$ with corresponding inclusion map $j_1$. Similarly, ${ }^X({ }^{X^2} R_i[X^1\cap X^2])\sim { }^{X} R_i[X^1\cap X^2]$ is a sub-$R[X]$ bimodule of ${ }^X R_i[X^2]$, with corresponding inclusion map $j_2$.

Now, consider the canonical monomorphisms of $R[X]$-bimodules 
$$i_1: \ { }^X R_i[X^1] \rightarrow { }^X R_i[X^1]\oplus { }^X R_i[X^2]$$ 
$$i_2: \ { }^X R_i[X^2] \rightarrow { }^X R_i[X^1]\oplus { }^X R_i[X^2]$$ 

Consider now map $\alpha: \ { }^X R_i[X^1\cap X^2] \rightarrow { }^X R_i[X^1]\oplus { }^X R_i[X^2]$ with $\alpha(x)=(x,-x)$. This map $\alpha$ is obviously a monomorphism ($Ker \ \alpha=0$). 

As $(X,X^1)$ is a relative pair, ${ }^X R_i[X^1]$ is a sub-$R[X]$ bimodule of $R_i[X]$, similarly for ${ }^X R_i[X^2]$. We can then define 
the map $\beta: \ { }^X R_i[X^1] \oplus { }^X R_i[X^2] \rightarrow R_i[X]$ with $\beta(x,y)=x+y$. This map is obviously surjective when $(X^1,X^2)$ is a good cover. Indeed, in that case, all $(i-1)$-cube chains $c=(c_1,\ldots,c_k)$ of $X$ are such that there exists indexes $l$ and $m$ such that
$c_1,\ldots, c_l \in X_1$, $c_m,\ldots,c_k \in X_1$ and for all $j$ with $l < j < m$, $c_j \in X^1$ or for all $j$ with $l < j < m$, $c_j \in X^2$. Therefore, $c'=(c_{l+1},\ldots,c_{m-1})$ is a $(i-1)$-cube chain of either $X^1$ or $X^2$ and $c$ can be identified with the whiskering $(c_1,\ldots,c_l)\otimes c' \otimes (c_m,\ldots,c_k)$ in ${ }^X R_i[X^1]$ or ${ }^X R_i[X^2]$. 

Finally, $\beta \circ \alpha=0$, so $Im \ \alpha \subseteq Ker \ \beta$. Conversely, an element of $Ker \ \beta$ is a pair $(x,y)$ such that $x+y=0$, hence $x=-y$ and $(x,y)\in Im \ \alpha$. 

All this proves that the short sequence, Equation (\ref{eq:shortMayer}), is exact. 

Furthermore, the boundary operator obviously commutes with maps $\alpha$ and $\beta$, hence this short sequence is actually a short sequence of chain complexes of $R[X]$-bimodules. As the category of bimodules over the (fixed) algebra $R[X]$ is abelian, we get the following Mayer-Vietoris exact sequence: 

\begin{theorem}
\label{thm:MayerVietoris}
Let $X$ a precubical set, $X_1$ and $X_2$ be two sub precubical sets that form a good cover of $X$, and such that $(X,X_1)$ and $(X,X_2)$ are relative pairs. 
Then we have the following long exact sequence: 
\begin{center}
    \begin{tikzcd}[arrow style=math font,cells={nodes={text height=2ex,text depth=0.75ex}}]
    & & 0 \arrow[draw=none]{d}[name=X,shape=coordinate]{} \\
       HM_{1}[X] \arrow[curarrow=X]{urr}{} 
       &  { }^X HM_1[X_1]\oplus { }^X HM_1[X_2]    \arrow[l] \arrow[draw=none]{d}[name=Y, shape=coordinate]{} & \arrow[l] \cdots \\
       HM_{i}[X] \arrow[curarrow=Y]{urr}{} & { }^X HM_{i}[X_1]\oplus { }^X HM_i[X_2] \arrow[l] \arrow[draw=none]{d}[name=Z,shape=coordinate]{} & { }^X HM_i[X_1\cap X_2]) \arrow[l] \\
       HM_{i+1}[X] \arrow[curarrow=Z]{urr}{} & { }^X HM_{i+1}[X_1]\oplus { }^X HM_{i+1}[X_2] \arrow[l] & \cdots \arrow[l]
   \end{tikzcd}
\end{center}
\end{theorem}




    

\begin{example}
Consider again the ``two holes on the diagonal" and the ``two holes on the antidiagonal" examples, see Example \ref{ex:2holes}. We indicate in red the sub precubical set $X_1$ of $X$. $X_2$ is the precubical set in black, together with edges common to $X_1$: on the left picture, $X_2$ contains $j$, $h$, $f$, $l$, $e$, $b$, $g$, $D$ and the vertices at their boundary; on the central picture, $X_2$ contains $j$, $h$, $l$, $f$, $d$, $e$, $c$, $a$ and the vertices at their boundary; and on the right picture below, $X_2$ contains $j$, $h$, $l$, $f$, $E$, $g$, $e$, $b$ and the vertices at their boundary. On the left, the cover $(X_1,X_2)$ of the ``2 holes on the antidiagonal" is not a good cover: the 3 cube-chain $(C,D)$ is not a whiskering of a 3 cube-chain in $X_1$ nor in $X_2$, they do not even contain 3 cube-chains. In the center, the cover $(X_1,X_2)$ of the ``2 holes on the antidiagonal" is a good cover, as in the right hand side, the decomposition $(X_1,X_2)$ is a good cover for the ``2 holes on the diagonal example". But $(X,X_1)$ is not a relative pair for the central example. In  the examples on the left and on the right, $(X,X_1)$ and $(X,X_2)$ are relative pairs. 

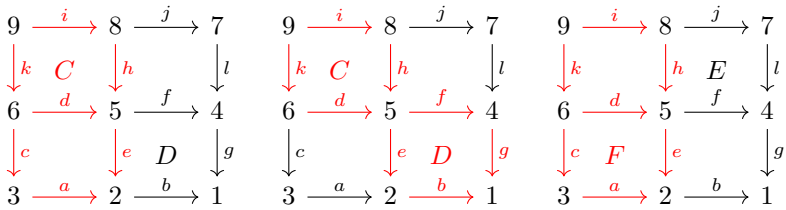
\begin{center}
    \begin{minipage}{3.5cm}
    \[\begin{tikzcd}
  9 \arrow[red, r,"i"] \arrow[red,d,"k"] \arrow[red,dr,phantom,"C"] & 8 \arrow[red,d,"h"] \arrow[r,"j"] & 7 \arrow[d,"l"]\\
  6 \arrow[red,r,"d"] \arrow[red,d,"c"] & 5 \arrow[red,d,"e"] \arrow[dr,phantom,"D"] \arrow[r,"f"] & 4 \arrow[d,"g"] \\
  3 \arrow[red,r,"a"] & 2 \arrow[r,"b"] & 1
\end{tikzcd}
\]
\end{minipage}
    \begin{minipage}{3.5cm}
    \[\begin{tikzcd}
  9 \arrow[red, r,"i"] \arrow[red,d,"k"] \arrow[red,dr,phantom,"C"] & 8 \arrow[red,d,"h"] \arrow[r,"j"] & 7 \arrow[d,"l"]\\
  6 \arrow[red,r,"d"] \arrow[d,"c"] & 5 \arrow[red,d,"e"] \arrow[red,dr,phantom,"D"] \arrow[red,r,"f"] & 4 \arrow[red,d,"g"] \\
  3 \arrow[r,"a"] & 2 \arrow[red,r,"b"] & 1
\end{tikzcd}
\]
\end{minipage}
    \begin{minipage}{3.5cm}
    \[\begin{tikzcd}
  9 \arrow[red,r,"i"] \arrow[red,d,"k"]  & 8 \arrow[red,d,"h"] \arrow[r,"j"] \arrow[dr,phantom,"E"] & 7 \arrow[d,"l"]\\
  6 \arrow[red,r,"d"] \arrow[red,d,"c"] \arrow[red,dr,phantom,"F"] & 5 \arrow[red,d,"e"]  \arrow[r,"f"] & 4 \arrow[d,"g"] \\
  3 \arrow[red,r,"a"] & 2 \arrow[r,"b"] & 1
\end{tikzcd}
\]
\end{minipage}
\end{center}

We consider the Mayer-Vietoris exact sequences for the ``two holes on the diagonal'' example, with its good cover, pictured on the right above.  
Let us consider this exact sequence from 9 to 1, 
it is concentrated in dimension 1: 
\begin{itemize}
\item ${ }^X HM_1[X_1](9,1)=R^5$
\item ${ }^X HM_1[X_2](9,1)=R^5$
\item ${ }^X HM_1[X_1\cap X_2]=R^6$
\end{itemize}
And the corresponding Mayer-Vietoris sequence reads:
\[
\begin{tikzcd}[column sep=small]
0  \arrow[r] & R^6 \arrow[r] &  R^5 \oplus R^5 \arrow[r] &  HM_1[X](9,1) \arrow[r] &  0
\end{tikzcd}
\]
\noindent which implies that $HM_1[X](9,1)=R^4$. 
\end{example}









\section{Conclusion and future work}

In this paper, we studied a simple directed homology theory for certain precubical sets, based on the category of bimodules over the underlying path algebra. Doing so, we made links to other homology theories, first with persistent homology and then to natural homology, used as a directed topology invariant. 

Doing so, we unveiled the importance of the restriction of scalar functor between abelian categories of bimodules based on related algebras, and its homological properties. This showed when discussing the K\"unneth theorem, and other exact sequences. We will, in a forthcoming paper, axiomatize homology theories that take values in such ``varying abelian categories'', making use of slightly ``deformed'' exact sequences. 

We believe that this framework can be extended so as to deal with rewriting applications on e.g. monoids, see Appendix \ref{sec:digression}. 
Given that we have a reasonable notion of relative pair and relative homology exact sequence, another natural application would be to define a directed Conley index \cite{conleyindex}, using relative homology of a pair defined by the exit set of some suitable ``directed isolating neighborhood". A first step towards this endeavour would be to study a directed version of Wazewski's property and look for Morse like decomposition of directed spaces. This will require extending this work to deal with general directed spaces, we give some hints in Appendix \ref{sec:homdir}. 

One open question still concerns the right notion of $HM_0$ of a precubical set. This should definitely be linked to (pair, see \cite{paircomponents}) component categories in general, but this still has to be properly studied. This is linked to ``tameness'' issues in persistence, that we describe briefly in Appendix \ref{sec:tameness}, and that we believe are central to both persistent homology and directed homology theories. 

We also plan on using the nice computational methods (such as fringe representations, see e.g. \cite{miller2020data}) used in the context of persistence modules to make actual practical computations of directed invariants for precubical sets.











 
 


\appendix

\section{Associative algebras and modules over algebras}

\label{sec:A}

\paragraph{Algebras}

\label{sec:backgroundalgebra}
Let $R$ be a commutative field.  

\begin{definition}
\label{def:alg}
A \look{(unital) associative algebra} on $R$, or \look{$R$-algebra} $A=(A,+,.,\times)$ is an $R$-vector space $(A,+,.)$, with external multiplication by elements of the ring $R$ denoted by ., that has an internal monoid operation, which is an associative operation (``multiplication" or ``internal multiplication") $\times: A \times A \rightarrow A$ that is bilinear. We denote by $0$ the neutral elements for $+$ (which we use also for denoting the 0 of the ring $R$) and $1$ the neutral element (or identity) for $\times$. 
\end{definition}


\begin{definition}
\label{def:mor}
Let $A$ and $B$ be two $R$-algebras. A \look{morphism of $R$-algebras} $f: A \rightarrow B$  is a linear map from $A$ to $B$ seen as $R$-vector spaces, such that it commutes with the internal multiplication: 
$$f(a\times a')=f(a)\times f(a')$$
\end{definition}

We write $Alg$ for the \look{category of (unital) associative algebras} and morphisms as in Definition \ref{def:alg} and \ref{def:mor} above. 

An $R$-submodule $B$ of an $R$-algebra $A$ is an \look{$R$-subalgebra} of $A$ if the identity of $A$ belongs to $B$ and $b_1\times b_2 \in B$ for all $b_1, b_2 \in B$. 
An $R$-submodule $I$ of an $R$-algebra $A$ is a {right ideal} of $A$ (resp. left ideal of $A$) if $x\times a \in I$ (or $a\times x \in I$, respectively) for all $x \in I$ and $a \in A$. A two-sided ideal of $A$ (or simply an \look{ideal} of $A$) is both a left ideal and a right ideal of $A$.

It is easy to see that if $I$ is a two-sided ideal of an $R$-algebra $A$, then the quotient $R$-vector space $A/I$ has a unique $R$-algebra structure such that the canonical surjective linear map $\pi : A \rightarrow A/I$, $a \rightarrow a = a + I$, becomes an $R$-algebra homomorphism.

If $I$ is a two-sided ideal of $A$ and $m \geq 1$ is an integer, we denote by $I^m$ the two-sided ideal of $A$ generated by all elements $x_1\times x_2 \times \ldots \times x_m$, where $x_1, x_2, \ldots, x_m \in I$, that is, $I^m$ consists of all finite sums of elements of the form $x_1\times x_2\times \ldots \times x_m$, where $x_1,x_2,\ldots,x_m \in I$. We set $I^0 = A$.

The \look{(Jacobson)} radical $rad \ A$ of an $R$-algebra $A$ is the intersection of all the maximal right ideals in $A$.

Let $A$ and $B$ be $R$-algebras. Since $A$ and $B$ may both be regarded as $R$-vector spaces, their tensor product
$A\otimes_{R} B$
is also an $R$-vector space. This is the $R$-vector space where all elements can be written (non uniquely in general) as finite sums of elements of the form 
$a \otimes b$ with 
$a \in A$ and $b \in B$, with the property that $\otimes$ is bilinear. 

The \look{tensor product} can be given the structure of a ring by defining the product on elements of the form $a\otimes b$ by: 

$$(a_1\otimes b_1)\times (a_2\otimes b_2) = (a_1 \times a_2)\otimes (b_1\times b_2)$$
\noindent and then extending by linearity to all of $A \otimes_R B$. 
This ring is an $R$-algebra, associative and unital with identity element given by $1_A \times 1_B$ if $A$ and $B$ are unital. If $A$ and $B$ are commutative $R$-algebras, then the tensor product is commutative as well.

The tensor product turns the category of $R$-algebras into a symmetric monoi\-dal category (and even monoidal closed, with the internal Hom functor in $Alg$).



\paragraph{Modules over an associative algebra}



\begin{definition}[\cite{assocalg}]
Let $A$ be an (unital) $R$-algebra. A \look{right $A$-module} (or a right module over $A$) is a pair $(M, \bullet)$, where $M$ is an $R$-vector space and $\bullet : M \times A \rightarrow M$, $(m, a) \rightarrow m\bullet a$, is a binary operation satisfying the following conditions, for all $x, y \in M$, $a, b \in A$ and $\lambda \in R$:
\begin{itemize}
    \item $(x+y)\bullet a=x\bullet a+y\bullet a$ 
    \item $x\bullet (a+b)=x\bullet a+x\bullet b$ 
    \item $x\bullet (ab) = (x\bullet a)\bullet b$
\item $x\bullet 1=x$
\item $(x\lambda)\bullet a = x\bullet (a\lambda) = (x\bullet a)\lambda$
\end{itemize}
\end{definition}

The definition of a left $A$-module and $A$-bimodule is analogous. Throughout, we write $M$ or $M_A$ instead of $(M, \bullet)$. When in need for disambiguating formulas, we will write explicitely ${ }_A \bullet_M$ for the left action of $A$ on $M$ and ${ }_M \bullet_A$ for the right action of $A$ on $M$, or simply ${ }_A \bullet$ (resp. $\bullet_A$) when the context makes it clear on which module the action is taken, or simply again, $\bullet_M$ (resp. ${ }_M \bullet$) when the context makes it clear which algebra acts on $M$. We write $A_A$ and ${}_AA$ whenever we view the algebra $A$ as a right or left $A$-module, respectively, with $\bullet$ being the algebra multiplication. We write $A$ for the $A$-bimodule $A$. 
Note that an $A$-bimodule can be seen as a left (or right) $A\otimes A^{op}$-bimodule, where $A^{op}$ is the algebra which has the same elements as $A$ but with the multiplication $a \times_{A^{op}} b=b\times_{A} a$. 


A \look{morphism of (right) $A$-modules} $f: \ M \rightarrow N$ is a linear map between the underlying $R$-vector spaces of $M$ and $N$ which preserve the right action of the $R$-algebra $A$. 

We write ${ }_R Mod \ A$ for the category of left $A$-modules, $Mod_R \ A$ for the category of right $A$-modules, and ${ }_R Mod_R \ A$ for the category of $A$-bimodules. 

A module $M$ is said to be \look{finite dimensional} if the dimension $dim_R M$ of the underlying $R$-vector space of 
$M$ is finite. The category of left $A$-modules (resp. right and bi-modules) of finite dimension is denoted by ${ }_R mod \ A$ (resp. $mod_R \ A$, ${ }_R mod_R \ A$).

An $R$-subspace $M'$ of a right $A$-module $M$ is said to be an \look{$A$-submodule} of $M$ if $m\bullet a\in M$, forall $m\in M$ and  all $a\in A$. In this case the $R$-vector space $M/M'$ has a natural $A$-module structure such that the canonical epimorphism $\pi: \ M \rightarrow M/M'$ is an $A$-module homomorphism. Similarly for left-modules and bimodules. 


Let $M$ be a right $A$-module and let $I$ be a right ideal of $A$. It is easy to see that the set $M\bullet I$ consisting of all sums $m_1\bullet a_1 +\ldots +m_s\bullet a_s$, where $s \geq 1$, $m_1,\ldots,m_s \in M$ and $a_1,\ldots,a_s \in I$, is a sub right $A$-module of $M$. Similarly for left-modules (with left ideals) and bimodules (with two-sided ideals). 

A right $A$-module $M$ is said to be \look{generated} by the elements $m_1,\ldots, m_s$ of $M$ if any element $m\in M$ has the form $m=m_1\bullet a_1+\ldots+m_s\bullet a_s$ for some $a_1,\ldots,a_s \in A$. In this case, we write $M = m_1 A + \ldots+ m_sA$. A module $M$ is said to be \look{finitely generated} if it is generated by a finite subset of elements of $M$.

Finally, let $M$ be an $A$-bimodule and $N$ be a $B$-bimodule. Then the $R$-vector space which is the \look{tensor product} of the underlying $R$-vector spaces of $M$ and $N$ (elements of which are generated by tensors $m\otimes n$, $m\in M$ and $n\in N$, $\otimes$ being $R$-bilinear) can be given the structure of a $A\otimes B$-bimodule by setting: 
$$
(a\otimes b)\bullet m \otimes n \bullet (a' \otimes b') = (a\bullet m \bullet a')\otimes (b \bullet n \bullet b')
$$
\noindent for $a$, $a'$ in $A$ and $b$, $b'$ in $B$.

\section{Quiver and path algebras}

\label{sec:B}

\begin{definition}[\cite{assocalg}]
We say that an algebra $A$ is \look{connected} (or indecomposable) if $A$ is not isomorphic to a product of two algebras. 
\end{definition}

As is well known \cite{assocalg}, for $Q$ a quiver, $R[Q]$ is a connected algebra if and only if $Q$ is connected, as a (undirected) graph. 


\begin{definition} (\cite{assocalg})
\label{def:basic}
A finite dimensional $R$-algebra $A$ is \look{basic} if and only if the algebra $B = A/rad \ A$ is isomorphic to a product $R\times R \times \cdots \times R$ of copies of $R$.
\end{definition}

A path algebra is basic when it is acyclic (in the standard undirected sense). 

Associative algebras are closely linked to quivers, path algebras are not ``just'' examples of associative algebras, they are central to the theory of associative algebras. 

\begin{definition}[\cite{assocalg}]
\label{def:ordquiver}
Let $A$ be a basic and connected finite dimensional $R$-algebra and ${e_1,e_2,\ldots,e_n}$ be a complete set of primitive orthogonal idempotents of $A$. The (ordinary) \look{quiver of $A$}, denoted by $Q_A$, is defined as follows:
\begin{itemize}
\item The points of $Q_A$ are the numbers $1, 2,\ldots , n$, which are in bijective correspondence with the idempotents $e_1, e_2,\ldots , e_n$.
\item Given two points $a,b \in (Q_A)_0$, the arrows $\alpha : \ a \rightarrow b$ are in bijective correspondence with the vectors in a basis of the $R$-vector space $e_a(rad \ A/rad^2 \ A)e_b$.
\end{itemize}
\end{definition}

The fact that the definition above is legal is proved in \cite{assocalg} ($Q_A$ does not depend on the basis which is chosen, in particular). 

\begin{definition}[\cite{assocalg}] 
\label{def:boundquiver}
Let $Q$ be a finite quiver and $RQ$ be the \look{arrow ideal} of the path algebra $R[Q]$, i.e. the ideal generated by edges of $Q$ (elements of $Q_1$). A two-sided ideal $I$ of $R[Q]$ is said to be \look{admissible} if there exists $m \geq 2$ such that
$RQ^m \subseteq I \subseteq RQ^2$.
If $I$ is an admissible ideal of $R[Q]$, the pair $(Q,I)$ is said to be a \look{bound quiver}. The quotient algebra $R[Q]/I$ is said to be the algebra of the bound quiver $(Q,I)$ or, simply, a \look{bound quiver algebra}.
\end{definition}

\begin{theorem}[\cite{assocalg}]
\label{thm:algfromquiver}
Let $A$ be a basic and connected finite dimensional $R$- algebra. There exists an admissible ideal $I$ of $R[Q_A]$ such that $A \equiv R[Q_A]/I$. 
\end{theorem}


\section{A digression on resolutions of associative algebras and monoids}

\label{sec:digression}

We note that the chain complex we have been defining in Lemma \ref{lem:precubboundary} has a natural augmentation, which will actually link to classical tools used in the construction of the homology of associative algebras: 

\begin{definition}
\label{def:augmentation}
Let \look{$R_0[X]$} be the $R[X]$-bimodule whose underlying $R$-vector space is generated by pairs of reachable points of $X$, i.e. by $(x,y)\in X\times X$ such that there exists a 0-cube path from $x$ to $y$, and whose action of $R[X]$ is generated by the relations:  
$$p_{u,v} \bullet (x,y) \bullet q_{u',v'}=\left\{\begin{array}{ll}
(u,v') & \mbox{if $v=x$ and $y=v'$} \\
0 & \mbox{otherwise}
\end{array}\right.
$$
\label{def:R0}
\end{definition}




\begin{lemma}
The morphism of $R$-vector spaces $\epsilon: R_1[X] \rightarrow R_0[X]$ defined on generators $c=(c_1,\ldots,c_n)$, $c_i \in X_1$ and $e_a$, $a \in X_0$ as:
$$
\begin{array}{rcl}
\epsilon(c_1,\ldots,c_n) & = & (d^0(c_0),d^1(c_n)) \\
\epsilon(e_a) & = & (a,a)
\end{array}
$$
is an epimorphism of $R[X]$-bimodule, which is such that $\epsilon\circ \partial: R_2[X] \rightarrow R_0[X]$ is equal to zero. Therefore, $\epsilon$ defines an augmentation of the chain complex $(R_i[X],\partial)$.
\end{lemma}

\begin{proof}
Consider $p=(p_1,\ldots,p_k)$, $c=(c_1,\ldots,c_l)$ and $q=(q_1,\ldots,q_m)$ in $Ch^{=0}(X)$, and suppose (the other case is obvious) that $d^1(c_l)=d^0(q_1)$ and $d^1(p_k)=d^0(c_1)$, we have:
$$
\begin{array}{rcl}
\epsilon(p\bullet c \bullet q) & = & \epsilon(p_1,\ldots,p_k,c_1,\ldots,c_l,q_1,\ldots,q_m) \\
& = & (d^0(p_1),d^1(q_m)) \\
& = & (p_1,\ldots,p_k) \bullet (d^0(c_1),d^1(c_l)) \bullet (q_1,\ldots,q_m)\\
 & = & p\bullet \epsilon(c) \bullet q
\end{array}
$$
The fact that this defines an epimorphism of $R[X]$-bimodule is obvious, by definition of $R_0[X]$. 

Finally, we compute, for a 1-cube chain $c=(c_1,\ldots,c_l)$ where $c_k$ is the unique 2-cell (by Remark \ref{rem:1cubechain}) and all other $c_i$s are 1-cells: 
$$
\begin{array}{rcl}
\epsilon(\partial(c)) & = & \epsilon(d_{k,\{0\}}(c)-d_{k,\{1\}}(c))\\
& = & \epsilon(d_{k,\{0\}}(c))-\epsilon(d_{k,\{1\}}(c))\\
& = & (d^0(c_1),d^1(c_l))-(d^0(c_1),d^1(c_l)) \\
 & = & 0
\end{array}
$$
\end{proof}

The usual construction of the homology of associative algebras \cite{maclane2012homology}, for instance of monoids $\mathcal M$ in the case we are considering here, is done through a free resolution of $R$ considered as the trivial $R\mathcal{M}$-module (the action is the identity), by some free $R\mathcal{M}$-modules, where $R\mathcal{M}$ is the monoid algebra of $\mathcal{M}$. In general, only left (or right) modules would be considered here, but this could as well use  bimodules as we do in this paper.

The monoid algebra is indeed the categorical algebra of the monoid $\mathcal{M}$ considered as the one object category (called $*$), whose morphisms are (left) multiplication by elements of $\mathcal{M}$. And $R$ as a $R\mathcal{M}$-bimodule is just the analogue of our construction $R_0[\mathcal{M}]$ which is generated as an $R$-vector space only by the pair $(*,*)$ and the action on the left and on the right by $R[X]$ is the identity. 

Suppose $\mathcal{M}$ is generated by $\Sigma$, a given finite alphabet. 
Groves and Anick \cite{Groves,Anick} constructed a cubical complex $\mathcal{C(M)}$ out of the directed graph of reductions of a rewrite system, presenting $\mathcal{M}$, with vertices being words in $\Sigma^*$. Indeed, when $\mathcal{M}$ can be presented by a finite canonical rewrite system \cite{lechenadec}, the cubical complex $\mathcal{CM}$ constructs a free resolution by (left) $R\mathcal{M}$-modules of $R$, which is finite dimensional. 
This is Squier's property $FP_\infty$ \cite{squier}. 

Consider the algebra epimorphism: $q: \ R\mathcal{M}\rightarrow R$ which maps $\sum r_i m_i$ (where $r_i$ are coefficients in $R$ and $m_i$ are elements of the monoid $\mathcal{M}$) to $\sum r_i \in R$ and the extension of scalars $q_!: \ { }_{R\mathcal{M}} Mod \rightarrow { }_R Mod$. This extension of scalars is the classical tensor operation with $R$ seen as a trival (left) $R\mathcal{M}$-module, see Section \ref{sec:assoc}, which is not left exact in general (as a left adjoint, it is right exact), and creates homology. 
By the $FP_\infty$ property, this implies that the monoid homology (with $R$-vector space coefficients) is finite dimensional. 
 Hence a monoid $\mathcal{M}$ cannot be presented by a finite canonical rewrite system if the homology of $\mathcal{M}$ is not finite dimensional (as an $R$-vector space). 

We believe that Groves' cubical complex based free resolution of $R$ can be recast using our construction of a precubical set based free resolution of $R$, with critical $n$-pairs generating $n$-cells, this will be developed elsewhere. Note that it needs a substantial generalization to the framework developed in this paper, to deal with modules over non-unital algebras in particular, along the lines of what is sketched in Section \ref{sec:homdir}.

The interest of this link is twofold. First, we can look at resolutions of $R_0[\mathcal{M}]$ and not just $R$, which would give a finer view of the geometry of rewriting ``between two words in $\mathcal{M}$''. Second, instead of using the extension of scalars to $R$, we could use any finer invariants, using extension of scalars to e.g. $R[Y]$, with $Y$ a path in the derivation graph of $\mathcal{M}$ presented by a given finite canonical system, or any filtration such as the one based on the length of derivations, creating instead a potentially insightful one-parameter persistent homology. 
\section{Homology modules for directed spaces and infinite precubical sets}

\label{sec:homdir}

In this section, we see how we would extend our homology bimodules from finitely presented directed spaces (finite precubical sets) to infinite discrete or continuous directed spaces, and try to make some connections between the two. In order to do so, we are restricting the class of directed spaces we are considering to so called po-spaces: 

\begin{definition}[\cite{thebook}]
  \begin{enumerate}
  \item A \look{po-space} is a topological space $X$ with a (global) closed partial order $\leq$ (i.e. $\leq$ is a closed subset of $X \times X$). The unit segment $[0,1]$ with the usual topology and the usual order, is a po-space denoted $\J$ and called the directed interval.
  \item A \look{dimap} $f:X\to Y$ between po-spaces $X$ and $Y$ is a continuous map that respects the partial orders (is non-decreasing).
  \item A \look{dipath} $f:\J\to X$ is a dimap whose source is the interval $\J$.
  \end{enumerate}
\end{definition}

When $C$ is a finite precubical set whose underlying graph is a DAG, the geometric realization $\mid C\mid $ of $C$, is a po-space. 
But the path algebra $R[\mid C\mid ]$, introduced in \cite{goubault2023semiabelian}, and counterpart of $R[C]$ would definitely not be a finite dimensional algebra and the notion of module over the path algebra is more difficult to define, as the path algebra is not unital in general.

Instead, we take advantage of the Morita equivalence between (bi)modules on the (unital) algebra $R[C]$ and the representation of the underlying quiver $FC_{\leq 1}$, to take as definition of a bimodule in the case of the path (non-unital) algebra of a po-space $X$, any functor as below with values in the category $Vect$ of $R$-vector spaces: 
$$
M: \ P_X \times P_X^{op} \rightarrow Vect 
$$
\noindent where $P_X$ is the category whose objects are points $x$ of $X$, and morphisms from $x$ to $y$ are any dipath from $x$ to $y$, modulo increasing reparameterisations (this is called a trace in e.g. \cite{goubault2023semiabelian}). $M$ is a bimodule on the path algebra in the sense that it is the categorical presentation of it as in Definition \ref{def:categoricalrepresentation}. To each pair of points $(x,y)$ in $X$, it associates $M(x,y)$, an $R$-vector space, and to each extension morphism $\langle u,v \rangle : (x,y)\rightarrow (x',y')$, where $u$ is a trace from $x'$ to $x$, and $v$ a trace from $y$ to $y'$, it associates an $R$-vector space homomorphism from $M(x,y)$ to $M(x',y')$. 




Indeed, this point of view allows also to extend the approach we took here to infinite precubical sets. 

\section{Tameness issues}

\label{sec:tameness}

In this section, we show that for particular precubical sets $X$, there is an finite encoding in the sense of \cite{miller2020modules} of the first homology bimodule $HM_1[X]$ making it amenable to a tame, or finite presentation. 

First, we need to recap the definitions of directed homotopy (dihomotopy) and of the fundamental category \cite{thebook} of a directed space.

\begin{definition}
Let $X$ be a directed space and $f$ and $g$ two dipaths from $x \in X$ to $y \in X$. A \look{directed homotopy} (or \look{dihomotopy}) $H$ between $f$ and $g$ is a directed map $H: \ \vec I \times \vec I \rightarrow X$ such that $H(0,.)=f$, $H(1,.)=g$, $H(.,0)=x$ and $H(.,1)=y$. 

Two dipaths are elementary dihomotopic if there exists a dihomotopy between them. \look{Dihomotopy equivalence} is the transitive closure of elementary dihomotopy. Two dipaths $f$ and $g$ are \look{dihomotopic} is short for saying that $f$ and $g$ are dihomotopy equivalent. 
\end{definition}

Now we are in a situation to recap the definition of the fundamental category of a directed space: 

\begin{definition}
Let $X$ be a directed space. The \look{fundamental category} $\dipi{X}$ of $X$ is the category:
\begin{itemize}
\item whose objects are points $x \in X$
\item whose morphisms $[f]: \ x \rightarrow y$ are dihomotopy classes of dipaths $f: \ \vec I \rightarrow X$ from $x$ to $y$
\end{itemize}
\end{definition}

In \cite{components}, a way to reduce the fundamental category $\dipi{X}$ of a po-space $X$ has been defined, through the concept of component category. The main property is that ${\cal C}=\dipi{X}$ enjoys  
Proposition 7 of \cite{apcs} that we recall below (since for a po-space, $\dipi{X}$ is a loop-free category), after introducing Yoneda systems of a category:

\begin{definition}
Given a (small) category {\cal C}, a morphism $\sigma: \ x \rightarrow y$ is a \look{Yoneda morphism} when for each object $z \in {\cal C}$ such that $\mathcal{C}[y,z] \neq \emptyset$, the map $pre(\sigma): 
\mathcal{C}[y,z] \rightarrow \mathcal{C}[x,z]$ which associates to each $\tau \in \mathcal{C}[y,z]$, $pre(\sigma)=\tau \circ \sigma$ is a bijection, and when for each object $z \in \mathcal{C}$ such that $\mathcal{C}[z,x]\neq \emptyset$, the map $post(\sigma): \mathcal{C}[z,x] \rightarrow \mathcal{C}[z,y]$ which associates to each $\tau \in \mathcal{C}[z,x]$, $post(\sigma)=\sigma \circ \tau$ is a bijection. 
\end{definition}

Now, a Yoneda system \cite{components} is a particular set of Yoneda morphisms that enjoys nice ``stability'' properties: 

\begin{definition}
Let $\mathcal{C}$ be a small category, $\Sigma$ a subset of the set of morphisms of $\mathcal{C}$ is a \look{Yoneda system} if and only if:
\begin{itemize}
\item $\Sigma$ is stable under composition
\item $\Sigma$ contains all isomorphisms of $\mathcal{C}$ and is composed of Yoneda morphisms
\item $\Sigma$ is stable under pushouts with any morphism in $\mathcal{C}$
\item $\Sigma$ is stable under pullbacks with any morphism in $\mathcal{C}$
\end{itemize}
\end{definition}

We can now formulate: 

\begin{proposition} 
\label{prop7}
\label{lifting}
Let $\cal C$ be a category in which all endomorphisms are identities (this is true in particular
if $\cal C$ is loop-free, see e.g. \cite{scowls}).
Let $\Sigma$ be any Yoneda system 
on $\cal C$ and $C_1, C_2\subset Ob({\cal C})$ be two objects of ${\cal C}/\Sigma$ such that
  the set of morphisms (in ${\cal C}/\Sigma$) is {finite}.
  Then, for every $x_1\in C_1$ there exists $x_2 \in C_2$ such that the
  quotient map
  $${\cal C}(x_1,x_2)\to {\cal C}/_{\Sigma}(C_1,C_2),\; f\mapsto
  [f]$$
  is {bijective}.
\end{proposition}

Consider a bimodule $M: \ P_X \times P^{op}_X \rightarrow R\mbox{-Vect}$ that factors through a bimodule $M': \ \dipi{X} \times \dipi{X}^{op} \rightarrow R\mbox{-Vect}$, by map $f$, meaning, in plain english, that the linear maps of $R$-vector spaces, images $M(u,v)$ of morphisms of $P_X \times P^{op}_X$ $\langle u,v \rangle$, where $u$ and $v$ are traces, only depend on the dihomotopy class of $u$ and $v$. Then the component categories provide a map: 
$$
\pi: \ \dipi{X}\times \dipi{X}^{op} \rightarrow \dipi{X}/_\Sigma \times (\dipi{X}/_{\Sigma})^{op}
$$
\noindent and a bimodule $N$ on $\dipi{X}/_\Sigma$ such that $$M=(\pi\circ f)_* N=\mathop{\oplus}\limits_{q \in P_{X}\times P_{X}^{op}} N(\pi\circ f(q))$$
\noindent with this coproduct being finite in the case $X=\mid C\mid $ is the geometric realization of a finite precubical set $C$, since in that case, the dihomotopy classes between two points are finitely many (by cubical approximation, \cite{Fajstrup2005DipathsAD}). 
This is the notion of encoding of 
\cite{miller2020modules} of a $\dipi{X}\times \dipi{X}^{op}$-module $M$ by morphism $\pi\circ f$, generalized to a categorical setting (we are considering modules on categories, not just on posets). 
$HM_1[X]$ provides such cases of bimodules $M$, hence the category of components provides an encoding of $HM_1[X]$. It has been shown that for certain classes of precubical sets $X$, in particular those generated by PV terms without loops, $\mid X\mid $ has the property that it has finitely many components \cite{concur2005}, and the bimodule $HM_1[X]$ has then a finite encoding, in the sense of \cite{miller2020modules}, which is a notion of tameness, amenable to actual computations. 

We conjecture all $HM_i[X]$ are tame in that sense, for all $i\geq 2$ and for at least all finite precubical sets which are the semantics of PV terms without loops. 
We also believe the fringe representations of \cite{miller2020homological} should be useful for computational characterizations of the homology bimodules in that case. 

\end{document}